\pdfoutput=1
\documentclass[12pt,reqno]{amsart}
\usepackage[english]{babel}
\usepackage[square,numbers]{natbib}

\numberwithin{equation}{section}
\usepackage[tt=false]{libertine}
\usepackage{mathtools}
\usepackage{comment}
\usepackage{amssymb,mathrsfs}
\usepackage{amsmath,amsfonts, amsthm}
\usepackage[varbb]{newpxmath}

\usepackage[margin=1in]{geometry}
\usepackage{graphicx}
\usepackage{enumerate}
\usepackage{bbm}
\usepackage{hyperref,color}
\usepackage{amsthm}
\usepackage[]{mdframed}
\usepackage{float}
\usepackage[dvipsnames]{xcolor}
\hypersetup{
	colorlinks=true,
	pdfpagemode=UseNone,
    citecolor=OliveGreen,
    linkcolor=NavyBlue,
    urlcolor=black,
	pdfstartview=FitW
}
\usepackage{appendix}
\usepackage{bm}
\usepackage{autonum}
\usepackage{overpic}
\usepackage{tikz}

\let\savedbigtimes\bigtimes
\let\bigtimes\relax

\let\bigtimes\savedbigtimes

\newcounter{Ccnt}
\makeatletter
\newcommand\Co[1]{%
\@ifundefined{C-#1}%
  {\stepcounter{Ccnt}\expandafter\xdef\csname C-#1\endcsname{\arabic{Ccnt}}}%
  {}%
C_{\csname C-#1\endcsname}}
\makeatother

\makeatletter
\newcommand{\pushright}[1]{\ifmeasuring@#1\else\omit\hfill$\displaystyle#1$\fi\ignorespaces}
\newcommand{\pushleft}[1]{\ifmeasuring@#1\else\omit$\displaystyle#1$\hfill\fi\ignorespaces}
\makeatother

\mathcode`l="8000
\begingroup
\makeatletter
\lccode`\~=`\l
\DeclareMathSymbol{\lsb@l}{\mathalpha}{letters}{`l}
\lowercase{\gdef~{\ifnum\the\mathgroup=\m@ne \ell \else \lsb@l \fi}}%
\endgroup

\renewcommand{\r}[2]{\begin{minipage}[c]{1.5in}
  \renewcommand{\baselinestretch}{0.9}
  \raggedright{\scriptsize #1\par}
  \end{minipage}
  &&#2&&
}

\newtheorem{theorem}{Theorem}[section]
\newtheorem{proposition}[theorem]{Proposition}
\newtheorem{lemma}[theorem]{Lemma}
\newtheorem{corollary}[theorem]{Corollary}

\theoremstyle{definition}
\newtheorem{definition}[theorem]{Definition}

\newtheorem{assumption}[theorem]{Assumption}
\newtheorem*{assumption*}{Assumption}
\newtheorem{remark}[theorem]{Remark}
\newtheorem*{notation}{Notation}

\newcommand{\E}{\mathbb{E}}

\newcommand{\floor}[1]{\left\lfloor #1 \right\rfloor}

\renewcommand{\epsilon}{\varepsilon}

\newcommand{\Q}{\mathbb{Q}}

\renewcommand{\P}{\mathbb{P}}

\newcommand{\aleq}{\ensuremath{\leq}\kern-1.05em\lower1.5ex\hbox{\ensuremath{\sim}}\;}
\newcommand{\ale}{\ensuremath{<}\kern-1.05em\lower1.5ex\hbox{\ensuremath{\sim}}\;}
\newcommand{\ageq}{\ensuremath{\geq}\kern-1.05em\lower1.5ex\hbox{\ensuremath{\sim}}\;}
\newcommand{\age}{\ensuremath{>}\kern-1.05em\lower1.5ex\hbox{\ensuremath{\sim}}\;}
\newcommand{\conv}[1]{\overset{#1}{\rightarrow}}
\newcommand{\comb}{\binom{k}{xk}\binom{p-k}{(1-x)k}}
\renewcommand{\l}{\ell}
\newcommand{\Bi}[1]{\ensuremath{\text{Binomial}\left(#1\right)}}
\newcommand{\gl}{y + (1-y)(2^{1-x} - 1)}
\newcommand{\der}{\;\partial}
\newcommand{\oeq}[1]{=\Theta\left ( #1 \right)}
\newcommand{\oleq}[1]{=O\left(#1\right)}
\newcommand{\ole}[1]{=o\left(#1\right)}
\renewcommand{\[}{\begin{equation}}
\renewcommand{\]}{\end{equation}}
\newcommand{\scr}{\mathcal}
\renewcommand{\|}{\bigg{|}}

\newcommand{\kset}{random MAX k-set cover }
\newcommand{\ngrow}{\ensuremath{n \conv{} +\infty} }
\newcommand{\bhigh}{\begin{mdframed}[backgroundcolor=yellow!10,rightline=false,leftline=false, topline = false, bottomline = false]}
\newcommand{\ehigh}{\end{mdframed}}

\definecolor{figred}{RGB}{221,162,162}

\allowdisplaybreaks

\newcommand{\range}[1]{\ensuremath{\{0,1,\dots, #1\}}}

\newcommand{\bogp}{$b$-OGP }

\newcommand{\const}{\eqref{eq:r:constraint}-\eqref{eq:uniConstraint} }

\newcommand{\mSet}{\mathcal{M}}
\newcommand{\pSet}{\mathcal{P}}

\newcommand{\stateScale}{Assume that $M=\mSet, p=\pSet$ are deterministic and satisfy Assumption \ref{as:MPScale}. }

\makeatletter \def\l@subsection{\@tocline{2}{0pt}{1pc}{5pc}{}} \def\l@subsection{\@tocline{2}{0pt}{2pc}{6pc}{}} \makeatother

\newcommand{\kpert}{\log\left(n\right)^{-1}}

\newcommand{\err}{c}

\usepackage{amsmath}

\begin{document}

\title[On the MCMC performance in Bernoulli Group Testing]{On the MCMC performance in Bernoulli Group Testing\\ and the Random Max-set cover problem}

\author[M. Lovig, I. Zadik]{Max Lovig$^\mathsection$, Ilias Zadik$^\mathsection$}
\thanks{
$^\mathsection$ Department of Statistics and Data Science, Yale University.\\ Emails:  \texttt{max.lovig@yale.edu}, \texttt{ilias.zadik@yale.edu}}

\begin{abstract}%
The group testing problem is a canonical inference task where one seeks to identify $k$ infected individuals out of a population of $n$ people, based on the outcomes of $m$ group tests. Of particular interest is the case of Bernoulli group testing (BGT), where each individual participates in each test independently and with a fixed probability. BGT is known to be an ``information-theoretically'' optimal design, as there exists a decoder that can identify with high probability as $n$ grows the infected individuals using $m^*=\log_2 \binom{n}{k}$ BGT tests, which is the minimum required number of tests among \emph{all} group testing designs. 

An important open question in the field is if a polynomial-time decoder exists for BGT which succeeds also with $m^*$ samples. In a recent paper (Iliopoulos, Zadik COLT '21) some evidence was presented (but no proof) that a simple low-temperature MCMC method could succeed. The evidence was based on a first-moment (or ``annealed'') analysis of the landscape, as well as simulations that show the MCMC success for $n \approx 1000s$. Interestingly, in (Coja-Oghlan et al COLT '22) it was proven that if $k=n^{\alpha}$ for $\alpha \in (0,1)$ small enough, all low-degree polynomials as decoders \emph{fail} to work with $m^*$ tests if $n$ is large enough, raising the stakes for the success of an MCMC method in that regime.

In this work, we prove that, despite the intriguing success in simulations for small $n$, the class of MCMC methods proposed in previous work for BGT with $m^*$ samples takes super-polynomial-in-$n$ time to identify the infected individuals, when $k=n^{\alpha}$ for $\alpha \in (0,1)$ small enough. We show that the suggested first-moment picture by the previous work has been an artifact of ``rare bad'' events, an issue that has not appeared before in the first-moment landscape analysis of similar sparse inference models. Appropriate conditioning and a delicate truncated second moment method, allow us to conclude that a certain disconnectivity takes place in the landscape of BGT, known as Overlap Gap Property for inference problems (Gamarnik, Zadik AoS '22), leading to bottlenecks for the MCMC methods. Towards obtaining our results, we establish the tight max-satisfiability thresholds of the random $k$-set cover problem, a result of potentially independent interest in the study of random constraint satisfaction problems.
\end{abstract}

% \begin{keywords}
% \end{keywords}

\maketitle

\date{\today}
 \newpage
\tableofcontents

\section{Introduction}

In this work, we focus on the group testing problem, introduced by Dorfman in \cite{Dorfman_1943}, which is the following statistical estimation problem. We have \( n \) individuals, of which \( k \) are ``infected'' by a certain disease of interest. Let us denote by $\sigma^* \subseteq [n], |\sigma^*|=k$ the $k$-subset of infected individuals. We assume an ``agnostic'' prior on $\sigma^*$, that is $\sigma^*$ is chosen uniformly at random among all $k$-subsets of the $n$ individuals. While the statistician is unaware of the infection status of each individual, they have access to a series of $N$ \emph{group tests}. Formally, for each of the $N$ tests, one chooses a subset \( \scr{C} \subseteq [n] \) to be tested. Then, the result of the test is defined as being positive if and only if at least one individual in the tested subset is infected,

\[
\text{Result}(\scr{C}) = \begin{cases}
    + & \text{if } \scr{C} \cap \sigma^* \not = \varnothing\\ 
    - & \text{otherwise}.
\end{cases}.
\]
As such tests are often applied in practice over a short time horizon, we focus on this work in the case of the so-called non-adaptive group testing, where we conduct all the \( N \) tests \emph{in parallel}. The ultimate goal of the statistician would be to identify the $k$ infected individuals by using the minimal possible number of tests, i.e., with the minimal possible $N$.

The group testing problem is naturally motivated by a series of real-world applications such as DNA sequencing \cite{Kwang_2006,Ngo_2000}, protein interaction experiments \cite{Mourad_2013,Thierry_2006} and machine learning \cite{Emad_2015}. Yet, perhaps the most recently relevant application was during the COVID-19 pandemic \cite{Mutesa_2021,McMahan_2012} where group testing has played a key role in multiple occasions such as reopening schools \cite{augenblick2022pooled}. On top of that, the underlying mathematical structure of group testing has also led it to be a topic of intense algorithmic and mathematical study (see e.g., the survey \cite{Aldridge_2019}). Interestingly, group testing is more relevant in reducing the number of required tests in practice when the prevalence of the infection (i.e., the ratio $k/n$) is small. For this reason, in this work as performed often in the theory of group testing we adopt the asymptotic sublinear setting that $n$ is growing to infinity, $n \rightarrow +\infty$ and $k = n^{\alpha + o(1)}$, for some $\alpha \in (0,1)$ (see \cite[Section 1]{Aldridge_2019} for a relevant discussion).

The problem admits a useful bipartite graph theoretic reformulation. Consider a bipartite graph with $n$ nodes on the one side that corresponds to the individuals, $k$ of which are infected, and $N$ nodes on the other side corresponding to the tests. We then can connect each test to an individual via an undirected edge if and only if the subset that corresponds to the test contains the individual, yielding an equivalent description of the group testing instance. 

%In agreement with standard graph-theory notation, in this work, we say that an individual ``covers" a test if they are included among the test's participants. 
\begin{figure}[ht]
    \centering
    \includegraphics[scale = .65]{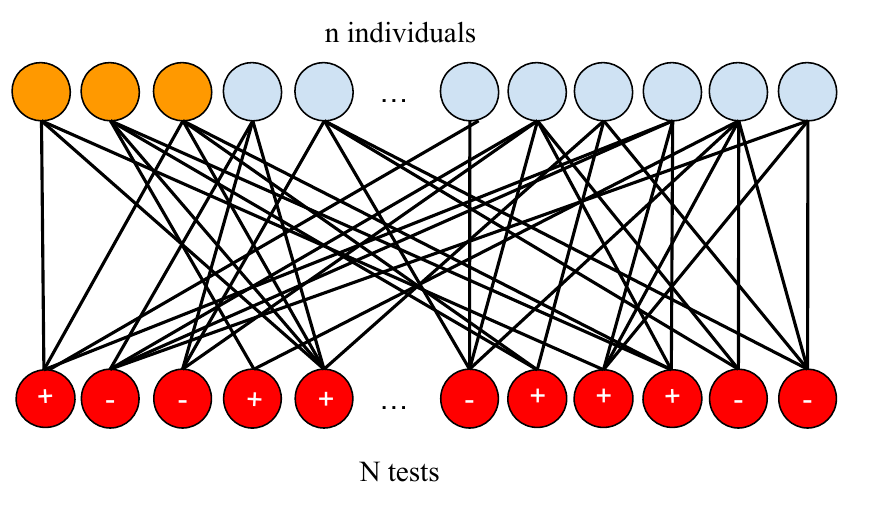}
    \caption{A realization for an instance of Bernoulli group testing.} %The goal is to find a set of \(k\) individuals that maximizes the number of positive tests covered while ensuring no negative tests are covered.}
    \label{fig:DoneBi}
\end{figure}

The construction and estimation in non-adaptive group testing can then be understood as the following two-step process: (1) first, we need to design the bipartite graph, i.e., determine which individual is included in which test, and (2) second, we need to choose a ``recovery'' algorithm which utilizes the resulting group tests outcomes from step 1 and outputs an estimator $\hat{\sigma}$ of $\sigma^*$.

Throughout this work, similar to earlier works such as \cite{scarlett2016phase, scarlett2018near, iliopoulos2021group, cojaoghlan2022statistical}, we focus on the following notion of successful estimation (or recovery) of $\sigma^*$, often called in the literature as ``almost perfect recovery". Specifically, our goal for step (2) above is to construct a $\hat{\sigma} \subseteq [n], |\hat{\sigma}|=k$ such that \begin{align}\label{eq:rec}
    \lim_n |\hat{\sigma}\cap \sigma^*|/k=1,
\end{align}
asymptotically almost surely (a.a.s.)\footnote{By a.a.s. throughout the paper, we refer to an event that holds with probability tending to one as $n$ grows to infinity.} with respect to the randomness of the prior of $\sigma^*$. In words, our goal is to recover asymptotically an $1-o(1)$ fraction of the infected individuals.

It is a folklore information theoretic argument in the literature of group testing that whenever $N \leq (1-\epsilon) \log_2 \binom{n}{k}$ for some $\epsilon>0,$ then there is no design of the group tests that can lead to a successful recovery algorithm \cite{Aldridge_2019,truong2020all, niles2023all}. Interestingly, the above result is tight as there are designs of group testing that lead to a successful recovery algorithm whenever $N \geq (1+\epsilon) \log_2 \binom{n}{k}$ for any $\epsilon>0$ \cite{Aldridge_2019}.

In this work we focus on one of the simplest such ``information-theoretically optimal'' designs called the \emph{Bernoulli group testing} design. This is a probabilistic design where for some $q \in (0,1)$ each individual is included in any given test independently with probability \( q \), leading to an Erd\H{o}s-Renyi structure in the associated bipartite graph. Interestingly, by appropriately choosing $q \approx \log 2/k$ \footnote{All logarithms in this work are with base $e$.}, it holds that whenever $N \geq (1+\epsilon) \log_2 \binom{n}{k}$ for some $\epsilon>0,$ the Bernoulli group testing design leads to a successful recovery algorithm, a.a.s. with respect to both the randomness of the prior and the Bernoulli design as $n \rightarrow +\infty$ \cite{Aldridge_2019}. The underlying reason for this striking success of the (vanilla) probabilistic method is a simple graph theoretic property which holds whenever $N \geq (1+\epsilon) \log_2 \binom{n}{k}$ in this setting (recall Figure \ref{fig:DoneBi}): any $k$-subset of the individuals that is \emph{covering} sufficiently many positive tests \footnote{We say that a $k$-subset of individuals ``covers'' a given test if at least one of the $k$ individuals took part in this test.} is almost-perfectly recovering $\sigma^*$ a.a.s. as $n \rightarrow +\infty$ (see e.g., \cite[Lemma 5]{iliopoulos2021group}). Due to this property, an interesting connection between Bernoulli group testing and the so-called \emph{random set cover problem} emerge -- we discuss more about this below. Now, given this property, a simple brute-force search algorithm over all $k$-subsets can solve the set cover problem and therefore recover the infected individuals for these values of $N$. 

While the Bernoulli group testing design is both simple to implement and optimal information-theoretically, it suffers from the fact that all known successful recovery algorithms require super polynomial-time to identify $\sigma^*$ when $N\geq  (1+\epsilon) \log_2 \binom{n}{k}$ for $\epsilon>0$ small enough. More specifically, exactly because of the $\mathcal{NP}$-hardness of the set-cover problem, as we also mentioned above, the optimal known decoding algorithm that works for step (2) whenever $ N= (1+\epsilon) \log_2 \binom{n}{k}$ for any $\epsilon>0$ requires in principle a brute-force search over all $k$-subsets and therefore has super-polynomial runtime in the worst-case. Since the Bernoulli group testing design is random, one can of course hope that some polynomial-time algorithm could also solve the set cover instance with a similar requirement on the test size to brute force search. Yet, the best known polynomial time recovery algorithm for this setting is known as Separate List decoding and requires $ N \geq (\log 2)^{-1} \log_2 \binom{n}{k}$ tests \cite{Aldridge_2019, scarlett2018near}, hence a multiplicative factor $1/\log 2\approx 1.44$ more tests compared to brute-force search approach. It remains unknown if some polynomial-time algorithm can achieve successful recovery for some $\log_2 \binom{n}{k} \leq N \leq (\log 2)^{-1} \log_2 \binom{n}{k}$. This potential trade-off between the running time and the required test size for any successful recovery algorithm places Bernoulli group testing into a family of statistical estimation tasks exhibiting what is known as a ``computational-statistical gap''; an area receiving a great deal of attention in recent works (see e.g., \cite{kunisky2019notes, gamarnik2022disordered} for two recent surveys). Albeit the fact that the gap in Bernoulli group testing is at the level of a different constant factor, in applications of group testing the multiplicative overhead in the required number of tests plays a major role. In fact, the study of this gap has been asked as one out of the nine main open problems for future work in the group testing survey \cite[Open Problem 3]{Aldridge_2019}.

Listening to the call of \cite[Open Problem 3]{Aldridge_2019}, researchers have already studied the ``hardness'' of this gap. The authors of \cite{cojaoghlan2022statistical} proved that no $O(\log n)$-degree polynomial estimator can recover $\sigma^*$ when $N <(\log 2)^{-1} \log_2 \binom{n}{k}$ as long as $k=n^{\alpha+o(1)}$ for $\alpha \in (0,1)$ a small enough constant \footnote{Formally, the lower bound has been proven for a detection variant of the model, but it is customary expected to generalize to the estimation question we focus on this work.}. Now, this low-degree lower bound is also potentially offering more than solely a rigorous lower bound against a large class of powerful estimators. It is intriguingly conjectured in the community of computational-statistical gaps that $O(\log n)$-degree polynomials as estimators are capturing the power of all polynomial-time estimators, something formalized for detection tasks in what is known as the ``low-degreee conjecture'' \cite{sam-thesis}. In particular, based on the above conjecture, \cite{cojaoghlan2022statistical} provides strong evidence that the computational statistical gap of Bernoulli group testing could be fundamental and no polynomial-time algorithm can succeed when $N <(\log 2)^{-1} \log_2 \binom{n}{k}.$

One year earlier compared to \cite{cojaoghlan2022statistical}, but again motivated by \cite[Open Problem 3]{Aldridge_2019}, \cite{iliopoulos2021group} also studied the computational-statistical gap but from a ``landscape'' point of view. They investigated whether a bottleneck for certain MCMC methods attempting to identify $\sigma^*$ appears in the landscape of Bernoulli group testing in the regime $\log_2 \binom{n}{k} \leq N \leq (\log 2)^{-1} \log_2 \binom{n}{k}$. The bottleneck is often referred to as ``Overlap Gap Property (OGP) for inference" \cite{gamarnik2022sparse}. For simplicity, we refer to this property as bottleneck-OGP (b-OGP) from now on. b-OGP in Bernoulli group testing refers to the phenomenon that for all $k$-subsets $\sigma \subseteq [n]$ which cover sufficiently many positive tests, the number of infected individuals in $\sigma$ (i.e., $|\sigma \cap \sigma^*|$) is either ``small'' (often due to high entropy effects) or ``large'' (as $\sigma^*$ covers all positive tests by definition). In particular, for any such $\sigma,$ $|\sigma \cap \sigma^*|$ cannot take a growing number of ``medium'' values. b-OGP is known to imply in many similar problems slow mixing for natural families of low-temperature MCMC methods that try to identify $\sigma^*$ (see e.g., \cite{gamarnik2021overlap, gamarnik2022sparse, arous2023free, gamarnik2024overlap, chen2023almost, chen2024low}). Moreover, $b$-OGP has been known to coincide with the threshold for the fast/slow mixing of low-temperature MCMC methods for a number of models, including sparse regression \cite{gamarnik2022sparse, chen2024low}, planted clique \cite{gamarnik2024overlap} and sparse tensor PCA \cite{chen2024low}.

The authors of \cite{iliopoulos2021group} showed that under the assumption of sufficient concentration of certain key quantities around their expectation (also called ``first-moment'' approximations, or ``annealed'' approach in statistical physics \cite{zdeborova2016statistical}) then b-OGP should in fact \emph{never be present} for Bernoulli group testing for any $N \geq (1+\epsilon)\log_2 \binom{n}{k}, \epsilon>0$. Judging on other models where b-OGP appears exactly when the low-temperature MCMC methods fail to identify in polynomial-time the planted signal $\sigma^*,$ the authors of \cite{iliopoulos2021group} asked whether these MCMC methods are always able to identify $\sigma^*$ in polynomial-time throughout the information-theoretic possible regime. Albeit an interesting question, the authors of \cite{iliopoulos2021group} do not prove that b-OGP is never present (let alone that the MCMC methods indeed identify $\sigma^*$ in polynomial-time) because the required concentration results appeared significantly difficult to establish. Despite that, they simulated these low-temperature MCMC methods for $n \approx 10^3$ and observe that indeed whenever $N= (1+\epsilon)\log_2 \binom{n}{k}$ for any $\epsilon>0$ they quickly find a $k$-subset that covers all positive tests, which as we mentioned above is sufficient to recover $\sigma^*$ for large enough $n$ \cite[Lemma 5]{iliopoulos2021group}. Besides the clear importance of proving any such positive result, the stakes are also raised given the discussed low-degree lower bound from \cite{cojaoghlan2022statistical}. Indeed, if MCMC were successful, it would be the first time in the literature of computational statistical gaps that an MCMC method run for polynomial time can provably outperform all $O(\log n)$-degree polynomials. Moreover, it would contradict any extension of the low-degree conjecture from detection tasks \cite{sam-thesis} to estimation tasks. One of the main motivations of this work is to understand whether such a significant advantage of MCMC method exists or not for Bernoulli group testing.

Notably, besides \cite{iliopoulos2021group}, we are not aware of any other theoretical work on MCMC methods for Bernoulli group testing. On the other hand, multiple applied papers have used MCMC methods for group testing \cite{schliep2003group,knill1996interpretation,furon2012decoding} and it is the general understanding that their \emph{``...empirical performance appears strong in simulations ''} \cite[Section 3.3.1]{Aldridge_2019}. To buttress these claims, it is essential to pursue an improved theoretical understanding of MCMC methods for information-theoretic optimal designs such as Bernoulli group testing, which is the central focus of this work.

Lastly, as briefly also mentioned above, the Bernoulli group testing is inherently connected with the random (or average-case) set cover problem. The set cover problem has been one of the 21 famous Karp's $\mathcal{NP}$-complete problems \cite{karp2010reducibility}, which yet remains one of the least well-understood among them on the average case. Indeed, only a few mathematical results have been established for this setting \cite{telelis2005absolute, arpino2023greedy} and, while they are very interesting, they provide only ``up to constants'' results, not offering a sufficiently tight understanding for our group testing application. It should be noted that a somewhat tighter but non-rigorous analysis is offered via statistical physics methods in \cite{mezard2007statistical}. Our relatively poor mathematical understanding of the random set-cover problem remains in sharp contrast with the very rich and detailed understanding of the community of the average-case analysis of other famous $\mathcal{NP}$-complete problems such as random SAT (see e.g., the seminal work by Ding, Sly, and Sun \cite{Ding_sat} and references therein), or more classical settings such as the random subset sum problem \cite{lagarias1985solving, frieze1986lagarias} and the maximum clique problem in random graphs dating back to the original work of Bollobas and Erd\H{o}s \cite{bollobas1976cliques}. In this work, we offer significantly tight results for the random set cover model, by exactly identifying up to $o(1)$ additive error the so-called maximum satisfiability thresholds for the problem. Our result is analogous to the celebrated work on the maximum satisfiability thresholds for random SAT by Achlioptas, Naor and Peres in \cite{achlioptas2007maximum}, and could be of independent interest.

\subsection{Contributions}

In this work, our main focus is on the power of MCMC methods for Bernoulli group testing. The gist of our theoretical results on this topic is a new strong negative result on MCMC methods. We prove that the class of low temperature MCMC methods suggested in  \cite{iliopoulos2021group} is not only unable to ``close'' the computational trade-off for Bernoulli group testing (answering the main question of \cite{iliopoulos2021group}), but in fact it is even underperforming compared to the best known polynomial-time algorithms for the setting (conceptually agreeing with a series of recent works on low temperature MCMC methods on different inference models \cite{chen2023almost, chen2024low}). 

\subsubsection{Existence of b-OGP}We start with turning to the open question for b-OGP as raised in \cite{iliopoulos2021group}. Our first result is that contrary to the first moment analysis of \cite{iliopoulos2021group} b-OGP does in fact exist for a part of the information theoretically possible regime for Bernoulli group testing. This is somewhat surprising given the success of the first moment landscape analysis in multiple inference settings, including sparse regression \cite{gamarnik2022sparse}, planted clique \cite{gamarnik2024overlap} and sparse tensor PCA \cite{arous2023free, chen2024low}. We summarize this finding in an informal theorem. \begin{theorem}[Informal theorem, see Theorem \ref{thm:QualInc_0}]\label{thm:1.4749}
    For Bernoulli group testing, suppose $k=\floor{n^{\alpha}}$ for some constant $\alpha \in (0,1)$ which is less than a sufficiently small constant. If the test size satisfies $N  \leq 1.4749 \log_2 \binom{n}{k}$ then b-OGP exists a.a.s. as $n \rightarrow +\infty.$
\end{theorem}

The reason for the discrepancy to the prediction in \cite{iliopoulos2021group} is that, as we prove, the conjectured concentration around the first-moment approximations, stated in  \cite[Conjecture 26]{iliopoulos2021group}, turns out to be incorrect. The underlying mathematical reason is the existence of certain rare ``lottery'' events that cause the first moment to ``explode" but yet are misleading as they can be conditioned away. Indeed in this work, we identify these atypical events, which depend on the fluctuations of the degrees of the infected individuals. Then we appropriately condition the first moment approximations from \cite{iliopoulos2021group} on them, and execute a technical but delicate first and second moment method to prove the correctness of these now conditional first moment approximations.  The exact constant $1.4749$ is computed via numerical methods (see Section \ref{sec:num} for more details on this).
 
 \subsubsection{MCMC lower bound}\label{sec:mcmc_intro} Following recent but relatively standard tools from the literature (see e.g., \cite{arous2023free, chen2024low}), we then prove that because b-OGP appears, all elements of a natural class of low-temperature local MCMC methods fail to identify the set of infected set of individuals $\sigma^*$ in polynomial-time.
 
 More specifically, the focus is on the following class of Markov chains. As explained above, a sufficient condition for the recovery of $\sigma^*$ if $N \geq \log_2 \binom{n}{k}$ is to find any $k$-subset that ``covers'' all the positive tests \cite[Lemma 5]{iliopoulos2021group}. Hence, it is natural to focus on Markov chains that attempt to maximize this objective by having a stationary measure supported on $k$-subsets $\sigma \subseteq [n]$ given by \[\pi_{\beta}(\sigma) \propto \exp\left(-\beta \frac{\# \text{ of positive tests uncovered by }\sigma}{M}\right)\]for sufficiently large values of $\beta>0$ (or equivalently of sufficiently ``low-temperature''). Now we also focus on ``local'' Markov chains, meaning the underlying neighborhood graph on the $k$-subsets of $[n]$ connects two subsets if and only if their Hamming distance equals to 2, i.e., the chain swaps one individual at every step. This neighborhood graph is also commonly referred to as the Johnson graph \cite[p. 300]{Holton_Sheehan_1993}.
 
 We prove the following corollary of our b-OGP result.
 
 \begin{corollary}(Informal corollary, see Corollary \ref{thm:mix})\label{cor:mcmc_inf}
     For Bernoulli group testing, suppose that $k=\floor{n^{\alpha}}$ for some constant $\alpha \in (0,1)$ which is less than a sufficiently small constant. If $\beta \geq \Co{beta} k \log(n/k)$ for a sufficiently large $\Co{beta}>0$ and $\log_2 \binom{n}{k} \leq N \leq 1.4749 \log_2 \binom{n}{k}$ then all local Markov chains with stationary measure $\pi_{\beta}$ take super-polymomial time to recover $\sigma^*$, a.a.s. as $n \rightarrow +\infty.$ 
 \end{corollary}
 Two key remarks are in order:
 \begin{itemize}
    \item[(a)] Our results prove that all these MCMC methods fail to surpass the lower bound against $O(\log n)$-degree polynomials as proven in \cite{cojaoghlan2022statistical}, settling the main question from \cite{iliopoulos2021group}. This provides further support for the ``low-degree conjecture'' in the context of statistical estimation.
    \item[(b)] Notice that these MCMC methods in fact fail to even achieve the performance of Separate List Decoding (SLD), the currently best known polynomial-time algorithm for Bernoulli group testing, as SLD works whenever $N \geq (\log 2)^{-1} \log_2 \binom{n}{k}$ \cite{Aldridge_2019} and of course $1/\log 2<1.47$. This is another case of a provable underperformance of low temperature MCMC methods for statistical estimation tasks (known as local-to-computational statistical gap) which is similar in spirit to works on Langevin dynamics for tensor PCA \cite{arous2020algorithmic} and the Metropolis process for the planted clique model \cite{chen2023almost} and sparse tensor PCA model \cite{chen2024low}. 
 \end{itemize} 

\subsubsection{Random MAX $k$-set cover}\label{sec:intro_cover}
As we mentioned above, towards proving the existence of b-OGP which led to the MCMC lower bound, we interestingly need to tackle a problem in the study of random constraint satisfaction problems of independent interest. Specifically, to prove the b-OGP we need to understand tightly how many positive tests any $k$-subset of individuals can cover, which entails to studying \emph{the random MAX $k$-set cover problem} which we describe as follows in an independent way from Bernoulli group testing. 

 Let $n$ be a growing parameter and consider for some $p=p_n$ a universe of $[p]$ elements. Then, for some $q=q_n \in (0,1)$ we independently sample $M=M_n$ subsets of $[p]$, $\mathcal{S}_i, i=1,\ldots,M$ where each element appears with probability $q$ in an i.i.d. fashion. We say that a $k$-subset of $[p]$ covers one $\mathcal{S}_i$ if it has non-empty intersection with it.  The random MAX-set cover problem asks for a given $k=k_n$ what is the asymptotic value of \begin{align}\label{max_sat}
    \Phi_k:=\max_{\sigma \subseteq [p], |\sigma|=k} \frac{\# \text{ of } \mathcal{S}_i, i=1,\ldots,M \text{ covered by }\sigma}{M},
\end{align}that is of the maximum fraction of the number of the $M$ random sets that some $k$-subset of $[p]$ can intersect or cover. One can easily convince themselves of the relation to Bernoulli group testing, where $p$ corresponds to the number of non-infected individuals \footnote{Later, we explain that $p$ in fact should correspond to the number of non-infected but ``possibly infected'' individuals. For simplicity, we omit this detail for now.} and the ``target'' sets $\mathcal{S}_i$ correspond to the positive tests. Then $\Phi_k$ corresponds to the fraction of the positive tests that can be covered by some $k$-subset of non-infected individuals.

In the literature of random constraint satisfaction problems (CSP), the random variable $\Phi_k$ is a well-known quantity which is also commonly referred to as the max-satisfiability thresholds of a random CSP, in particular here of random set cover. The max-satisfiability threshold is meaningful in the ``unsatisfiable'' regime of a random CSP where it quantifies how many constraints can be possibly satisfied. A quite attractive feature that motivates the detailed study of the max-satisfiability thresholds of $\mathcal{NP}$-hard problems (such as $k$-set cover) in the average-case is that they shed light to interesting connections with approximation complexity. A notable such result is the celebrated Feige's hypothesis \cite{feige2002relations} which revealed connections between the hardness of achieving the max-satisfiability thresholds for random 3-SAT via polynomial-time methods, and the approximation complexity of a series of other $\mathcal{NP}$-hard problems. For this reason, researchers have studied in detail the asymptotic properties of max-satisfiability thresholds $\Phi_k$ of random $\mathcal{NP}$-hard problems. Notable such results include general asymptotic formulas for the thresholds by leveraging connections with spin glass theory such as \cite{sen2018optimization, panchenko2018k,jones2022random}, but also even tighter more precise results such as the seminal work by Achlioptas, Naor and Peres on the thresholds of random k-SAT \cite{achlioptas2007maximum}. We highlight that understanding the max satisfiability thresholds at a similar level of precisions as in \cite{achlioptas2007maximum} is an arguably significant mathematical task accomplished only in limited cases, often involving a delicate second moment method argument.

In our work, we calculate the max-satisfiabilty thresholds of random set cover, $\Phi_k$,  which is a significant departure compared to random SAT. Interestingly, we achieve a comparable level of precision as \cite{achlioptas2007maximum} by identifying its asymptotic value up to $o(1)$ error. Perhaps unsurprisingly our proof proceeds by a careful conditional second moment method. Yet the application of the method is quite delicate and our main technical tool is to employ an appropriately adjusted version of the so-called ``flatness'' technique, while executing the second moment method. The flatness idea was initiated in the study of the densest subgraph problem in random graphs \cite{balister2018dense,gamarnik2024overlap} and has recently been applied also in analyzing the sparse principal components of a Gaussian tensor \cite{chen2024low}. To the best of our knowledge, this is the first time this technique has been adjusted to work in the context of sparse random graphs.

%In this work, we prove a tight concentration result for the value $\Phi_k$ for a large regime of the random MAX $k$-set cover problem, similar to the celebrated result for $k$-SAT in \cite{achlioptas2007maximum}. Moreover, also our proof follows from a significantly delicate second-moment method. Our key technique enabling us to get this results is to employ some inspiring ``flatness'' ideas from random graph theory \cite{balister2018dense, gamarnik2019landscape} for our second moment method analysis.

In terms of parameters, we choose $q=q_n$ so that $(1-q)^k=1/2$. As we discussed above, this is a natural choice in the Bernoulli group testing literature \cite{Aldridge_2019}, but it also provides an elegant normalization from a random CSP point of view. Indeed, this choice of $q$ implies that a uniform random $k$-subset of $[n]$ covers each $\mathcal{S}_i$ with probability exactly $1/2$. Hence, by the law of large numbers, a uniform random $k$-subset of $[n]$ will cover $1/2+o(1)$-fraction of the $M$ random sets a.a.s. as $n \rightarrow +\infty.$ In contrast, the quantity of interest $\Phi_k$ concerns what is the maximum possible fraction that can be covered by any $k$-subset of $[n]$ and can be compared with the $1/2$ fraction which is the performance of the trivial ``random guess'' algorithm.

It also turns out that for $\Phi_k$ to be asymptotically constant, we need to choose $M$ to scale like the entropy of the feasible region, i.e., $M=\Theta(\log \binom{n}{k})$, which is the scaling we adopt. Finally, to exactly follow the corresponding scaling for our Bernoulli group testing application we appropriately assume that for some parameters $\alpha \in (0,1), C \in (1,2)$ $k=n^{\alpha+o(1)}$, $M=(C/2+o(1)) \log_2 \binom{n}{k}$  and $p=n(k/n)^{C/2+o(1)}=n^{1-(1-\alpha)C/2+o(1)}$ (see Section \ref{sec:post_proc_intro} for further motivation and the exact this choice of scaling). We remark that, albeit natural in Bernoulli group testing, the perhaps stringent dependence of $p$ on $n$ is expected to be able to be generalized using a variation of our proof technique (see Remark \ref{rem:gen} for a relevant discussion).

Under these assumptions, we prove the following result which exactly characterizes the limiting value of $\Phi_k$. 

\begin{theorem}(Informal theorem, see Theorem \ref{thm:Phi_kLim})
     Let $n \rightarrow +\infty$. For any $C \in (1,2)$ and $\alpha \in (0,1)$ sufficiently small, if $k=\floor{n^{\alpha}}$, $ M=\floor{C \log_2 \binom{n}{k}/2}$ and $p=n^{1-(1-\alpha)C/2+o(1)}$ then a.a.s. as $n \rightarrow +\infty,$ 
     \[
         \lim_{n \rightarrow +\infty} \Phi_k=1-h^{-1}_2\left(2-2/C\right),
     \]where $h_2(x):=-x\log_2 x-(1-x)\log_2 (1-x), x \in [0,1/2]$ is the left branch of the binary entropy. 
 \end{theorem}

A plot of the limiting $\Phi_k$ as a function of $C$ versus the performance of ``random guess'' is shown in Figure \ref{fig:PhiVal}.
 \begin{figure}[ht]\label{fig:PhiVal}
    \centering
    \begin{overpic}[scale=1]{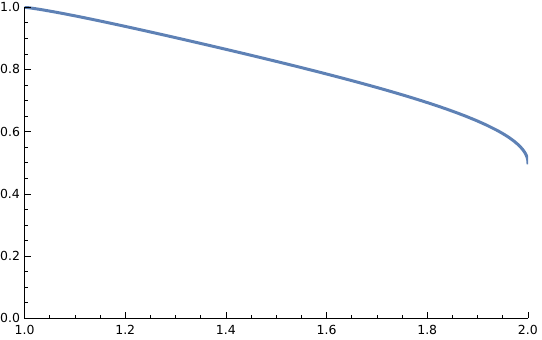}
        \put(48, -2){$C$}
        \put(-5,33){$\Phi_k$}
        \put(5,32.5){\tikz \draw[dashed,black] (0,0)--(8.5,0);}
        \put(15, 28){Trivial Lower Bound via Random Guess}
    \end{overpic}
    \caption{ $\Phi_k$, the maximal proportion of covered sets for some size $k$ set of elements, as a function of $C$ for the \kset problem, where $C$ control the number of ``target'' sets (or constraints).}
\end{figure}
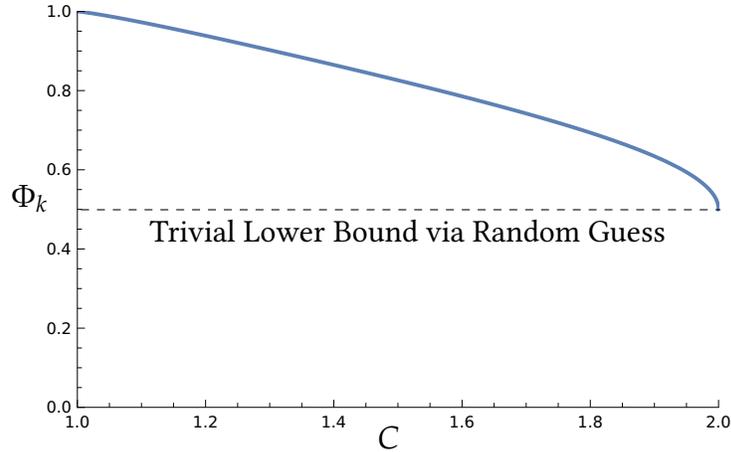

%-----------------------

\subsection{Notation}

\begin{comment}
\begin{notation}
    Define a collection of constants $C_i$ for $i \in \mathbb{N}$, each $C_i \geq 0$, labelled in order of appearance. Unless otherwise specified, these constants are universal. Constants without a numerical sub-script will vary from context to context. Also, constants used in statements such as lemmas and theorems may show up in our proofs, when such a case occurs it is because we constructed the constant over the course of the proof. This may lead to a situation when we have used $C_{14}$ and, suddenly, $C_{12}$ is introduced without warning. This is because $C_{12}$ was used in the statement of the Lemma/Theorem and is that explicit constant. \textcolor{red}{IZ: This is confusing.}
\end{notation}
\end{comment}

    We use standard asymptotic notation. For any two positive sequences $A_n, B_n, n \in \mathbb{N},$ we write $A_n \oleq{B_n}$ if and only if $\limsup_n A_n/B_n <+\infty$, $A_n= \Omega(B_n)$ if and only if $B_n \oleq{A_n}$, $A_n \oeq{B_n}$ if and only if $A_n \oleq{B_n}$ and $B_n \oleq{A_n}$, $A_n \ole{B_n}$ if and only if $\lim_n A_n/B_n=0$ and $A_n=\omega(B_n)$ if and only if $B_n \ole{A_n}$.

     We say that a sequence of events $(A_n)_{n \in \mathbb{N}}$ happen asymptotically almost surely (a.a.s) if and only if $\lim_{n \conv{} \infty}\P(A_n) = 1$ as \ngrow.
     
    Given a function $f$ of possibly many variables, one of which is $\gamma$, define $\partial_\gamma f$ to represent the derivative of $f$ with respect to the variable $\gamma$. We also denote for $q_1,q_2 \in [0,1]$, the two point Kullback-Leibler (KL) divergence by
    \begin{align}\label{eq:two_KL}
    D(q_1||q_2) = q_1 \log(q_1/q_2) + (1-q_1)\log((1-q_1)/(1-q_2)).
    \end{align}
    Also we denote for any $C>1$,
    \begin{align}\label{def:H_C}
    H_C := h_2^{-1}(2-2/C),
    \end{align} where $h_2$ is the left branch of the binary entropy function.

Finally, throughout the paper, we denote some important positive constants by $C_i, i \in \mathbb{N}$. Importantly, $C_i$ will represent a specific constant when defined and will never change its value between two instances. There will also be a collection of constants using a different notation (such as $C>0$) and these constants can vary from context to context.

%\subsection{For Reference: Scaling Of The Parameters}
%For the reader's convenience, the following table contains a description and the relative scale for each major parameter in the paper:

%\begin{table}[H]
%\begin{tabular}{c|l|l}
%Parameter & Description & Scale \\ \hline
 %   $n$      &      Number of individual units to be tested       &    $n$   \\
 %    $k$     &      Number of infected individuals       &  $n^\alpha$ ($\alpha \in (0,1)$)    \\
   %   $N$    &      Number of tests       &    $C k %\log_2(n/k)$ ($C \in (1,2)$)  \\
   %   $q$    &   Probability of an individual being in %a test & $\approx \log(2)/k$\\
  %  $M$      &      Number of positive tests       &   % $\approx N/2$  \\
   %  $p$     &     Number of possible infecteds after %post-processing & $\approx n (k/n)^{C/2}$        
%\end{tabular}
%\end{table}

%------------------------------

\section{Getting Started}\label{sec:set_up}
In this section, we provide some required background to formally state our main results.
\subsection{Set-up}\label{sec:setup} We start with properly defining the Bernoulli group testing instance. Consider \(n\) to be the number of individuals. We assume that $n$ grows to infinity and all other growing parameters grow as a function of $n.$  

\begin{definition} Fix some constants $\alpha \in (0,1)$ and $C>1$. We call the $(\alpha,C)$-group instance the following setting. Among the $n$ individuals, we assume there is a subset of $k=\floor{n^{\alpha}}$ infected ones, denoted by $\sigma^*,$ which are chosen uniformly at random among all $k$-subsets of $[n]$. 

The statistician observes \(N = \floor{C  \log_2 \binom{n}{k}}\) group tests, where each individual participates in each test with an assignment probability \(q \in (0,1)\) satisfying
\[
(1-q)^k = \frac{1}{2}.
\label{eq:qEq}\]
\end{definition}The goal of the statistician is given an $(\alpha,C)$-instance and complete knowledge of the parameters, to construct a $k$-subset $\hat{\sigma} \subseteq [n]$ such that a.a.s. as \ngrow the recovery condition \eqref{eq:rec} holds.

\begin{remark}
    We make a few remarks on the choice of the parameters.  First, the choice $C>1$ is necessary because if $C<1,$ a standard information-theory packing argument implies that no $\hat{\sigma}$ is possible to be constructed for Bernoulli group testing so that \eqref{eq:rec} holds \cite{Aldridge_2019}. Second, the assumption on $q$ satisfying \eqref{eq:qEq} is also standard in Bernoulli group testing, and it is motivated by the fact that for this exact choice of $q$ some (time-inefficient) $\hat{\sigma}$ is possible to be constructed whenever $C>1$ so that \eqref{eq:rec} holds (see e.g., \cite[Lemma 5]{iliopoulos2021group}). It will be also convenient for us to notice the asymptotic that as \(n\) grows it holds \(q = (\log(2)+o(1))/k\). Moreover, with this choice, each test is positive with probability \(1/2\), resulting in \(M = (1+o(1))N/2\) positive tests, a.a.s. as $n \rightarrow +\infty.$
\end{remark}

\subsection{Post-processing step}\label{sec:post_proc_intro}

We start with an important post-processing step that most algorithmic constructions for the estimators $\hat{\sigma}$ naturally apply as a first step to a vanilla Bernoulli group testing instance, as pictured in Figure \ref{fig:DoneBi}. Notice that each negative test must be testing only non-infected individuals. Hence a natural post-processing step, known also as Combinatorial Orthogonal Matching Pursuit (COMP) \cite{Aldridge_2019}, is to immediately discard from consideration all individuals participating in at least one negative test. Interestingly, after this removal step, if $C>2,$ COMP outputs only the infected individuals a.a.s. as \ngrow and hence recovers $\sigma^*$ \cite{Aldridge_2019}. In particular, if $C>2$ the recovery problem can be considered trivial, and from now on we only consider the regime $1<C<2$.

Moreover, in this regime where $1<C<2,$ Lemma \ref{lem:M:limit} and Lemma \ref{lem:p:limit} (which follow from standard concentration of measure inequalities), imply that there are $M=(1+o(1)) N/2$ positive tests and $p=(1+o(1)) n(k/n)^{C/2}+k$ remaining individuals that are possibly infected. Pictorially, this post-processing step when applied to Figure \ref{fig:DoneBi}, results in Figure \ref{fig:PPBi}.
\begin{figure}[h]
    \centering
    \includegraphics[scale = .65]{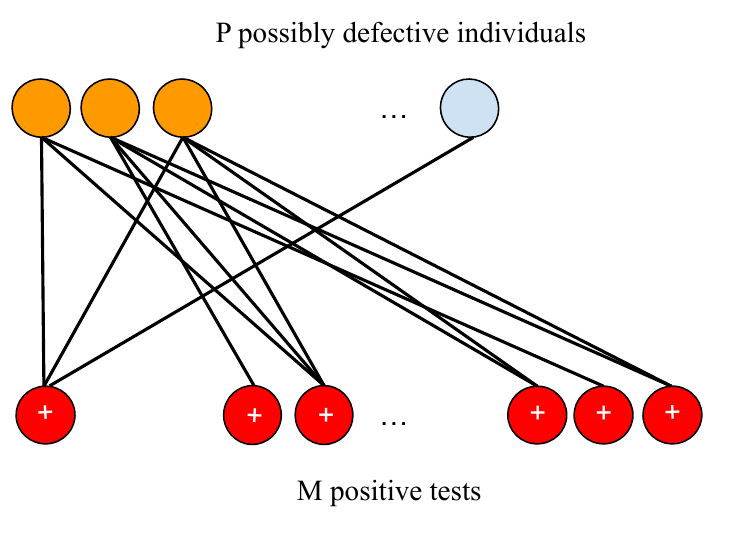}
    \caption{A realization for an instance of Bernoulli group testing, now with the COMP post-processing applied.}
    \label{fig:PPBi}
\end{figure}

\subsection{The information-theoretic, but time-inefficient, optimal algorithm}
As we mentioned above, for arbitrary $\alpha \in (0,1)$ and for all $C>1$, there exists a time-inefficient $\hat{\sigma}$ that can recover $\sigma^*$ per \eqref{eq:rec}. We now explain the details.

This algorithm consists of first applying COMP as above and then outputting any $k$-subset $\sigma$ of the $p$ possibly infected individuals, that ``covers'' all positive tests (i.e., any positive test is connected to at least one individual in $\sigma$) in the post-processed Bernoulli group testing instance (Figure \ref{fig:PPBi}). The success of this algorithm is standard in the literature, see e.g., \cite[Lemma 5]{iliopoulos2021group}. The natural implementation of this strategy is to brute-force search over all $k$-subsets of $[n]$ and output the first one that covers all the positive tests. While this algorithm successfully recovers $\sigma^*$ a.a.s. as \ngrow\!, from a run-time point of view it needs to visit $\binom{p}{k}=\exp( \Theta(k \log (p/k)))$ subsets in the worst-case, which of course is not polynomial-in-$n$ as $k=\Theta(n^{\alpha})$ and $p = \Theta\left(n (k/n)^{C/2}\right) = \omega(k)$ for $1<C<2$.

As we mentioned in the Introduction the success of this algorithm for all $C>1$ should be understood in contrast with SLC, the best known polynomial-time estimator for Bernoulli group testing, which recovers $\sigma^*$ only when $C>1/\log 2 \approx 1.44.$

\subsection{Markov chains}\label{sec:MCMC} The primary motivation of this work is the performance of Markov chains in constructing an estimator $\hat{\sigma}$. 

Now, as the information-theoretical optimal estimator is to brute-force search for a $k$-subset $\sigma \subseteq [p]$ (recall that $p$ is the set of possible infected individuals) that covers all the positive tests, equivalently our goal is to minimize the (normalized) Hamiltonian,
\begin{align}\label{eq:H} H(\sigma) \coloneqq \# \text{ of positive tests non-covered by }\sigma /M,
\end{align}
over all $k$-subsets $\sigma.$ 

Viewed from this perspective, a natural ``local-search'' approach to try to approximately minimize $H(\sigma)$ and recover $\sigma^*$ is to run a Markov chain with state space all $k$-subsets $\sigma$ and stationary distribution given by $\pi_{\beta}(\sigma) \propto \exp(-\beta H(\sigma))$, for a sufficiently large choice of $\beta.$ This leads to the class of ``low-temperature local MCMC methods" defined in Section \ref{sec:mcmc_intro} over the Johnson graph. 

For concreteness, a popular such example would be simply running the Glauber Dynamics, described as follows.

\begin{definition}\label{def:Glauber}
    Let $d_H$ be the Hamming distance on $k$-subsets. Given a group testing instance, we define the Glauber Dynamics over $k$-subsets and inverse temperature $\beta$ to have transition kernel $P_\beta(\sigma,\sigma')$ given by,
    \[P_{\beta}(\sigma, \sigma') = \begin{cases}
        \frac{1}{k(p-k)}\frac{\exp({-\beta H(\sigma')})}{\exp(-\beta H(\sigma'))+\exp(-\beta H(\sigma))} & \text{if } d_H(\sigma, \sigma') = 2, |\sigma| = k,\\
        \sum_{\sigma':d_H(\sigma,\sigma')=2}\frac{1}{k(p-k)}\frac{\exp({-\beta H(\sigma)})}{\exp(-\beta H(\sigma'))+\exp(-\beta H(\sigma))} & \text{if } \sigma = \sigma'\\
        0 & \text{otherwise}.
    \end{cases}\]
\end{definition}

%--------------------

\section{Main Results}
In this section we formally present our landscape \bogp results, resulting in our lower bounds for low-temperature MCMC methods. In all that follows, as explained in Section \ref{sec:set_up} we consider only the $p$ possibly infected individuals and subsets $\sigma$ of them.  Similar to \cite{iliopoulos2021group}, our first key step is to study the following (random) restricted optimization problems over $l \in \range{k}$,

\[\phi(\ell):= \min \{ H(\sigma): |\sigma|=k, |\sigma \cap \sigma^*|=\ell\}\label{eq:phil},\]where $H$ is defined in \eqref{eq:H}. 

The non-monotonicity of $\phi(l)$ is known to be linked with \bogp \cite{gamarnik2022sparse}, defined as follows.

\begin{definition}\label{def:OGP}
    Let constants $\zeta_1, \zeta_2 \in [0,1]$ with $\zeta_1 < \zeta_2$, threshold value $r=r_n> 0$ and height value $\delta=\delta_n > 0$. A group testing instance exhibits the bottleneck Overlap Gap Property (\bogp) for parameters $\zeta_1, \zeta_2, r, \delta$ if the following conditions hold.
    \begin{enumerate} 
        \item There exist size $k$ subsets $\sigma_1,\sigma_2$ with $\frac{1}{k}|\sigma_1 \cap \sigma^*| \leq \zeta_1$, $\frac{1}{k}|\sigma_2 \cap \sigma^*| \geq \zeta_2$, for which it holds $\max\{ H(\sigma_1), H(\sigma_2)\} < r.$
        
       %\item $\max_{0 \leq l \leq \zeta_1k}\min_{\sigma: |\sigma \cap \sigma^*| = l, |\sigma| = k}H(\sigma) \leq r$ and $\max_{\zeta_2k \leq l \leq 1}\min_{\sigma: |\sigma \cap \sigma^*| = l, |\sigma| = k}H(\sigma) \leq r$.
        \item For any $k$-subset $\sigma$ with $|\sigma \cap \sigma^*| \in [\zeta_1, \zeta_2]$ it holds $H(\sigma) \geq r + \delta$.
    \end{enumerate}
\end{definition}

%\begin{mdframed}
%    OLD

%\begin{definition}
 %   Let $0 \leq r=r_n \leq 1$  and $0 < \zeta_1=(\zeta_1)_n < \zeta_2 = (\zeta_2)_n  < \zeta_3=(\zeta_3)_n < \zeta_4 = (\zeta_4)_n < 1$ with $\zeta_{i+1} - \zeta_i = \Theta(1)$ for $i \in \{1,2,3\}$. 
    
  %  An $(\alpha, C)$ group testing instance exhibits the bottleneck Overlap Gap Property (\bogp) for parameters $r, \zeta_1, \zeta_2, \zeta_3, \zeta_4$ if the following conditions hold. 
    
  %  \begin{enumerate}
   %     \item There exist size $k$ subsets %$\sigma_1,\sigma_2$ with $\frac{1}{k}|\sigma_1 \cap \sigma^*| \leq \zeta_1$, $\frac{1}{k}|\sigma_2 \cap \sigma^*| \geq \zeta_4$, for which it holds $\max\{ H(\sigma_1), H(\sigma_2)\} < r.$ 
    %    \item There exists a region, $R = \{\sigma:|\sigma| = k, \frac{1}{k}|\sigma \cap \sigma^*| \in (\zeta_2, \zeta_3)\}$, where $\min_{\sigma \in R}H(\sigma) > r$.
   % \end{enumerate}
%\end{definition}
%\end{mdframed}
It is well-known in the literature that \bogp is related to the (non)-monotonicity of $\phi(l)$. Indeed, \cite[Lemma 20]{iliopoulos2021group} implies that the non-monotonicity of $\phi(l)$ is necessary for the existence of  \bogp\! and a simple argument, used for example in \cite[Theorem 2]{gamarnik2024overlap}, implies that the non-monotonicity of $\phi(l)$ is also sufficient for the existence of \bogp\!.  

Characterizing $\phi(l)$ leads to studying the count of size $k$ subsets $\sigma$ which have a given overlap $l$ and objective value $t$. 

\begin{definition}\label{def:Ztl}
    For $t \in \range{M}, l \in \range{k}$ define $Z_{t,l}$ to be the random variable
    \[Z_{t,l} = |\{\sigma : |\sigma| = k, |\sigma \cap \sigma^*| = l, \sigma \text{ leaves at most }t\text{ positive tests uncovered}\}|\]
\end{definition}Notice that $\phi(l)  \leq t/M$ if and only if $Z_{t,l} \geq 1$. Hence, it suffices to find the minimal $t>0$ such that $Z_{t,l} \geq 1$ a.a.s. as $n \rightarrow +\infty.$ Naturally, this can be accomplished using the first and second moment methods. 

\subsection{The Vanilla First Moment Function As In \cite{iliopoulos2021group}}\label{subsec:VFMF}
Following this perspective, to approximate $\phi(l)$ the authors of \cite{iliopoulos2021group} define an implicit ``first-moment'' equation in $t$,
\[\E[Z_{t,l}] = 1\label{eq:basicFMF}.\]
The motivation for this choice is two-fold. To explain this, let us fix a $l \in \range{k}$.
\begin{itemize}
    \item[(a)] If for some $t_1>0$ it holds that $\E[Z_{t_1,l}] =o(1)$, then by Markov's inequality $Z_{t_1,l}=0$ a.a.s. as \ngrow\!, and therefore $\phi(l) \geq t_1.$ This is customary called the first moment method.

    \item[(b)] On the other hand, if for some $t_2>0$ (ideally relatively ``close'' to $t_1>0$) it holds that $\E[Z_{t_2,l}] =\omega(1)$ and the distribution of $Z_{t_2,l}$ concentrates, for example with $\mathrm{Var}[Z^2_{t_2,l}]=o(\E[Z_{t_2,l}]^2)$, then $Z_{t_2,l} \geq 1$ a.a.s. as \ngrow\!, giving $\phi(l) \leq t_2.$ This is customary called the second moment method.
\end{itemize}Thus, if $t_{\ell}$ is the ``first-moment'' solution for \eqref{eq:basicFMF} with respect to $t$ and one establishes sufficient concentration of $Z_{t,l}$ for $t \approx t_l$, then one could naturally predict that a.a.s. as \ngrow\!, it holds
\[\label{eq:iliop}
    \phi(l) \approx t_{\ell}.
\] Following the first and second moment method strategy, tight approximations such as \eqref{eq:iliop} to their corresponding first moment solutions have already been successfully established for a plethora of similar sparse settings to Bernoulli group testing, including sparse regression \cite{gamarnik2022sparse, chen2024low}, planted clique \cite{gamarnik2024overlap} and sparse tensor PCA \cite{arous2023free, chen2024low}.

Now, \emph{under the assumption of \eqref{eq:iliop}}, \cite{iliopoulos2021group}, analyzed the monotonicity properties of $\phi(\ell)$ via $t_{\ell}$ and concluded that $\phi(\ell)$ should be monotonic, implying that the \bogp never appears.

\subsection{The Conditional First Moment Function}\label{subsec:CFMF}
A crucial contribution of this work is demonstrating that in Bernoulli group testing \eqref{eq:iliop}, as well its conclusion on non-existence of the \bogp\!, are incorrect due to the presence of rare events. Notice that one can consider a variation of the first-moment equation \eqref{eq:basicFMF} under a conditioned event $\scr{A}$, 
\[\E[Z_{t,l}|\scr{A}] = 1. \label{eq:FmfCond}\] 
The key idea is that a ``conditional'' first moment method also holds: if $\scr{A}$ occurs a.a.s. as \\\ngrow, then for any $t'_1>0,$ with $\E[Z_{t'_1,l}|\scr{A}]=o(1)$, it must hold that $\phi(l) \geq t'_1$ a.a.s. as \ngrow\!, with the potential $t'_1$ being much larger than $t_1$ coming from the vanilla first moment method. Albeit a natural idea, no such conditioning has been required in the analysis of similar sparse problems \cite{gamarnik2022sparse, gamarnik2024overlap, arous2023free, chen2024low}. 

Notice that in Bernoulli group testing the degrees of the individuals in Figure \ref{fig:PPBi} are random. Interestingly, in \cite{cojaoghlan2022statistical} the fluctuations of these degrees were shown to be detrimental for directly proving a low-degree lower bound, which led the authors of \cite{cojaoghlan2022statistical} to employ a more involved method using the Franz-Parisi potential \cite{NEURIPS2022_daff6824}. We observe that the degree fluctuations are also detrimental to the vanilla first moment equation, as conditioning on an event that upper bounds these degrees is crucial to get an accurate approximation of $\phi(l)$.

We first define the key conditioning event. 
\begin{lemma}[\cite{cojaoghlan2022statistical}, Section 9.2.1 (arxiv version)]\label{lem:aSet}
Consider an $(\alpha, C)$ instance of group testing. If $a$ is an element of the set
\[\left\{a : \log(2)C(a \log(a) - a + 1) > \frac{\alpha}{1-\alpha}\right\},\label{eq:a:set}\]
then for $$\scr{A} := \{\deg(i) \leq 2aqM, \;\forall i \in \sigma^*\}$$ it holds that $P(\scr{A}) = 1 - o(1).$
\end{lemma}

Using this choice of $\scr{A}$ in equation \eqref{eq:FmfCond}, we denote by $t'_l=t'_l(\scr{A})$ the (conditional now) first moment solution of \eqref{eq:FmfCond} with respect to $t$ given the value of $l\in \range{k}$.

One could aim to solve for $t'_l$ and seek to get a simpler formula for it. Using linearity of expectation, standard concentration of measure asymptotics, and a direct computation with \eqref{eq:FmfCond} (deferred to Section \ref{sec:FMFDerivation}), we indeed get a simpler (but still implicit) set of equations satisfied by a very close proxy to $t'_l$. 

To explain the derived equations, notice that both $t$ and $l$ take values in growing regions, $\range{M}$ and $\range{k}$ respectively. Hence, it is convenient to re-parameterize our setting in terms of the proportional overlap $ \frac{l}{k} = x \in [0,1]$. Moreover, we also denote our proxy for the re-scaled quantity $\frac{t'_l}{M}=\frac{t'_{xk}}{M}$ by $y(x) \in [0,1].$ To define $y(x)$ we first remind the reader the definition of the two point KL divergence from \eqref{eq:two_KL}.
We now define $y(x)$ as follows.
\begin{definition}\label{def:FMF}
    Consider $r(x) \coloneqq  4 \cdot 2^{-x}(1-2^{-x})$, $s(x) \coloneqq 1-2^{x-1}$, with $x\in [0,1]$, $\alpha \in (0,1)$, $C \in (1,2)$, constants $\Co{roomR}, \Co{roomS}, \Co{roomI} > 0$, and $a$ an element of the set \eqref{eq:a:set}.

    For any $x \in [0,1]$ define the $(\Co{roomR},\Co{roomS},\Co{roomI})$-\emph{first moment function} at $x$, denoted by $y=y(x)$ as the solution to the equation,
     \begin{align}\label{FMF}
     \frac{1}{M}\log\left(\binom{k}{\floor{xk}} \binom{p - k}{\floor{(1-x)k}}\right) &=  (1-y)D\left(\frac{2a\log(2)x}{1-y}\|\|r(x)\right)  + D(y||s(x))
     \end{align}
    satisfying the following four constraints,
        \begin{align}\frac{2a\log(2)x}{1-y} &\leq (1-\Co{roomR})r(x)\label{eq:r:constraint}\\
        y &\leq (1-\Co{roomS})s(x)\label{eq:s:constraint}\\
       2a\log(2)x &\leq (1-\Co{roomR})r(x)\label{eq:existConstraint}\\
       D\left(1-\frac{2a\log(2)x}{(1-\Co{roomR})r(x)}\|\|s(x)\right) + \frac{2a\log(2)x}{(1-\Co{roomR})r(x)}D((1-\Co{roomR})r(x)||r(x)) &\leq (1-\Co{roomI})(1-x)(2-C)\log(2)/C\label{eq:uniConstraint}
    \end{align}
    Often we will reference the region of $x$ where $(x,y(x))$ satisfy \const\!, in which there is an implicit choice of $\alpha, C, \Co{roomR}, \Co{roomS}, \Co{roomI}, a$. 
\end{definition}

The definition of the first moment function is unfortunately quite technical. For this reason, we defer explaining the exact relation between $t'_l/M$ and $y(x)$ to Section \ref{sec:FMFDerivation} and proceed with a few high level explanatory remarks.

\begin{remark}\label{rem:exists}
The equation \eqref{FMF} turns out to be equivalent to \eqref{eq:FmfCond} up to lower order terms. This is an outcome of a standard concentration of measure argument on the product Bernoulli distribution that constraints \eqref{eq:r:constraint} and \eqref{eq:s:constraint} allow to be applied. Moreover, under constraints \eqref{eq:r:constraint} and \eqref{eq:s:constraint}, the additional constraints \eqref{eq:existConstraint} and \eqref{eq:uniConstraint} allow us to restrict to values of $x$ that the first moment function $y(x)$ provably exists and is unique. The proof of this fact is given in Section  \ref{sec:uniqueProof}. Moreover, as long as the first moment function exists on an interval, a similar argument allows us to conclude the continuous differentiability of $y(x)$ on the interval (see also Section \ref{sec:uniqueProof}).
\end{remark}

\begin{remark}[Comparison to \cite{iliopoulos2021group}]
    Definition \ref{def:FMF} without $(1-y)D\left(\frac{2a\log(2)x}{1-y}\|\|r(x)\right)$ on the right-hand side of \eqref{FMF}, under the constraint \eqref{eq:s:constraint}, and missing the constraints \eqref{eq:r:constraint}, \eqref{eq:existConstraint}, \eqref{eq:uniConstraint} was also utilized in \cite{iliopoulos2021group} to define their (unconditional) first moment function. The additional term in \eqref{FMF} is due to the conditioning event $\scr{A}$ from Lemma \ref{lem:aSet}. In Figure \ref{fig:FMFComp}, we plot several solutions to our (conditional) first moment function for different values of $C$ and $n$, and compare it with the unconditional first moment function from \cite{iliopoulos2021group}. It is interesting how important the conditioning appears to be; for large finite values of $n$ the unconditional first moment function is monotonic (as established in \cite{iliopoulos2021group}), while the conditional becomes not monotonic (as we prove later in Theorem \ref{lem:FMFnonMono_0}).
\end{remark}
    
\begin{remark}[The role of $\Co{roomR}$, $\Co{roomS}$ and $\Co{roomI}$]
    The introduction of the constants $\Co{roomR}$, $\Co{roomS}$ and $\Co{roomI}$ in the definition is purely for technical convenience. They do not change the value of the solution to $y(x)$ in \eqref{FMF}, they simply slightly restrict the region of $x$ where $(x,y(x))$ is defined to avoid certain degeneracies in our arguments in Section \ref{sec:FMFDerivation}. For this reason, we consider them to be arbitrarily small constants. 
\end{remark}

\begin{remark}\label{rem:integ}
Lastly, we highlight that often in what follows (but not always) we consider the values of $x$ to be restricted on the set $\{0, 1/k, 2/k, \ldots, 1\}$. In those cases, for notational simplicity and when clear from context, we drop the floor function from the binomial coefficients in \eqref{FMF}.
\end{remark}
\begin{figure}
    \centering
    \includegraphics[width=0.7\textwidth]{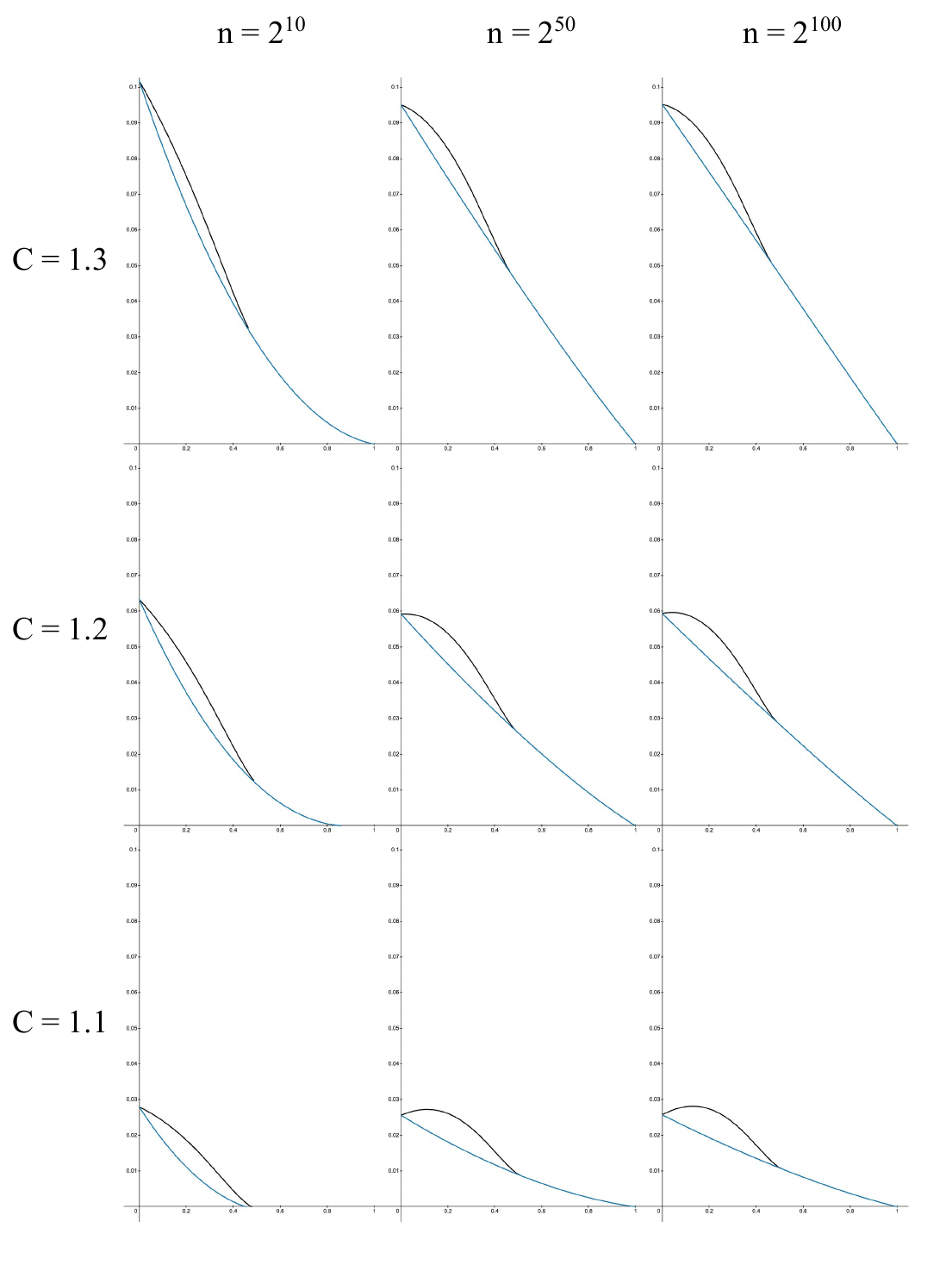}
    \caption{Solutions to two differing first moment functions, \cite{iliopoulos2021group}'s unconditional first moment function in blue and our conditional first moment function in black. These plots were made with parameters $\alpha = .01$, $a = 1.17$ and varying $C$ and $n$. The unconditional first moment function is monotonic for all of our chosen values of $C$ and $n$, confirming the analysis done by \cite{iliopoulos2021group}. Our conditional first moment function is non-monotonic for $C$ sufficiently small and $n$ sufficiently large, confirming our Theorem \ref{lem:FMFnonMono_0}. 
    %(2). The non-monotonicity of our conditional first moment function is highly dependent on the value of $n$, this presents a possible explanation as to why evidence of a \bogp was not found in simulations of Glauber dynamics of group testing conducted in \cite{iliopoulos2021group} as they only tested $n$ less than 10000.
    }
    
    %(3). The point where the conditional first moment function no longer uniquely exists is at the intersection point with \cite{iliopoulos2021group}'s unconditional first moment function. This has a direct probabilistic interpretation, our first moment function stops uniquely existing in the case that constraint \eqref{eq:r:constraint} is violated. This would mean that the Chernoff bound we applied in \eqref{eq:FMF:upperbound} is no longer valid and the corresponding binomial probability can only be trivially upper bounded by one, leading to the same derivation presented in \cite{iliopoulos2021group}.}
    \label{fig:FMFComp}
\end{figure}
\subsection{Local Monotonicity Of A First-Moment Function}\label{subsec:LM}
Recall that our goal is to prove that $\phi(l)$ ($l \in \range{k}$) is non-monotonic for some regime of $\alpha,C$ to conclude the existence of \bogp.  Moreover, as we aim to approximate $\phi(l)$ using the deterministic $y(l/k)$, a natural question is whether $y(l/k)$ is non-monotonic. On top of that, following the plots in Figure \ref{fig:FMFComp}, it is natural to expect that the non monotonicity to take place around $l/k \approx 0.$ Hence, we now focus on whether there exists a region of $x=l/k$ close to $0$ where we can prove the non-monotonicity behavior of $y(x)$.

To answer this question, we first naturally need to guarantee that for some $\epsilon>0$ the first moment function exists for all $x \in [0,\epsilon]$ which, as explained in Remark \ref{rem:exists} it is guaranteed if the constraints \eqref{eq:r:constraint}-\eqref{eq:uniConstraint} are satisfied for all $x \in [0,\epsilon]$.
The following assumption suffices to guarantee this part.
%Recall the definition of $H_C$ from \eqref{def:H_C}in what follows. \textcolor{red}{Plot?}
\begin{assumption}\label{as:FMFExists} We assume that the parameters $(\alpha, C, a, \Co{roomR}, \Co{roomI})$ satisfy
\[D\left(1- \frac{a}{2(1-\Co{roomR})}\big{|}\big{|}\frac{1}{2}\right) \leq (1-\Co{roomI})\frac{2-C}{C}\log(2)\label{eq:FMFExists1}\]
and 
\[\frac{a}{2(1-\Co{roomR})} < 1,\label{eq:FMFExists2}\] 
where $\Co{roomR}, \Co{roomI} > 0$ and $a$ being a valid choice from \eqref{eq:a:set}.
\end{assumption}
Because of the complexity of the assumption, we plot the range of $\alpha$ and $C$ for which Assumption \ref{as:FMFExists} holds in Figure \ref{fig:310}, by setting $a$ and $\Co{roomR}, \Co{roomI}$ to their lowest possible values. It is worth pointing out that the assumption is satisfied for any $1<C<2$ as long as $\alpha>0$ is small enough.

\begin{figure}[h]
    \centering
    \begin{overpic}[scale=.8]{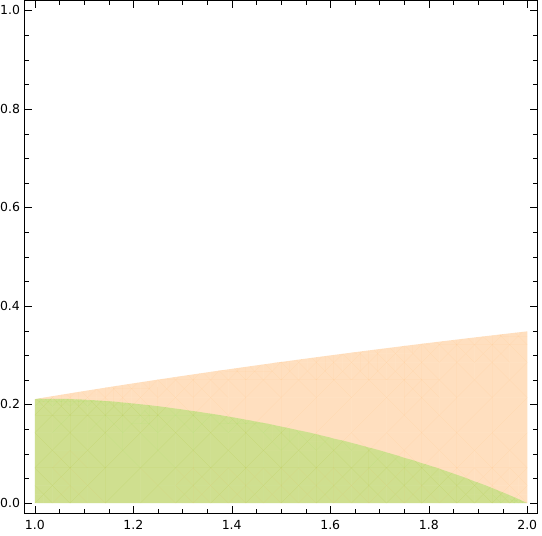}
        \put(49, -5){$C$}
        \put(-5,48){$\alpha$}
    \end{overpic}
    \caption{The green and orange regions in the above plot represent the values of $\alpha$ and $C$ for which conditions \eqref{eq:FMFExists1} and \eqref{eq:FMFExists2} from Assumption \ref{as:FMFExists} are satisfied under the choice of $a$ from the lower boundary of the set \eqref{eq:a:set} and setting $\Co{roomR}, \Co{roomI} = 0$. Note that the region in green is a subset of the region in orange.}\label{fig:310}
\end{figure}

Under Assumption \ref{as:FMFExists}, we have the following result.
\begin{lemma}\label{lem:FMFnear0}
  If $(\alpha, C, a, \Co{roomR}, \Co{roomI})$ satisfy Assumption \ref{as:FMFExists}, then there exists an $\epsilon > 0$ such that the first moment function $y(x)$ according to Definition \ref{def:FMF} for  $x \in [0,\epsilon]$ exists and is unique  a.a.s. as \ngrow\! (with respect to the randomness of $p,M$). Moreover, $y(x)$ is continuous and differentiable over $[0,\epsilon]$.
\end{lemma}The proof of this result is given in Section \ref{sec:uniqueProof}.

Now that we have established that the first moment function exists and is unique around zero, we also make the following assumption on our parameters which allows us to conclude the desired monotonicity of the first moment function at $0.$

\begin{assumption}\label{as:der_0}
    Recall $H_C$ from Definition \ref{def:H_C}. We assume that $(\alpha,C, a)$ satisfies
    \[C < \frac{1-\frac{\alpha}{1-\alpha}}{a\left(1 - \log\left(\frac{a}{2(1-H_C)}\right)\right) + H_C - 1},\]
    and that $a$ is a valid choice from \eqref{eq:a:set}.
\end{assumption}

This cumbersome assumption appears quite naturally by calculating the discrete derivative of $y(l/k)$ around $l/k \approx 0$ and checking when it is strictly positive (See Section \ref{sec:monotone}). Given a pair $(\alpha,C)$, if one chooses $a$ to be the lowest feasible value from \eqref{eq:a:set}, then the pairs $(\alpha,C)$ that satisfy this assumption are given in Figure \ref{fig:JustDerCond}. In particular, we highlight that the condition is valid for all $0<C<C^* \approx 1.4749$ for $\alpha>0$ sufficiently small.

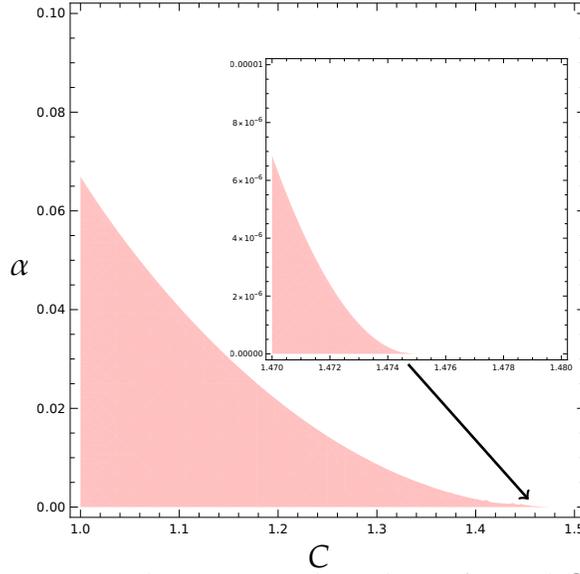
\begin{figure}[ht]
    \centering
    \begin{overpic}[scale = .8]{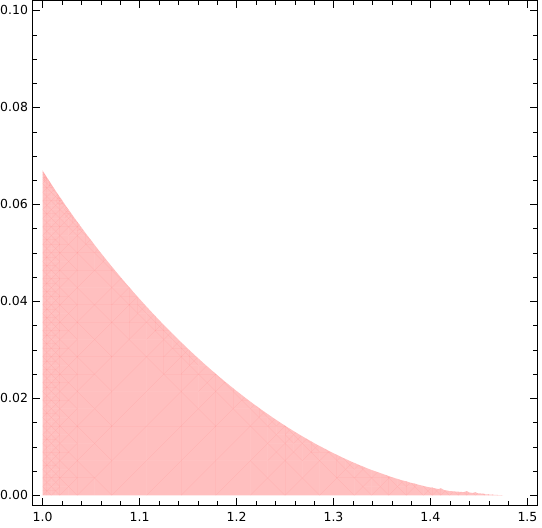}
        \put(49, -5){$C$}
        \put(-5,48){$\alpha$}
        \put(35,30){\includegraphics[scale = .5]{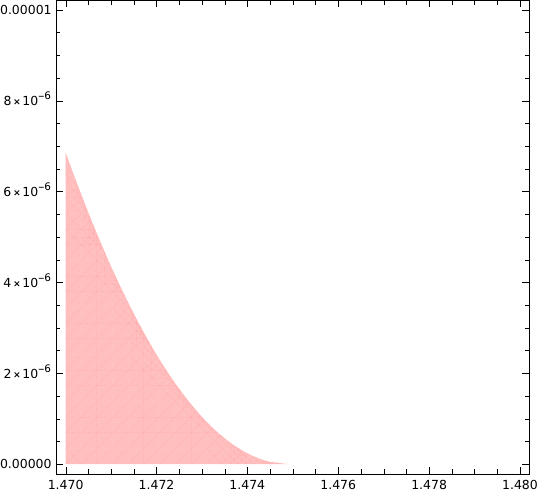}}
        \put(67,7){\tikz \draw[->, line width = 1pt] (0,0) -- (1.6,-1.8);}
    \end{overpic}
    \caption{The region in red represents the values of $\alpha$ and $C$ for which Assumption \ref{as:der_0} holds when choosing of $a$ from the lower boundary of the set \eqref{eq:a:set}.}
    \label{fig:JustDerCond}
\end{figure}

Now, under the above assumptions we prove that indeed the first moment function must increase near $0.$

\begin{theorem}\label{lem:FMFnonMono_0}
  If the parameters $(\alpha, C, a, \Co{roomR}, \Co{roomI})$ satisfy Assumption \ref{as:FMFExists} and Assumption \ref{as:der_0}, then, a.a.s. as $n \rightarrow +\infty$ (with respect to the randomness of $p,M$), there exist constants $\epsilon_1 > 0$ and $\delta_1 > 0$ such that for all $0 \leq l \leq \epsilon_1 k$ it holds
  \[y(l/k) - y(0) \geq \delta_1 l/k.\] 
\end{theorem}The proof of the theorem is deferred to Section \ref{sec:monotone}.

\subsection{Local Monotonicity Of $\phi(l)$ Via First Moment Function Approximations}\label{subsec:LMFMF}

From Theorem \ref{lem:FMFnonMono_0}, we know that $y(l/k)$ increases for all $l \leq \epsilon k$ for some small $\epsilon>0$. We now investigate whether $\phi$ inherits this monotonic increase near zero from the first moment's functions behavior. To establish this, it suffices to show that $y(l/k) - o(1)$ a.a.s. lower bounds $\phi(l)$ over the region $l/k \in [0,\epsilon]$ and demonstrate an equivalent $y(0) + o(1)$ a.a.s. upper bound for $\phi(0)$.

Similar to the above result on the first moment function, the following result on $\phi(l)$ is subject to a few parameter assumptions. This assumption is again rather cumbersome, an outcome of an involved second moment method argument that leverages it. Crucially, however, this assumption is satisfied for all $1<C<2$ when $\alpha$ is sufficiently small (see Figure \ref{fig:PhiCond}). We also direct the reader to Section \ref{subsec:qual} for more details on this assumption.

\begin{assumption}\label{as:alphaC_0}
    The pair of parameters $(\alpha, C)$ satisfy $\alpha<28/1000$ and 
    \[C < 2\frac{1-2\alpha}{1-\alpha}\label{eq:alpha:stirling_0}.\]
    Moreover, the pair satisfies the following two conditions with $H_C$ from Definition \ref{def:H_C},
    \[\hspace{-1cm}\label{eq:2mmCond1_0}
    C\bigg{[}(1-H_C)(1-\log(2(1-H_C))) - \frac{h_2(H_C)}{2}-7\sqrt{\frac{\alpha}{1-\alpha}}\left(\frac{1}{2}\log(2(1-H_C)) \right)\bigg{]} > 4 \alpha/(1-\alpha)\]
    and 
    \[C\left[\frac{h_2(H_c)}{2} + \frac{1}{2}\log\left(\frac{1-H_C}{H_C} \right)\left( 1-H_C - 5\sqrt{\frac{\alpha}{1-\alpha}}\right) +H_C -1 \right] > 3 \alpha / (1-\alpha).\label{eq:2mmCond2_0}\]
\end{assumption}

Using this assumption we can then get our desired bounds on $\phi(l)$.

\begin{theorem}\label{thm:combinedBound}
    If the parameters $(\alpha, C, a, \Co{roomR}, \Co{roomI})$ satisfy Assumption \ref{as:FMFExists} and Assumption \ref{as:alphaC_0}, then there exists an $\epsilon' > 0$ such that, for all $x = l/k \in [0,\epsilon']$, we have a.a.s. as $n \conv{} + \infty$ that,
    \[\phi(l) \geq y(l/k) - O(1/k).\]
 Moreover, a.a.s. as $n \conv{} + \infty$,
    \[\phi(0) = y(0) + o(1) = H_C + o(1).\]
\end{theorem}

This result combines an a.a.s, as \ngrow lower bound on $\phi(l)$ for all $l=0,1,\ldots,k$ as well as an a.a.s, as \ngrow\! upper bound on $\phi(0)$, both of which are shown in  Section \ref{sec:combProof} and Section \ref{subsec:Up}. The former relies on a relatively straightforward application of a conditional first moment method. The latter part is highly non-trivial to prove. We prove it via an elaborate conditional second moment method and is far more technical due to the necessity for delicate control over shared positive tests between two non-infected individuals. In particular, obtaining our result for $\phi(0)$ amounts to a very tight understanding of the so-called random set cover model, a connection we describe in Section \ref{subsec:MAX}.

%\begin{remark}\label{rem:as_3.13} 
 %   We can immediately see for any $C \in (1,2)$ that we can choose $\alpha$ sufficiently close to $0$ such that the first condition in Assumption \ref{as:alphaC_0} is satisfied. The other conditions in Assumption \ref{as:alphaC_0} are more complex, so we instead provide a visualization for what values of $C$ and $\alpha$ do these two conditions holds, see Figure \ref{fig:PhiCond}. This visualization demonstrates that for all $C \in (1,2)$ there still exists a sufficiently small value of $\alpha \in (0,1)$ such that both conditions \eqref{eq:2mmCond1_0} and \eqref{eq:2mmCond2_0} are satisfied.
%\end{remark}

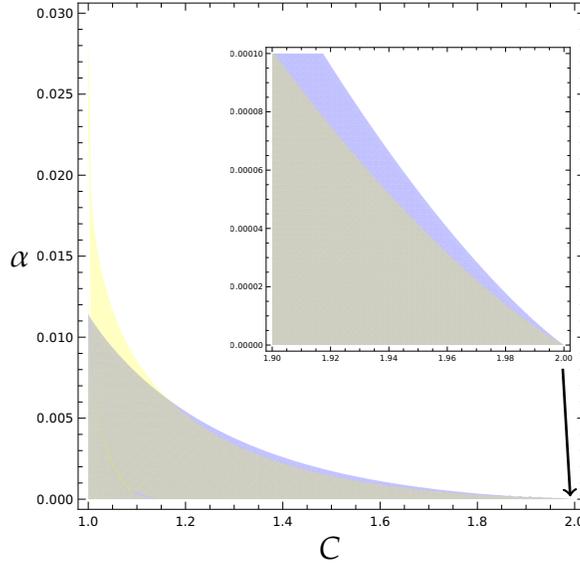
\begin{figure}[ht]
    \centering
    \begin{overpic}[scale = .8]{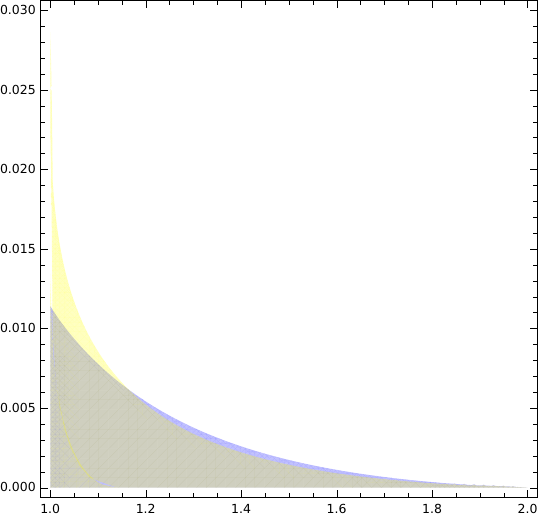}
        \put(51, -5){$C$}
        \put(-5,48){$\alpha$}
        \put(35,30){\includegraphics[scale = .5]{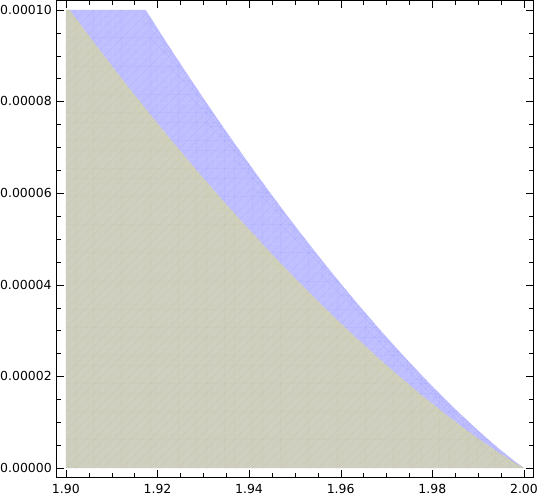}}
        \put(95,6){\tikz \draw[->, line width = 1pt] (-.1,0) -- (0,-1.7);}
    \end{overpic}
    \caption{A visual representation for when Assumption \ref{as:alphaC_0} holds. The $x$-axis represents the value of $C$ and the $y$-axis represents the value of $\alpha$. The blue region contains the values for which the condition \eqref{eq:2mmCond1_0} holds, and the yellow region contains the values for which the condition \eqref{eq:2mmCond2_0} holds. The intersection of both colors represents the region where both conditions are satisfied.}
    \label{fig:PhiCond}
\end{figure}

\subsection{\bogp and MCMC Failure In Bernoulli Group Testing}\label{subsec:MCMC}

Combining Theorem \ref{thm:combinedBound} with Theorem \ref{lem:FMFnonMono_0} lets us directly conclude that $\phi(l)$ is increasing for small $l/k$. Moreover, notice that $\phi(k)=0$ by the definition of $\sigma^*$. Combining this fact with Theorems \ref{lem:FMFnonMono_0} and \ref{thm:combinedBound}, with $\alpha$ and $C$ satisfying Assumptions \ref{as:FMFExists}, \ref{as:der_0}, \ref{as:alphaC_0}, we can conclude that $\phi(\ell)$ is non-monotonic and in particular, using standard arguments in the literature, that \bogp appears.
\begin{theorem}\label{thm:QualInc_0}
    For an $(\alpha, C)$ instance of group testing, a valid choice of $a$ from \eqref{eq:a:set} and arbitrarily small $\Co{roomR}, \Co{roomI} > 0$ satisfy Assumptions \ref{as:FMFExists}, \ref{as:der_0}, \ref{as:alphaC_0}, then there exists $\delta > 0$ and $0 < \epsilon_1 < \epsilon_2$ such that for all $l$ with $l/k \in [\epsilon_1, \epsilon_2]$, we have a.a.s. as \ngrow\!, $\phi(l) - \phi(0) \geq \delta$. In particular, as $\phi(k) = 0$, \bogp holds in this regime.
\end{theorem}
The proof of this result is deferred to Section \ref{sec:q}.

%\begin{figure}[ht]
 %   \centering
    %\includegraphics[scale = .7]{figs/Final Picture.pdf}
    %\caption{\textcolor{red}{Unsure where to place that}The cartoon landscape of group testing when $(\alpha, C)$ satisfy Assumptions \ref{as:der_0} and \ref{as:alphaC_0}, assuming that the first moment function uniquely exists for all $x \in [0,1]$. We have also provided where each part of this sketch was derived in the Appendix. The overlap is given by $x$ and the corresponding minimizer of $H$ for $x = l/k$ overlap is given by $y$.}
    %\label{fig:finalResult}
%\end{figure}

Using now also standard bottleneck arguments in the literature \cite{gamarnik2024overlap, LevinPeresWilmer2006}, we conclude via the existence of \bogp that all local MCMC methods sampling from $\pi_{\beta}$ for $\beta$ large enough, take a super-polynomial time to recover $\sigma^*.$ This result is formally described in the following theorem and is the main contribution of this work, answering the main question of \cite{iliopoulos2021group}.

 \begin{corollary}\label{thm:mix} 
 For an $(\alpha, C)$ instance of group testing, a valid choice of $a$ from \eqref{eq:a:set} and arbitrarily small $\Co{roomR}, \Co{roomI} > 0$ satisfy Assumptions \ref{as:FMFExists}, \ref{as:der_0}, \ref{as:alphaC_0}, then there exists $\epsilon_1, \epsilon_2  \in (0,1) $ with $\epsilon_1 < \epsilon_2$ and an $\epsilon_1$ dependent constant $C_\epsilon>0$ such that if $\beta \geq C_\epsilon k \log(p/k)$ the following holds a.a.s. as $n \rightarrow +\infty.$ 
 
For any local Markov chain on the Johnson graph with stationary distribution $\pi_\beta$, there exists an initialization for which the Markov chain requires at least $\exp(\Omega(k \log(p/k)))$ iterations to reach any $k$-subset $\sigma$ with $| \sigma \cap \sigma^* | \geq \epsilon_2 k.$ 
\end{corollary}The proof of the corollary is deferred to Section \ref{sec:mix_pf}.

\subsection{Numerics for the critical $C \approx 1.47491$ when $\alpha \approx 0$}\label{sec:num}

Both of our main theorem \ref{thm:QualInc_0} and Corollary \ref{thm:mix} holds under the technical assumptions \ref{as:FMFExists}, \ref{as:der_0}, \ref{as:alphaC_0}. We here combine our numerical results also presented in Figure \ref{fig:310}, Figure \ref{fig:JustDerCond}, Figure \ref{fig:PhiCond}, to describe the region of $(\alpha, C) \subset [0,1] \times [1,2]$ that satisfy all of them. To do so, we set the two ``slack'' constants in the definition of the first moment function equal to zero; $\Co{roomR}, \Co{roomS} = 0$. For any condition involving $a$ we choose it from the infimum of \eqref{eq:a:set}, that is 
\[a = \inf\left\{a : \log(2)C(a \log(a) - a + 1) > \frac{\alpha}{1-\alpha}\right\}.\]

Given the above, the region for where Theorem \ref{thm:QualInc_0} and Corollary \ref{thm:mix} hold is shown in Figure \ref{fig:conditionrange}. We can see in this visualization that when $\alpha$ is sufficiently close to zero, Assumption \ref{as:der_0} is the first condition which is violated as $C$ grows. We numerically found where this condition is violated, i.e. solving
\[C = \frac{1}{a\left(1 - \log\left(\frac{1}{2(1-H_C)}\right)\right) - 1},\]
which has the solution $C \approx 1.47491$. This calculation justifies our discussion in the introduction and specifically informal Theorem \ref{thm:1.4749} and Corollary \ref{cor:mcmc_inf}.
\begin{figure}[ht]
        \centering
        \begin{overpic}[scale = .8]{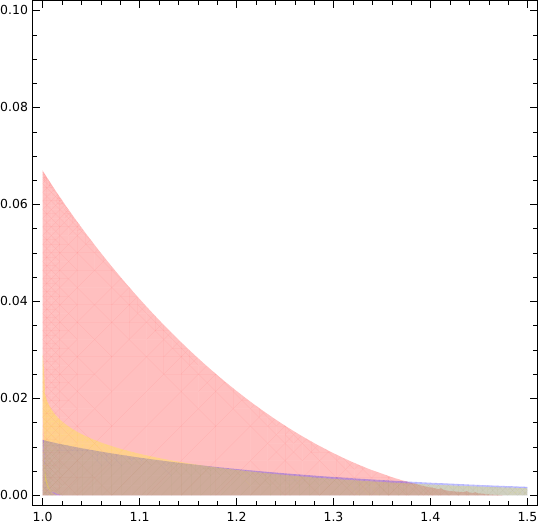}
            \put(49, -5){$C$}
            \put(-5,48){$\alpha$}
            \put(6,31.5){\tikz \draw[dashed, line width = 1pt] (0,0) -- (6.9,0);}
            \put(65, 34){$\alpha < 28/1000$}
        \end{overpic}
        \caption{A plot of the values for $\alpha$ (on the $y$-axis) and $C$ (on the $x$-axis), such that an $(\alpha, C)$ group testing instance satisfied the conditions of Theorem \ref{thm:QualInc_0}. Note that we have only plotted the three most restrictive conditions: Blue corresponding to \eqref{eq:2mmCond1_0}, yellow corresponding with \eqref{eq:2mmCond2_0} and red corresponding to Assumption \ref{as:der_0}.}
        \label{fig:conditionrange}
\end{figure}
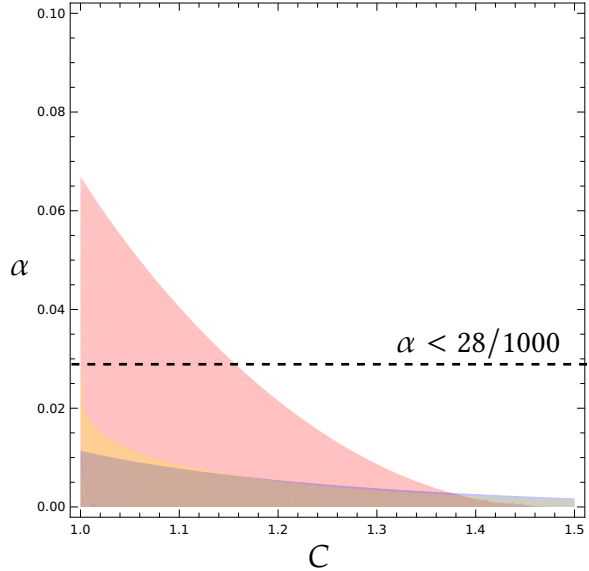

\section{The Thresholds For Random MAX-Set Cover}\label{subsec:MAX}

As we mentioned in the previous section, towards proving the upper bound on $\phi(0)$ for Theorem \ref{thm:combinedBound}, we establish a result on the random MAX-set cover problem, which could be of independent interest. Notice that $\phi(0)$ is about maximizing the number of covered positive tests over all $k$-subsets $\sigma$ disjoint from $\sigma^*.$ In particular, $\phi(0)$ has no relation with $\sigma^*$ and its distribution can be interestingly (and independently from Bernoulli Group Testing) also defined as the following random set cover setting.

 Let a growing parameter $n \rightarrow +\infty.$ Consider a universe of $\pSet=\pSet_n$ elements and $k=k_n \in \mathbb{N}, k \leq \mathcal{P}$. We then choose $\mSet=\mSet_n$ random sets $\mathcal{S}_1, \dots, \mathcal{S}_{\mSet}$ where each set contains each element of the universe with probability $q=q_n$ in an i.i.d. fashion, where $q$ is the solution to $(1-q)^k=1/2$. We then consider the  random MAX-set-cover question on $\{\mathcal{S}_i\}_{i \in [\mSet]}$: what is the maximum fraction of the sets $\{\mathcal{S}_i\}_{i \in [\mSet]}$ that can be covered by some $k$-subset? Recall that this random fraction is defined in \eqref{max_sat} as $\Phi_k$, where we simply replace for this section $M$ by $\mSet$ and $p$ by $\pSet$. 

\begin{remark}
    As explained in Section \ref{sec:intro_cover} observe that in this setting simply outputting a $k$-subset covers approximately a $1-(1-q)^k+o(1)=1/2+o(1)$ fraction of the $\mSet$ sets, a.a.s. as $n \rightarrow +\infty$ over the randomness of $\{\mathcal{S}_i\}_{i \in [\mSet]}$. Thus studying $\Phi_k$ investigates how much better the optimal $k$-subset performs as opposed to the trivial ``random guessing'' method.
\end{remark}

 Comparing with Bernoulli group testing, notice that by setting $k=n^{\alpha+o(1)}$, $\mSet$ equal to the (random) number of positive tests in Bernoulli group testing and $\pSet$ equal to the (random) number of possible infected individuals, $1-\Phi_k$ equals in distribution to $\phi(0).$ Hence, proving the second part of Theorem \ref{thm:combinedBound} amounts to proving under an appropriate scaling of the parameters that it holds $\Phi_k=1-H_C+o(1)$, where $H_C$ is defined in \eqref{def:H_C}.

For this reason, we focus on the scaling on $k$, $\mSet$ and $\pSet$ as functions of $n$ that aligns with the asymptotic scaling of the number of infected individuals, number positive tests, and number of possible infected individuals as motivated by Bernoulli group testing (see Lemmas \ref{lem:M:limit} and \ref{lem:p:limit} for the concentration properties of the last two). The exact scaling is described in the following assumption.

\begin{assumption}\label{as:MPScale}
  For some constants $\alpha>0$, $C>1$ let $k=\floor{n^{\alpha}}$ and $N = \floor{C \log_2\binom{n}{k}}$. We assume that $(\mSet,\pSet)$ satisfy the following constraints, for some $1/4 > c > 0$, as $n \conv{} \infty$,
    \[(1-N^{-\err})N/2 \leq \mSet \leq (1 + N^{-\err})N/2,\]
    \[(1-k^{-\err})n\left( \frac{k}{n} \right)^{\frac{C}{2}(1+k^{-\err})} \leq \pSet \leq (1+k^{-\err})n\left( \frac{k}{n} \right)^{\frac{C}{2}(1-k^{-\err})}.\]
\end{assumption}

Under this assumption, we prove the following result.
\begin{theorem}\label{thm:Phi_kLim}
    Assume $1 < C < 2$, the pair $(\alpha, C)$ satisfy Assumption \ref{as:alphaC_0} and constraints \eqref{eq:r:constraint}-\eqref{eq:uniConstraint} are satisfied at $x = 0$ and the triplet $(k,\mSet, \pSet)$ satisfy Assumption \ref{as:MPScale}, then
    \begin{align}\label{eq:phi_k}
        \lim_{n \rightarrow +\infty} \Phi_k = 1 - h^{-1}_2\left(2 - 2/C\right)
    \end{align}
    a.a.s. as \ngrow\!.
\end{theorem}The proof of the theorem is deferred to Sections \ref{sec:combProof} and Section \ref{subsec:Up}.

\begin{remark}
    We remind the reader that Assumption \ref{as:alphaC_0} is (numerically) observed to be satisfied for any $1<C<2$ as long as $\alpha \in (0,1)$ is sufficiently small. Moreover, we highlight the interesting fact that the asymptotic value of $\Phi_k$ is independent of $\alpha>0$ but only depends on $C>1$.
\end{remark}

\begin{remark}\label{rem:gen}
    One might be eager to generalize Theorem \ref{thm:Phi_kLim} to a scaling independent of Bernoulli group testing, and identify the limiting $\Phi_k$ for any constants $D>0, \beta \in (0,1)$ when $\mSet = \floor{D \log_2 \binom{\pSet}{k}}$ and $k = \floor{\pSet^{\beta}}$. Albeit we do not attempt this generalization in the present work, we expect that, in this general setting, as long as $\beta$ is small enough, for any $D>1$ the limiting value of $\Phi_k$ would be equal to $1-h_2^{-1}(1-1/D).$ We believe that an appropriate modification of our conditional second moment method, via similar key flatness ideas \cite{balister2018dense}, can establish this more general result and we leave this as an interesting direction for future work.
\end{remark}

\section{Key Technical Results on the First Moment Function}\label{sec:FMFDerivation} 

\subsection{Deriving the Form of the First Moment Function}
Consider the post-processing discussed in Section \ref{sec:post_proc_intro}, where we prune all negative tests and consider the set of $M$ positive tests and $p$ possibly infected individuals (specifically, individuals who do not participate in a negative test). To derive the first moment function given in Definition \ref{def:FMF}, we calculate the expectation in \eqref{eq:FmfCond} under this induced post-processed probability measure $\P$. 

Throughout this derivation we use the following short-hands. For any test $b$, set $\sigma$ and $s \in [M]$:
\begin{itemize}
    \item[(a)] $E^{\leq s}(\sigma)$ represents the event where the number of positive tests left uncovered by $\sigma$ is less than or equal to $s$. Similarly $E^{s}(\sigma)$ is the event that $\sigma$ leaves exactly $s$ positive tests uncovered.
    \item[(b)] $E_b(\sigma)$ refers to the event that $\sigma$ covers the test $b$. 
    \item[(c)] We define $Z^\sigma_{t,l}$ to be the following indicator random variable:
        \[Z^\sigma_{t,l} = I\{|\sigma \cap \sigma^*| = l, \sigma \text{ leaves at most }t\text{ positive tests uncovered}\}.\]
\end{itemize}
 Finally, we also consider the ``null model" $\Q$ defined in the following way. For $p$ individuals and $M$ positive tests, in an i.i.d. fashion we let each of the $p$ elements take part in each of the $M$ positive tests with probability $q.$ From a graph theoretic viewpoint, $\Q$ models an Erd\"os-Renyi bipartite graph between $p$ and $M$ nodes with connection probability $q$. Moreover, one can directlycompute the likelihood ratio between $\P$ and $\Q$. This can be shown by conditioning $\Q$ on the event $E^0(\sigma^*)$, the ``planting" of the true signal into the null model, and using that any size $k$ subset will cover a test with probability $1/2$, which $\sigma^*$ must be doing $M$ times. This argument gives for any event $E$,
\[
    \P(E) = \Q(E | E^0(\sigma^*)) = \frac{\Q(E \cap E^0(\sigma^*))}{\Q(E^0(\sigma^*))} = 2^M \Q(E \cap E^0(\sigma^*)).\label{eq:PQRelation}
\]

Now for any $k$ subset $\sigma,$ with $ |\sigma \cap \sigma^*|=l$, let us define $\alpha_s \coloneqq \Q(E^0(\sigma^*) \cap E^s(\sigma) )$ and $\beta_s \coloneqq \Q(\scr{A}|E^0(\sigma^*) \cap E^{s}(\sigma) )$. 
Using $\P({\scr{A}}) = 1-o(1)$ in line \eqref{eq:usedForDer1} and \eqref{eq:PQRelation} in line \eqref{eq:condcalc2}, we calculate the conditional expectation of $Z_{t,l}$ (from Definition \ref{def:Ztl}) given $\scr{A}$ as,
\begin{align}
	\E[Z_{t,l}|\scr{A}] &= \binom{k}{l} \binom{p - k}{k-l} \P(Z^\sigma_{t,l}|\scr{A}) = \binom{k}{l} \binom{p - k}{k-l} \frac{\P(Z^\sigma_{t,l}\cap \scr{A})}{\P(\scr{A})} \\
 &= (1+o(1)) \binom{k}{l} \binom{p - k}{k-l} \P(Z^\sigma_{t,l}\cap \scr{A})\label{eq:usedForDer1}\\
	&= (1+o(1))\binom{k}{l} \binom{p - k}{k-l}  2^M\sum_{s = 0}^t \Q(E^0(\sigma^*) \cap E^{s}(\sigma) \cap \scr{A})\label{eq:condcalc2}\\
 &= (1+o(1))\binom{k}{l} \binom{p - k}{k-l}  2^M\sum_{s = 0}^t \alpha_s \beta_s.\label{eq:origionalEbound}
\end{align}

We now notice that $\alpha_s = \binom{M}{s}\Q((E_b(\sigma))^C \cap E_b(\sigma^*))^s\Q(E_b(\sigma) \cap E_b(\sigma^*))^{M-s}$. Moreover, the probability all the $k$ elements in $\sigma$ miss test $b$ and the $k-l$ elements in $\sigma^* \setminus \sigma$ to cover~test~$b$~equals
\[
	 \Q((E_b(\sigma))^C \cap E_b(\sigma^*)) = (1-q)^k(1-(1-q)^{k-l}) = \frac{1}{2}\left(1-2^{-\frac{k-l}{k}}\right).\label{eq:probcalc}
\]

Recalling $|\sigma \cap \sigma^*| = l$, we find that $\Q(E_b(\sigma \cap \sigma^*)) = 1- (1-q)^l= 1-2^{-l/k}$ and thus,
\[
    \Q(E_b(\sigma) \cap E_b(\sigma^*)) = (1-2^{-l/k}) + 2^{-l/k}(1-2^{-(k-l)/k})^2= 1-2\cdot 2^{-l/k - (1-l/k)} + 2^{-l/k - 2(1-l/k)} = 2^{l/k - 2}.
\]
Using the above two equalities, we have
\[
    \alpha_s = \binom{M}{s} \left(\frac{1}{2}\left(1-2^{-\frac{k-l}{k}}\right)\right)^s \left(2^{l/k - 2}\right)^{M-s}.\label{eq:alphaBound}
\]
Recall that $\scr{A}$ is the event where every element in $\sigma^*$ covers at most $d$ positive tests, for some $d \in \{0, \dots, M\}$.
If we consider any $\sigma$ where $|\sigma \cap \sigma^*| = l$, then $\scr{A}$ implies that the total number of positive tests covered by the elements in $\sigma^* \cap \sigma$ is bounded by $ld$.
Under $E^0(\sigma^*) \cap E^s(\sigma)$, the number of positive tests that could be potentially covered by the intersection $\sigma \cap \sigma^*$ is $M-s$ (since $\sigma$ misses $s$ positive tests). Hence, for $r(l/k) = 4 \cdot 2^{-l/k}(1-2^{-l/k})$ and $B'_s \sim \text{Binomial}(M-s, r(l/k))$ we have
\[\beta_s \leq \P(B'_s \leq ld).\label{eq:betaBound}\]
Indeed, $r(l/k)$ is the probability that a given test $b$ contains at least one element in the intersection $|\sigma \cap \sigma^*| = l$ conditioned on the event that the test $b$ contains an element in both $\sigma$ and $\sigma^*$ respectively. Its derivation is given in \cite[Section 9.2.2 (Arxiv version)]{cojaoghlan2022statistical}.

Returning to the expected value calculation \eqref{eq:origionalEbound}, define two random variables $B_1\sim\text{Binomial}(M-t,r(l/k))$ and $B_2 \sim \text{Binomial}(M,s(l/k))$, with $s(l/k) = 1-2^{l/k-1}$. Plugging in our bounds for $\alpha_s$ and $\beta_s$ from \eqref{eq:alphaBound} and \eqref{eq:betaBound} respectively, recalling that a $\Bi{n,p}$ random variable is stochastically upper bounded by a $\Bi{m,p}$ random variables when $n \leq m$ in line \eqref{eq:stocDom}, and a Chernoff bound on binomial random variables in line \eqref{eq:FMF:upperbound} (see Lemma \ref{lem:KLBinomial}), we have

\begin{align}
	\hspace{-0.75cm}(1-o(1))\E[Z_{t,l}|\scr{A}] &\leq  \binom{k}{l} \binom{p - k}{k-l} 2^M \sum_{s = 0}^t\binom{M}{s}\left[\frac{1}{2}\left(1-2^{-\frac{k-l}{k}}\right)\right]^s [2^{l/k - 2}]^{M-s}\P(B'_s \leq ld)\\
	&\leq  \binom{k}{l} \binom{p - k}{k-l}\sum_{s = 0}^t\binom{M}{s}(1-2^{\frac{l}{k}-1})^s (2^{l/k - 1})^{M-s}\P(B_1 \leq ld)\label{eq:stocDom}\\
	&=  \binom{k}{l} \binom{p - k}{k-l} \P(B_2 \leq ld) \P(B_1\leq t)\\
	&\leq \binom{k}{l} \binom{p - k}{k-l} \exp\left({-(M-t)D\left(\frac{ld}{M-t}\|\|r\left(\frac{l}{k} \right)\right)} {-MD\left(t/M\|\|s\left(\frac{l}{k}\right)\right)}\right).\label{eq:FMF:upperbound}
\end{align}

Note that the constraints \eqref{eq:r:constraint} and \eqref{eq:s:constraint} allow for the application of the Chernoff bounds in \eqref{eq:FMF:upperbound} to be valid. Given \eqref{eq:FMF:upperbound}, we define the first moment solution to be the solution to the following implicit equation of $t$,
\begin{equation}
    1 = \binom{k}{l} \binom{p - k}{k-l} \exp\left({-(M-t)D\left(\frac{ld}{M-t}\|\|r(l/k)\right)} {-MD\left(t/M||s(l/k)\right)}\right).
\end{equation}
Taking the logarithm of both sides and substituting $l = xk, x \in [0,1]$, $t = yM, y \in [0,1]$, we get
\[0 = \log\left(\binom{k}{\floor{xk}} \binom{p - k}{\floor{(1-x)k}}\right) - M(1-y)D\left(\frac{xkd}{M(1-y)}\|\|r(x)\right) - MD(y||s(x)).\]
Now replacing $d$ with $2aqM$ and rearranging, we get the equation given in Definition \ref{def:FMF}. The above derivation elicits a convenient upper bound for the conditional expectation of $Z_{t,l}$ given $\scr{A}$. This upper bound will come in handy when we apply the first moment method in Section \ref{subsec:Lower}. For this reason, we state this result here as a proposition.

\begin{proposition}\label{cor:upperBound}
    For all values of $x$ where $(x,y(x))$ satisfy \const\!, 

    \[\E[Z_{t,l}|\scr{A}] \leq (1+o(1))\binom{k}{l} \binom{p - k}{k-l} e^{-(M-t)D(\frac{ld}{M-t}||r(l/k))} e^{-MD(t/M||s(l/k))}\]
\end{proposition}

\begin{proof}
    See the aligned equation ending in line \eqref{eq:FMF:upperbound}.
\end{proof}

\subsection{Existence, Uniqueness and Differentiability of the First Moment Function}\label{sec:uniqueProof}

Below we will provide justification as to why Assumption \ref{as:FMFExists} is a sufficient condition for the existence of the first moment function $y(x)$ in some small interval $x \in [0,\epsilon]$ with $\epsilon > 0$. 

\begin{remark}\label{rem:Scale}We will see in the following proofs that we make these claims under an a.a.s. as \ngrow guarantee. An alternative to this argument would be to assume the a.a.s. events on $M,p$ described in Lemmas \ref{lem:p:limit}, \ref{lem:M:limit}, (and also in the context of \kset in Assumption \ref{as:MPScale}) and directly establish these proofs for any such deterministic $M,p$. \end{remark}

\begin{remark}
    Note that, in slight contrast to Remark \ref{rem:integ}, for this subsection (Section \ref{sec:uniqueProof}), it is essential to consider $x \in [0,1]$ which is no longer constrained in the set of $\{0,1/k, 2/k, \dots, 1-1/k,1\}$. For this reason, we include back the floor symbols for the combinatorial terms in the definition of the first moment function.
\end{remark}
We start with a general lemma establishing the existence of the first moment function.
\begin{lemma}\label{lem:FMFUnique}
    Given $a \in (1,2)$ from \eqref{eq:a:set}, let $\scr{X} \subseteq [0,1]$ be a region where constraints \eqref{eq:existConstraint} and \eqref{eq:uniConstraint} hold. Then, there exists a solution $y(x)$ to the equation, 
    \[\frac{1}{M}\log\left(\binom{k}{\floor{xk}}\binom{p-k}{\floor{(1-x)k}} \right) = (1-y(x))D\left(\frac{2a\log(2)x}{1-y(x)}\|\|r(x) \right) + D(y(x)||s(x))\]
    satisfying constraints \eqref{eq:r:constraint} and \eqref{eq:s:constraint} for all $x \in \scr{X}$  a.a.s. as $n \rightarrow +\infty$. Moreover, this solution is unique.
\end{lemma}

\begin{proof}[Proof of Lemma \ref{lem:FMFUnique}]

    Consider $x \in \scr{X}$ and define $F_0(x,y) = (1-y)D\left(\frac{2a\log(2)x}{1-y}||r(x)\right) + D(y||s(x))$. 
    As $1<C<2$, we invoke \cite[Lemma 36]{iliopoulos2021group} to obtain a upper bound a.a.s as $n \rightarrow +\infty$ of the form,
    \[\frac{1}{M}\log\left(\binom{k}{{\floor{xk}}} \binom{p - k}{{\floor{(1-x)k}}}\right) \leq (1-x)\log(2).\]
    By the non-negativity of KL divergence, $s(0) = 1/2$ and that $\log(1/s(x))$ is increasing for $x\in [0,1)$, we can conclude that for any $x \in [0,1)$,
    \[F_0(x,0) = D(2a\log(2)x||r(x)) + \log\left(\frac{1}{s(x)}\right) \geq \log(2) \geq (1-x)\log(2).\]
    Thus, for $y_{x,0} := 0$ it holds $F_0(x,y_{x,0}) \geq \frac{1}{M}\log\left(\binom{k}{{\floor{xk}}} \binom{p - k}{{\floor{(1-x)k}}}\right)$ a.a.s. as \ngrow\!.
    Moreover, for $y = 0$ we see that constraint \eqref{eq:r:constraint} becomes $2a\log(2)x \leq (1-\Co{roomR})r(x)$, which must be satisfied as $x \in \scr{X}$ satisfies \eqref{eq:existConstraint}. Constraint \eqref{eq:s:constraint} becomes $0 \leq (1-\Co{roomS})s(x)$, which trivially holds. Thus, $y_{x,0} = 0$ also satisfies \eqref{eq:r:constraint} and \eqref{eq:s:constraint}.

    For any $\epsilon_1 > 0$, a combination of $\binom{a}{b}\geq \left(\frac{a}{b}\right)^{b}$ in line \eqref{eq:lowerComb}, $k = o(p)$ in line \eqref{eq:lowerComb1} and Lemma \ref{lem:plower} (to lower bound $\log(p/k)$) in line \eqref{eq:lowerboundComb1} gives the following lower bound a.a.s. \ngrow\!,
    \begin{align}
        \log\left(\binom{k}{\floor{xk}} \binom{p - k}{\floor{(1-x)k}}\right) &\geq \log\binom{p - k}{\floor{(1-x)k}}\\
        &\geq \floor{(1-x)k}\log\left(\frac{p-k}{\floor{(1-x)k}} \right)\label{eq:lowerComb}\\
        &\geq ((1-x)k - 1)\log\left(\frac{p-k}{(1-x)k + 1}\right)\\
        &\geq ((1-x)k - 1)\log\left(\frac{p}{k}\frac{1-k/p}{1-1/k}\right)\\
        &\geq (1-o(1))(1-x)k\log\left(\frac{p}{k} \right)\label{eq:lowerComb1}\\
        &\geq (1-\epsilon_1)k(1-x)(1-\alpha)(1-C/2)\log(n).\label{eq:lowerboundComb1}
    \end{align}
    Similarly, for any $\epsilon_2 > 0$, invoking Lemma \ref{lem:M:limit} and Lemma \ref{lem:Mbounds} (to upper bound $N/2$) gives a.a.s. as \ngrow that,
    \[M \leq (1+\epsilon_2)\frac{N}{2} \leq (1+\epsilon_2)\frac{C(1-\alpha)k\log(n)}{2\log(2)} + O(k)\label{eq:lowerboundComb2}\]
    Combining \eqref{eq:lowerboundComb1} and \eqref{eq:lowerboundComb2} elicits, for some $\epsilon_3 > 0$, that for sufficiently large $n$,
    \begin{align}\frac{1}{M}\log\left(\binom{k}{\floor{xk}} \binom{p - k}{\floor{(1-x)k}}\right) &\geq \frac{(1-\epsilon_1)k\left(1-x\right)\left(1-\alpha\right)\left(1-C/2\right)\log\left(n\right)}{(1+\epsilon_2)\frac{C\left(1-\alpha\right)k\log\left(n\right)}{2\log(2)} + O(k)} \\
    &\geq \frac{\left(1-\epsilon_1\right)\left(1-x\right)\left(1-C/2\right)2\log(2)}{(1+\epsilon_2)C + o(1)}\\
    &\geq \frac{\left(1-\epsilon_1\right)\left(1-x\right)\left(1-C/2\right)2\log(2)}{\left(1+\epsilon_2\right)\left(1+\epsilon_3\right)C}.\label{eq:CombLower}\end{align}
    Now for $y_{x,1}:= 1-\frac{2a\log(2)x}{(1-\Co{roomR})r(x)}$, we then see that,
    \[F_0(x,y_{x,1}) =  \frac{2a\log(2)x}{(1-\Co{roomR})r(x)}D((1-\Co{roomR})r(x)||r(x))+D\left(1-\frac{2a\log(2)x}{(1-\Co{roomR})r(x)}\|\|s(x)\right).\]
    By the assumed inequality constraint \eqref{eq:uniConstraint}, with some $\Co{roomI} > 0$, we have that we can choose sufficiently small $\epsilon_1, \epsilon_2, \epsilon_3 > 0$ such that,
    \[F_0(x,y_{x,1}) \leq \frac{(1-\epsilon_1)(1-x)(1-\frac{C}{2})2\log(2)}{(1+\epsilon_2)(1+\epsilon_3)C} \leq \frac{1}{M}\log\left(\binom{k}{\floor{xk}} \binom{p - k}{\floor{(1-x)k}}\right).\] 
    Thus, for $y_{x,1} \geq 0$ it holds $F_0(x,y_{x,1}) \leq \frac{1}{M}\log\left(\binom{k}{\floor{xk}} \binom{p - k}{\floor{(1-x)k}}\right)$ a.a.s. as \ngrow\!. We can see that $y_{x,1}$ satisfies constraint \eqref{eq:r:constraint} by simple algebra. Moreover, using that $\frac{a\log(2)x}{1-2^{-x}} \geq \frac{\log(2)x}{1-2^{-x}} > 1- \Co{roomR}$ for all $x \geq 0,1 > \Co{roomR} > 0$ and $a \in (1,2)$ we see that $\frac{2a\log(2)x}{(1-\Co{roomR})4 \cdot 2^{-x}(1-2^{-x})} > 2^{x-1}$. This further implies that $\frac{y_{x,1}}{s(x)} = \frac{1-\frac{2a\log(2)x}{(1-\Co{roomR})r(x)}}{s(x)} < 1$ for all $a \in (1,2)$, $1 > \Co{roomR} > 0$ and $x \geq 0$. Meaning, there exists a sufficiently small $\Co{roomS} > 0$ which gives $y_{x,1} \leq (1-\Co{roomS})s(x)$ satisfying constraint \eqref{eq:s:constraint}. Hence for our choice of $y_{x,1}$ the constraints \eqref{eq:r:constraint} and \eqref{eq:s:constraint} are also satisfied.
     
    By elementary inspection, $F_0(x,y)$ is continuous in $y$. Thus, we invoke the intermediate value theorem to give that a solution $y(x)$ exists for all $x \in \scr{X}$ a.a.s. as $n \rightarrow +\infty$.
    Further, by the monotonicity of constraints \eqref{eq:r:constraint} and \eqref{eq:s:constraint} in $y \in [0,1)$, we have that the solution $y(x)$ also 
    satisfies constraints \eqref{eq:r:constraint} and \eqref{eq:s:constraint}.

    To prove uniqueness, we calculate,
    \[\frac{\der}{\der y}F_0(x,y) = \frac{C_a x}{1-y}\log\left(\frac{\frac{2a\log(2)}{1-y}}{r(x)}\frac{1-r(x)}{1-   \frac{2a\log(2)}{1-y}     } \right) - D\left( \frac{2a\log(2)x}{1-y}\|\| r(x) \right) + \log\left( \frac{y}{s(x)}\frac{1-s(x)}{1-y}\right).\]
    Observe that for any value of $y \in [0,1]$ which satisfies constraints \eqref{eq:r:constraint} and \eqref{eq:s:constraint}, we have that that the above derivative is strictly negative. Meaning that for any fixed $x$, $F_0(x,y)$ is monotonically decreasing with respect to $y$ on the set $[y_{x,0},y_{x,1}]$. These collections of facts allow us to conclude the solution to the equation
    \[\frac{1}{M}\log\left(\binom{k}{\floor{xk}} \binom{p - k}{\floor{(1-x)k}}\right) = F_0(x,y),\]
     exists and is unique for any $x \in \scr{X}$ a.a.s. as $n \rightarrow +\infty$, as we wanted.

\end{proof}

%\subsection{Continuity and Differentiability Of The First Moment Function}\label{subsec:Cont}

\begin{lemma}\label{lem:contDiff}
    If there exists a region of $\scr{X} \subseteq [0,1]$ where constraints \eqref{eq:existConstraint} and \eqref{eq:uniConstraint} hold, then the solution to the equation \eqref{FMF}, $y(x)$, is continuously differentiable for all $x \in \scr{X}$.
    \end{lemma}

 \begin{proof}[Proof of Lemma \ref{lem:contDiff}]
         Define the function 
        \[F(x,y) \coloneqq \frac{1}{M}\log\left(\binom{k}{\floor{xk}} \binom{p - k}{\floor{(1-x)k}}\right) - (1-y)D\left(\frac{2a\log(2)x}{(1-y)}\|\|r(x)\right)  - D(y||s(x))\]
        By Lemma \ref{lem:FMFUnique}, we have that the unique solution $y(x)$ satisfies the constraints \eqref{eq:r:constraint} and \eqref{eq:s:constraint} for all $x \in \scr{X}$.
        Fixing a point $x_* \in \scr{X}$, we consider a small interval $\scr{I}$ centered about $x_*$, such~that~for~all~$x \in \scr{I}$,
        \[\frac{1}{M}\log\left(\binom{k}{\floor{xk}} \binom{p - k}{\floor{(1-x)k}}\right),\]
        is constant as a function of $x$ on $\scr{I}$. On this interval $\scr{I}$,~the~function~$F$~is~continuously~differentiable~and
        \begin{align}
            \hspace{-.5cm}\frac{\der}{\der y} F(x,y) &= \frac{\der}{\der y} \left[(1-y)D\left(\frac{2a\log(2)x}{(1-y)}||r(x)\right)  + D(y||s(x))\right]\\
            &= \frac{1}{1-y} \left[\frac{\der}{\der x_1}D\left(x_1||r(x)\right)\right]  \bigg{|}_{x_1=\frac{2a\log(2)x}{(1-y)}} - D\left(\frac{2a\log(2)x}{(1-y)}\|\|r(x)\right) + \left[\frac{\der }{\der x_1}D(x_1||s(x))\right] \bigg{|}_{x_1=y},
        \end{align}where by $x_1$ we refer to the first argument of the KL divergence.

        Conditions \eqref{eq:r:constraint}, \eqref{eq:s:constraint} allows us to invoke Lemma \ref{lem:D_3} (by setting $\delta = \min(\Co{roomR},\Co{roomS})/2$) and conclude that both the following conditions hold, $\left[\frac{\der}{\der x_1}D\left(x_1||r(x)\right)\right]  \bigg{|}_{x_1=\frac{2a\log(2)x}{(1-y)}} \leq 0$ and \\$\left[\frac{\der }{\der x_1}D(x_1||s(x)) \right]\bigg{|}_{x_1=y} \leq 0$. Hence, by constraint \eqref{eq:r:constraint}, for all $x \in \scr{X},$ it holds 
        \begin{align}
            \frac{\der}{\der y} F(x^*,y) \leq D\left(\frac{2a\log(2)x}{(1-y)}\|\|r(x)\right)<0.
        \end{align}
        By the two-dimensional implicit function theorem, we conclude that $y(x)$ is continuously differentiable for $x \in \scr{X} \subseteq [0,1]$.
\end{proof}

We are now in a position to prove the vital Lemma \ref{lem:FMFnear0}.

\begin{proof}[Proof of Lemma \ref{lem:FMFnear0}]
    By Lemma \ref{lem:FMFUnique} and Lemma \ref{lem:contDiff}, we just need to show that there exists a region $\scr{X} = [0,\epsilon]$ such that constraints \eqref{eq:existConstraint} and \eqref{eq:uniConstraint} both hold.

    When $x = 0$, using that $r(x) = 4\log(2)x - O(x^2)$ as $x \conv{} 0^+$, the constraint \eqref{eq:uniConstraint} is equivalent to $D\left(1- \frac{a}{2(1-\Co{roomR})}\big{|}\big{|}\frac{1}{2}\right) \leq (1-\Co{roomI})\frac{2-C}{C}\log(2)$, which is assumed by Assumption \ref{as:FMFExists}. By continuity of both sides of the inequality in \eqref{eq:uniConstraint}, we have, for a sufficiently small $\epsilon_1 > 0$, that constraint \eqref{eq:uniConstraint} (say, with constant $\Co{roomI}/2$) holds for all $x \in [0,\epsilon_1]$.
    
    Again using, $r(x) = 4\log(2)x - O(x^2)$ as $x \conv{} 0^+$, we have that constraint 
    \eqref{eq:existConstraint} is equivalent to $\frac{2a\log(2)x}{(1-\Co{roomR})4\log(2)x} \leq 1 - \frac{O(x^2)}{(1-\Co{roomR})4\log(2)x} = 1 - O(x)$. Hence, it suffices $\frac{a}{2(1-\Co{roomR})} \leq 1 - O(x)$, which is satisfied for all $x \in [0,\epsilon_2]$ for a sufficiently small $\epsilon_2 > 0$ when $\frac{a}{2(1-\Co{roomR})} < 1$. Taking $\epsilon = \min(\epsilon_1,\epsilon_2)$ gives the proof.
\end{proof}

%------------------------------------

\section{Proofs of Theorem \ref{thm:combinedBound} and Theorem \ref{thm:Phi_kLim}}\label{sec:combProof}

\subsection{Structure of the Proofs}
The proof of Theorem \ref{thm:combinedBound} is accomplished in three steps.

 In Section \ref{subsec:Lower}, we first establish the lower bound; we prove that for any $x=l/k \in [0,1]$ that the first moment solution $y(x)$ exists, $\phi(l)$ is larger than $y(x)=y(l/k)$ up to an additive $O(1/k)$ error. Notice that this proves the first part of Theorem \ref{thm:combinedBound}.

In Section \ref{subsec:soly(0)}, we calculate that the limiting value of the first moment function at $x=0$, i.e., $y(0)$, is $H_C$.

Finally, our last step is to prove that $\phi(0)$ is upper bounded from $H_C$ up to an additive $O(\kpert)$ error. This, combined with the first and second steps for $x=0$, proves the second and last part of the Theorem \ref{thm:combinedBound}. 

As explained in the main body of the paper, the third part is the most technical part of this proof. Moreover, establishing it turns out to be equivalent to identifying the max-satisfiability threshold $\Phi_k$ of \kset in an appropriate parameter regime, as described in Theorem \ref{thm:Phi_kLim}. We elaborate more on this connection in Section \ref{subsec:relate}. For these reasons, we establish the last third part of the Theorem \ref{thm:combinedBound} (and therefore also Theorem \ref{thm:Phi_kLim}) in Section \ref{subsec:Up}. 

\subsection{On the fluctuations of $M,p$}\label{subsec:Plan}
Proving Theorem \ref{thm:combinedBound} requires us to control the fluctuations in the number of positive tests $M$ and the number of possible infected $p$. Below we give two results that provide tight upper and lower bounds on these fluctuations, a.a.s. as \ngrow\!, for $M$ and $p$ respectively. Both of these Lemmas are proven in Appendix \ref{subsec:commonProofs}, and they are extensions of similar results in \cite{iliopoulos2021group}.

\begin{lemma}\label{lem:M:limit}
    Recall $N = \floor{C \log_2\binom{n}{k}}$ with $C \in (1,2)$, and $M$ as the number of positive tests. 
    %We then have for every constant in $\eta \in (0,1)$, a.a.s. as \ngrow\!, that
    %\[(1-\eta)\frac{N}{2}\leq M \leq (1+\eta)\frac{N}{2}.\] \textcolor{red}{Why we do not just keep the second stronger statement?}\maxwell{I use the former when the accuracy of the later is not needed, Its kept here for reference to the reader} \textcolor{red}{it is immediate so not need to display here}
    We have for every $\eta_M \in (0,1/2)$, a.a.s. as \ngrow\!, that
    \[(1-N^{-\eta_M})\frac{N}{2}\leq M \leq (1+N^{-\eta_M})\frac{N}{2}.\]
\end{lemma}

\begin{lemma}\label{lem:p:limit}
    Consider $C \in (1,2)$, $\alpha \in (0,1/3)$ and recall that $p$ denotes the number of possible infected. We have for every $\eta_p \in (0,C/4)$ that a.a.s. as \ngrow\!,
    \[(1-k^{-\eta_p})n\left(\frac{k}{n}\right)^{\frac{C}{2}(1+k^{-\eta_p})}\leq p \leq (1+k^{-\eta_p})n\left(\frac{k}{n}\right)^{\frac{C}{2}(1-k^{-\eta_p})}\hspace{-1.5cm}.\]
\end{lemma}

Notice that the previous lemmas imply that a.a.s. \ngrow, for any $c \in (0,1/4)$, as $n$ grows,
 \[(1-N^{-c})N/2 \leq M \leq (1+N^{-c})N/2 \label{eq:1} \]
 \[(1-k^{-c})n\left(\frac{k}{n}\right)^{\frac{C}{2}(1+k^{-c})}\leq p \leq (1+k^{-c})n\left(\frac{k}{n}\right)^{\frac{C}{2}(1-k^{-c})} \label{eq:2} \]

Given that, in the proof of Theorem \ref{thm:combinedBound}, we treat $M,p$ as arbitrary deterministic numbers satisfying the a.a.s. \ngrow conditions described in \eqref{eq:1}, \eqref{eq:2}.

\subsection{Lower Bounding $\phi(l)$}\label{subsec:Lower}

Our first step towards proving Theorem \ref{thm:combinedBound} is to establish a lower bound on $\phi(l)$ for all $l$ where $x = l/k$ is sufficiently small. This is accomplished using a conditional first moment method argument.

\begin{theorem}\label{lem:lowerbound}
  Assume that $M, p$ are deterministic and satisfies the conditions \eqref{eq:1}, \eqref{eq:2}. Moreover, assume that the parameters $(\alpha, C, a, \Co{roomR}, \Co{roomI})$ satisfy Assumption \ref{as:FMFExists} and Assumption \ref{as:alphaC_0}.
    
    Let $\phi(l)$ as defined in \eqref{eq:phil}. There exists a constant $\Co{low} > 0$ (dependent on $\Co{roomS}$ and $\Co{roomR}$) and $\epsilon' > 0$ such that, for all $l/k \in [0,\epsilon']$ and $n$ sufficiently large, we have that
    \begin{align}\label{eq:LB_phi}
        \phi(l) \geq y\left({\frac{l}{k}}\right) - \frac{\Co{low}}{k}.
    \end{align}
\end{theorem}

\begin{proof}[Proof of Theorem \ref{lem:lowerbound}]
   By Lemma \ref{lem:FMFnear0} and Remark \ref{rem:Scale}, we have that under condition \eqref{eq:1}, condition \eqref{eq:2} and Assumption \ref{as:FMFExists}, the first moment function $y(x)$ exists and is unique on the region $x \in [0,\epsilon']$ for some $\epsilon' > 0$.

    Recall the event $\scr{A}$ from Lemma \ref{lem:aSet}. We will demonstrate that under conditions \eqref{eq:1} and \eqref{eq:2},  with $\scr{S} = \{l : l/k \in [0,\epsilon']\}$, that
    \[\limsup_{n \conv{} \infty}\sum_{l \in \scr{S}} \E[Z_{t,l}|\scr{A}] = 0\label{eq:goal:Sum},\]
    for a choice of $t = My(l/k) - \Co{low}\log(k)$ with an appropriately chosen large constant $\Co{low} > 0$. By Markov's Inequality, condition \eqref{eq:goal:Sum} suffices to prove the theorem. Indeed, we can use condition \eqref{eq:1} with $k = \Theta(n^\alpha)$ to get that $M = \Theta(k \log(k))$ and absorb the implicit constant inside $\Co{low}$ to derive \eqref{eq:LB_phi}.
    Using Proposition \ref{cor:upperBound}\footnote{For this proof we drop the $(1-o(1))$ error from the proposition as it will not affect the limit \eqref{eq:goal:Sum}.} with $l = xk$ and $t =  My(x) - \Co{low}\log(k)$ gives that
    \[
        \label{eq:FMF:upperBound1}\begin{split}\E[Z_{ My(x) - \Co{low}\log(k),xk}|\scr{A}] &\leq \comb\exp\bigg{(}-M\bigg{[}\left(1-y(x) + \frac{\Co{low}\log(k)}{M}\right)\\
        &\qquad\times D\left(\frac{2a\log(2)x}{1-y(x) + \frac{\Co{low}\log(k)}{M}}\|\|r(x)\right) + D\left(y(x) - \frac{\Co{low}\log(k)}{M} \|\|s(x)\right)\bigg{]}\bigg{)}.\end{split}
    \]
    Defining $\rho = \frac{2a\log(2)x\frac{\Co{low}\log(k)}{M}}{(1-y(x) + \frac{\Co{low}\log(k)}{M})(1-y(x))} $, rearranging terms in the exponent of \eqref{eq:FMF:upperBound1} elicits,
    \begin{align}
        \begin{split}&-M\bigg{[}\left(1-y(x) + \frac{\Co{low}\log(k)}{M}\right) D\left(\frac{2a\log(2)x}{1-y(x) + \frac{\Co{low}\log(k)}{M}}\|\|r(x)\right) + D\left(y(x) - \frac{\Co{low}\log(k)}{M} \|\|s(x)\right)\bigg{]}\end{split}\\
        \begin{split}
        &\qquad= -M\bigg{[}(1-y(x))D\left(\frac{2a\log(2)x}{1-y(x)} - \rho\|\|r(x)\right) + \frac{\Co{low}\log(k)}{M}D\left(\frac{2a\log(2)x}{1-y(x)} - \rho\|\|r(x)\right) \\
        &\qquad\qquad + D\left(y(x) - \frac{\Co{low}\log(k)}{M} \|\|s(x)\right)\bigg{]}\label{eq:FMF:exponent}.
        \end{split}
    \end{align}
    Under condition \eqref{eq:1}, we have that $\log(k)/M = o(1)$ and thus, for large enough $n$, we have that $\rho < \frac{2a\log(2)x}{1-y(x)}$. Furthermore, we utilize \eqref{eq:r:constraint} and \eqref{eq:s:constraint}, as $y(x)$ exists over the region $[0,\epsilon']$, to justify the existence of constant $\Co{roomR}, \Co{roomS} > 0$ such that $\frac{2a\log(2)x}{1-y(x)} \leq (1-\Co{roomR})r(x)$ and $y(x) \leq (1-\Co{roomS})s(x)$ for all $x \in [0,\epsilon']$. Thus, we invoke Lemma \ref{lem:D_3}, with $\delta = \min(\Co{roomR}, \Co{roomS})$, to guarantee the following lower bounds with constant $c_0>0$:
    \[
        D\left(\frac{2a\log(2)x}{1-y(x)} - \rho\|\|r(x)\right) \geq D\left(\frac{2a\log(2)x}{1-y(x)}\|\|r(x)\right) +c_0\rho \geq D\left(\frac{2a\log(2)x}{1-y(x)}\|\|r(x)\right),
    \]
    \[
        D\left(y(x) - \frac{\Co{low}\log(k)}{M} \|\|s(x)\right) \geq D(y(x)||s(x)) + c_0 \frac{\Co{low} \log(k)}{M}.
    \]
    Thus, we conclude that our exponent in \eqref{eq:FMF:exponent} is bounded above by
    \[\label{eq:FMF:exp2}-M\bigg{[} \underbrace{(1-y(x))D\left(\frac{2a\log(2)x}{1-y(x)}\|\|r(x)\right) + D(y(x)||s(x))}_{(A)} + \frac{\Co{low}\log(k)}{M}\left(c_0 + D\left(\frac{2a\log(2)x}{1-y(x)}\|\|r(x)\right)\right)\bigg{]}.\]
    Noticing that term $(A)$ in \eqref{eq:FMF:exp2} is the solution to the first moment function in Definition \ref{def:FMF}, we can simplify \eqref{eq:FMF:exp2} to
    \[\label{eq:FMF:exp3}-\log\left(\comb\right) -M\bigg{[}\frac{\Co{low}\log(k)}{M}\left(c_0 + D\left(\frac{2a\log(2)x}{1-y(x)}\|\|r(x)\right)\right)\bigg{]}.\]
    The left most term in \eqref{eq:FMF:exp3} cancels the combinatorial pre-factor in \eqref{eq:FMF:upperBound1}, allowing us to further bound
    \begin{align}
        \E[Z_{ My(x) - \Co{low}\log(k),xk}|\scr{A}]&\leq \exp\left(- M\left[\frac{\Co{low}\log(k)}{M}\left(c_0 + D\left(\frac{2a\log(2)x}{1-y(x)}\|\|r(x)\right)\right)\right]\right)\\
        &\leq \exp\left(- {c_0\Co{low}\log(k)}\right)\label{eq:c_0:1}
    \end{align}
    as KL-Divergence is always positive. Thus, setting $t = My(l/k) + \Co{low} \log(k)$, using that $|\scr{S}| \leq k$ and choosing $\Co{low}$ such that $c_0\Co{low} \geq 1 + \epsilon$, for some $\epsilon > 0$, gives that
    \[\limsup_{n \conv{} \infty}\sum_{l \in \scr{S}}\E[Z_{t,l}|\scr{A}] \leq \limsup_{n \conv{} \infty}\sum_{l \in \scr{S}}k^{-(1+\epsilon)} \leq \limsup _{n \conv{} \infty}k^{-\epsilon} = 0,\]
    completing the proof.
\end{proof}

\subsection{Solving For $y(0)$}\label{subsec:soly(0)}

The second step to prove Theorem \ref{thm:combinedBound} is to identify the limiting value of the first moment function $y(x)$ from Definition \ref{def:FMF} at $x = 0$.

%This is a relatively straightforward calculation that mainly concerns showing that error terms arising from asymptotic approximations are negligible.

%\textcolor{red}{IZ: do we need the first moment function at a neighborhood of zero?}\maxwell{I don't think so, but if you satisfy the conditions in the FMF then you automatically do so If you }

\begin{lemma}\label{lem:solZeroRigor}
    Assume that $M, p$ are deterministic and satisfies the conditions \eqref{eq:1}, \eqref{eq:2}. Recall the first moment function, $y(x)$, from Definition \ref{def:FMF} and $H_C$ from Definition \ref{def:H_C}. If the parameters  $(\alpha, C, a, \Co{roomR}, \Co{roomI})$ satisfy Assumption \ref{as:FMFExists} and Assumption \ref{as:alphaC_0}, then,
    \begin{equation}
        y(0) = H_C + o(1).
    \end{equation}
\end{lemma}

%Before we give the proof of this result we can plot the curve $H_C = h_2^{-1}(2-2/C)$ for $C \in (1,2)$ to observe its behavior in Figure \ref{fig:HCPlot}. Interestingly this function is independent of the value of $\alpha$ chosen; something perhaps surprising as in principle $\phi(0)$ and $y(0)$ depend on the value of $k=n^{\alpha+o(1)}.$

%\begin{figure}[ht]
 %   \centering
 %   \includegraphics[scale = .35]{figs/H_C.pdf}
 %   \caption{ A plot of the function $H_C = h^{-1}(2-2/C)$ for $C \in (1,2)$.}
 %   \label{fig:HCPlot}
%\end{figure}

\begin{proof}[Proof of Lemma \ref{lem:solZeroRigor}]
    By Lemma \ref{lem:FMFnear0} and Remark \ref{rem:Scale}, we have that under conditions \eqref{eq:1}, condition \eqref{eq:2} and Assumption \ref{as:FMFExists}, $y(x)$ exists and is unique on the region $x \in [0,\epsilon']$ for some $\epsilon' > 0$. This means that $y(0)$ is well-defined as the solution to the following equation at $x = 0$,
    \[\frac{1}{M}\log \left(\binom{k}{kx}\binom{p-k}{k}\right) =  D \left( y(x) || s \left( x \right)\right) +  D \left(\frac{2\log \left(2\right) a  x }{1 -  y(x) }\|\| r \left( x \right)\right),\label{eq:FMF:zeroSetup}\]
    Plugging in $x = 0$, we have, with $y(0) = y$ for the remainder of the proof, that        
         \[\frac{1}{M}\log\binom{p-k}{k} = \log(2) - h(y).\]
         Rearranging terms and applying $h^{-1}$ (on the $[0,1/2]$ branch), we see that $y = h^{-1}\left(\log(2) - \frac{1}{M}\log\binom{p-k}{k}\right)$. Now we consider \ngrow\!, by the continuity of $h^{-1}$, we have
         \[\lim_{n \conv{} \infty} y = h^{-1}\left(\log(2) - \lim_{n \conv{} \infty}\frac{1}{M}\log\binom{p-k}{k}\right).\label{eq:zeroSol:setup}\]
         Thus, we need to calculate the asymptotic value of $\frac{1}{M}\log\binom{p-k}{k}$. By \eqref{eq:2}, $p = \Omega\left(n(k/n)^{(1+k^{-\err})C/2}\right)$, giving
         \[\frac{k^2}{p} = O\left(\frac{n^{2\alpha}}{n^{1+(\alpha-1)(1+k^{-\err})C/2}}\right) = O\left(n^{2\alpha + (1-\alpha)(1+k^{-\err})C/2 - 1}\right)\]
         Using Assumption \ref{as:alphaC_0}, we see that $2\alpha + \frac{C}{2}(1-\alpha) - 1 < 0$ and thus, for a sufficiently large $n$, since $k=\omega(1)$ we have that $\frac{k^2}{p} = o(1)$. Then, utilizing Stirling's approximation for $p-k$ and $k$ growing, we have that there exists a sequence $\delta_n \conv{} 0$ with
            \[\log\binom{p-k}{k} = \log\left((1+\delta_n)\left(\frac{(p-k)e}{k}\right)^{ k }\left(2\pi  k \right)^{ - \frac{1}{2}} e ^{\frac{ k ^{2}}{2 (p-k) }\left(1 + \delta_n\right)}\right).\]
         As we previously showed that $O(k/p) = O(k^2/p) = o(1)$, the leading order term in the above equation is $k\log(p/k)$ which means that there exists a sequence $\delta_n \conv{} 0$ a.a.s. as \ngrow where
         \[\frac{\log\binom{p-k}{k}}{M} = \frac{(1+\delta_n)k\log\left(\frac{p}{k}\right)}{M}.\label{eq:stir1}\]
        Again using \eqref{eq:2}, alongside condition \eqref{eq:1}, we have the following upper and lower bounds using Lemma \ref{lem:Mbounds}, Lemma \ref{lem:plower}:
        \[\frac{k(1-\alpha)(1-C/2)\log(n) - O(k^{1-\err})}{(1+ k^{-\err})\left(\frac{C(1-\alpha)k\log(n)}{2\log(2)} + O(k)\right)} 
        \leq \frac{k\log(p/k)}{M} \leq \frac{k(1-\alpha)(1-C/2)\log(n) + O(k^{1-\err})}{(1- k^{-\err})\left(\frac{C(1-\alpha)k\log(n)}{2\log(2)} - O(1)\right)}.\]
        One can then calculate that
        \[\frac{k\log(p/k)}{M} = \log(2)\frac{2-C}{C} + o(1).\]
        Combining this fact with \eqref{eq:stir1} leads to the existence of a sequence $\tilde{\delta}_n \conv{} 0$ where $\frac{1}{M}\log\binom{p-k}{k} = \log(2)\frac{2-C}{C} + \tilde{\delta}_n$ a.a.s. as \ngrow\!. Plugging this fact into \eqref{eq:zeroSol:setup} gives, $h^{-1}\left(\log(2)\left(1-\frac{2-C}{C}\right) + \tilde{\delta}_n\right) = y$.
        By the inverse function theorem we have that the derivative of $h^{-1}$ at any input $z\in [0,\log(2)]$ is given by $(h^{-1})'(z) = \frac{1}{h'(h^{-1}(z))}$.
        Thus, $(h^{-1})'(z)$ is bounded when $h^{-1}(z)$ is bounded away from $1/2$. As $h(z)$ is the $(0,1/2)$ branch of entropy, this remains true for all $z < \log(2)$. As $1<C<2$ and $z = \log(2)(1-\frac{2-C}{C}) + o(1)$ then this constraint will hold for large enough $n$. Thus, by the mean value theorem, there exists a sequence $\delta_n \conv{} 0$ as \ngrow such that,
        \[h^{-1}\left(\log(2)\left(1-\frac{2-C}{C}\right)\right) + \delta_n = y(0).\]
        Recognizing that $h^{-1}(\log(2)x) = h_2^{-1}(x)$ and $1- \frac{2-C}{C} = 2-2/C$ gives the desired result.

\end{proof}

\subsection{Relating Group Testing To Random MAX-Set Cover}\label{subsec:relate}

Now we turn our focus on proving the final part of Theorem \ref{thm:combinedBound}, which is that a.a.s. as $n \rightarrow +\infty$, $\phi(0)=y(0)+o(1)$. Using the result of the previous subsection it suffices to show that a.a.s. as $n \rightarrow +\infty$, $\phi(0)=H_C+o(1).$ This will be proven in Section \ref{subsec:Up} by establishing Theorem \ref{thm:Phi_kLim} as we explain below.

They key observation is that $\phi(0)=\min_{\sigma \cap \sigma*=\emptyset} H(\sigma)$ has, in fact, no dependence on the planted signal $\sigma^*$ as it can be simply rephrased as a maximization over all $k$-subsets $\sigma$ of the $p-k$ possibly infected \emph{but not infected} individuals. Interestingly, it is for this reason that as long as we fix $M,p$ to take deterministic values then $\phi(0)$ equals in distribution to $1-\Phi_k$, where $\Phi_k$ the maximum satisfiability threshold of a ``null'' model called the \kset problem for $\mSet=M,\pSet=p-k$, which is explicitly described in Section \ref{subsec:MAX}.

In terms of parameters, conditioning $M,p$ to be arbitrary numbers satisfying \eqref{eq:1}, \eqref{eq:2} then the assumptions of Theorem \ref{thm:combinedBound} for BGT are mapped identically to the assumptions of Theorem \ref{thm:Phi_kLim} (in particular $\mSet, \pSet$ satisfying Assumption \ref{as:MPScale}).
Hence, by the previous two subsections, we can conclude that for \kset the assumptions of Theorem \ref{thm:Phi_kLim} it holds $\Phi_k \leq 1-H_C+o(1)$. Moreover, if we prove that a.a.s. as $n \rightarrow +\infty$  it holds \[\Phi_k \geq 1-H_C+o(1) \label{eq:final_s}\] we get an equivalent upper bound on $\phi(0)$ and in particular complete simultaneously the proof of both Theorem \ref{thm:combinedBound} and Theorem \ref{thm:Phi_kLim}. This will be the topic of the following section.

\section{The Lower Bound On the Max-Satisfiability Threshold $\Phi_k$ }\label{subsec:Up}

As explained in Section \ref{subsec:relate}, we focus here on completing the proofs of Theorem \ref{thm:combinedBound} and Theorem \ref{thm:Phi_kLim}, for which it suffices to show \eqref{eq:final_s} under the assumptions of Theorem \ref{thm:Phi_kLim}. In particular, in this section we follow the (equivalent) notation of the \kset problem. In words, we aim to prove that there exists a set of elements with size $k$ that leave all but $H_C + o(1)$ ``target'' sets uncovered. For ease of notation we now set $M=\mSet$ and $p=\pSet$ which recall are now deterministic numbers satisfying Assumption \ref{as:MPScale} in the context of \kset.

We prove this result using a second moment method analysis on the random variable $Z_{yM,0}$ which counts the number of $k$-subsets covering at most $yM$ sets. As is often the case, a direct second moment argument has difficulties with obtaining tight results and the ``art'' is to appropriately condition it to make it succeed. To overcome this difficultly in our case, we instead consider a surrogate counting random variable which lower bounds the random variable and counts only a carefully chosen ``well-behaved'' (or ``flat'') set of the $k$-subsets (disjoint from $\sigma^*$) covering exactly $yM$ sets. This surrogate counting random variable is inspired by similar ``flatness" arguments from \cite{balister2018dense,gamarnik2024overlap}. Before we go into the specifics of this second moment calculation, we introduce the concept of flatness in our setting and build up the necessary tools for the second moment method proof. 

\subsection{Getting Started: Flatness In The Random Max K-Set Problem}

Flatness, speaking informally, is the condition that whenever a set $\sigma$ leaves $yM$ sets uncovered (with $y \in (0,1/2)$) then the number of sets covered by \emph{any subset} $\sigma' \subseteq \sigma$ concentrates around its conditional expectation given that $\sigma$ leaves $yM$ sets uncovered. Interestingly, this conditional expectation depends only on the size, $|\sigma'|$, of the subset $\sigma'$ of $\sigma$. This allows us to employ the following simplifying notation for our purposes.
\begin{notation}\label{not:sigmal}
   Given a set $\sigma$ and any $\ell \in \mathbb{N}$ with $0 \leq \ell \leq |\sigma|$ we denote by $\sigma_l \subseteq \sigma$ to be an arbitrary subset of $\sigma$ with $|\sigma_l| = l$. 
\end{notation}
Following this logic, let us first condition that a set of elements $\sigma$, with $|\sigma| = k$, leaves exactly $y M$ sets uncovered. We then find the expected number of sets left uncovered by any fixed subset $\sigma_l \subseteq \sigma$. To formally do so, we define the key random variable in question. 
\begin{definition}\label{def:Xsigma}
    Let random variable $X_\sigma$ to be the number of sets left uncovered by $\sigma$.
\end{definition}

We consider the expectation of $X_{\sigma_l}/M$ conditioned on the event that $X_{\sigma} = yM$. This expectation has a simple form based on $l$ and $y$ that we define now.
\begin{definition}\label{def:SubsetProp}
Given $q$ such that $(1-q)^k = 1/2$ and $0\leq l \leq k$, let
\begin{equation}
    y_{(l)} := y + (2(1-q)^l - 1)(1-y) = y + (2^{1-\frac{l}{k}} - 1)(1-y)\label{gamma l}.
\end{equation}
Similarly, given $x \in [0,1]$ define\footnote{Note that we denote this proportion of covered sets as $y_{(x)}$, not to be confused with the solution to the first moment function $y(x)$.}
\[y_{(x)} := y + (2^{1-x}-1)(1-y)\label{def:gammaeps}.\] 
\end{definition}
To calculate the conditional expectation we first make the following probabilistic calculation.
\begin{lemma}\label{lem:p_l}
    Given a set of elements $\sigma$ with $|\sigma| = k$, and our notation $\sigma_l$ given above, we have that for any target set $m,$ \[p_l \coloneqq \P(\sigma_l \text{ does not cover set } m | m \text{ is covered by $\sigma$}) = 2^{1-l/k} - 1\]
\end{lemma}

\begin{proof}[Proof of \ref{lem:p_l}]Recall that each element is included in test $m$ independently and with probability $q$. Hence, it holds
    \begin{align}
        p_l &= \frac{\P(\sigma_l \text{ nodes do not cover set } m \cap m \text{ is covered by $\sigma$})}{\P(m \text{ is covered by $\sigma$})}\\
        &= \frac{(1-q)^l(1-(1-q)^{k-l})}{1-(1-q)^k} = \frac{(1-q)^l-(1-q)^{k}}{1-(1-q)^k} = \frac{(1-q)^l-1/2}{1/2} = 2(1-q)^l - 1\\
        &= 2 \cdot 2^{-l/k} - 1,\label{eq:used(1-q)}
    \end{align}where we used that $(1-q)^k = 1/2$ in \eqref{eq:used(1-q)}.
\end{proof}
By applying Lemma \ref{lem:p_l}, we can see that the expected number of sets left uncovered by $\sigma_l$ is a sum of a deterministic value (after conditioning) of uncovered sets $yM$ and the expectation of a $\Bi{(1-y)M, p_l}$ random variable. The below lemma confirms our choice of $y_{(l)}$ in Definition \ref{def:SubsetProp}.

\begin{lemma}\label{lem:gammaepsexp}
    Given $X_{\sigma}$, $X_{\sigma_l}$ from Definition \ref{def:Xsigma} and $y_{(l)}$ from Definition \ref{def:SubsetProp}, the following statement holds. If $y \in (0,1/2)$, then $\E[X_{\sigma_l} | X_{\sigma} = y M] = y_{(l)} M.$
\end{lemma}
    
\begin{remark}\label{rem:epsGammaEpsLims}
      Before we proceed with the proof, we present some intuition on the formula of $y_{(\ell)}$ in the two extreme cases. When $l = 0$, $\sigma_{l}$ is empty, so it does not cover any sets. Indeed, it is easy to see that $y_{(0)} = 1$. When $l = k$, $\sigma_{\ell}=\sigma$ and therefore it must be true that $y_{(k)} = y.$ Indeed, that holds since $2(1-q)^l - 1 = 2(1/2) - 1 = 0$.  
\end{remark}

\begin{proof}[Proof of Lemma \ref{lem:gammaepsexp}]
    The expected number of sets left uncovered by $\sigma_l$ can be decomposed into two parts. The first part is the proportion of sets which are missed by the set $\sigma$, which $\sigma_l$ cannot possibly cover. The second is the expectation of a binomial over all the sets which are covered by $\sigma$. Thus, using Lemma \ref{lem:p_l}
    \begin{align}
        \frac{1}{M}\E[X_{\sigma_l} | X_{\sigma} = y M] &= y + \frac{1}{M}\E[\Bi{(1-y)M, p_l}] = y + p_l(1-y)\\
        & = y + (2^{1-l/k} - 1)(1-y) = y_{(l)}.
    \end{align}
\end{proof}

We now must demonstrate the rate that any such subset $\sigma_l$ can deviate from leaving $y_{(l)} M$ sets uncovered. Meaning that we want to find an appropriate $D_l > 0$ under which all possible subsets $\sigma_l$ have their number of uncovered sets in $[y_{(l)}M - D_l, y_{(l)}M + D_l] \cap \{0,1,\dots,M\}$ a.a.s. as \ngrow~\!\!\!~. The Lemma below provides us of a valid choice for $D_l$. The proof of this result is deferred to Appendix \ref{subsec:proofD_linterval} since it relies on some technical aspects of two-point KL divergence.
\begin{lemma}\label{lem:D_linterval} 
\stateScale Given $X_{\sigma}$ from Definition \ref{def:Xsigma} we condition on $X_\sigma = yM$ for some $y \in (0,1/2)$. Then for any $\sigma_l \subseteq \sigma$ with $|\sigma_l| = l$, $y_{(l)}$ from Definition \ref{def:SubsetProp}, $p_l$ from Lemma \ref{lem:p_l}, and any constant $\Co{D_l}>0$, define
    \[D_{l,\Co{D_l}} := \sqrt{6p_l(1-p_l)(1-y)M \left[\log \binom{k}{l} + (1+\Co{D_l})\log k\right]},\label{eq:D_l:def}\]the following holds.
If $\alpha < 28/1000$, then for every $\sigma_l$ with $0 \leq l \leq k$, we have that $|X_{\sigma_l} - y_{(l)} M| \leq D_l$ for sufficiently large $n$.
\end{lemma}

\begin{remark}
    Notice that for this proof we have considered $\alpha < 28/1000$ in order to be able to invoke Lemma \ref{lem:cornerBound} for specific bounds on $D_l$. Indeed, we can see that this condition is part of Assumption \ref{as:alphaC_0}.
\end{remark}

Lemma \ref{lem:D_linterval} motivates the following definition for a flat subset.

\begin{definition}\label{def:flat}
    Given $\sigma_l \subseteq \sigma$ and $\Co{D_l}>0$, define a set of elements $\sigma$, of size $k$, to be $\Co{D_l}$-flat if, for every $l \in \range{k}$, the number of sets left uncovered by each possible $\sigma_l$ is in the interval 
    \[[My_{(l)} - D_{l, \Co{D_l}}, My_{(l)} + D_{l, \Co{D_l}}] \cap \{0,\dots,M\}.\]
\end{definition}
Depending on the order of $l/k$ as \ngrow\!, the order of the proportional radius $D_l/M$ changes. Controlling this radius under differing regimes of $l/k$ is vital to our second moment method proof. Lemma \ref{lem:cornerBound} gives the following bounds on $D_{l, \Co{D_l}}/M$ when $\alpha < 28/1000$.
We repeat them for reader's convenience. When $M$ and $p$ satisfy Assumption \ref{as:MPScale}, for a sufficiently small $\Co{D_l}$ there exists a $\delta > 0$ such that, for large $n$, we have
\begin{align}
    \max_{l/k \leq \delta} \frac{D_{l, \Co{D_l}}}{M}&\leq 7\log(2)\sqrt{\frac{\alpha}{1-\alpha}}\frac{l}{k},\label{eq:boundComplex}\\
    \max_{l/k \geq 1-\delta}\frac{D_{l, \Co{D_l}}}{M} &\leq 5\log(2)\sqrt{\frac{\alpha}{1-\alpha}}\left(1-\frac{l}{k}\right)\label{eq:boundSimp},\\
    \max_{l = 1, \dots, k}\frac{D_{l, \Co{D_l}}}{M} &= O\left(\frac{1}{\sqrt{\log(n)}} \right)\label{eq:boundMid}.
\end{align}

\begin{remark}
    To invoke the above result it suffices to have $\Co{D_l}$ to be a sufficiently small constant. Because of this in what follows we only refer to the sets $\{D_{l, \Co{D_l}}\}_{l \in \{0, \dots, k\}}$ as simply $\{D_l\}_{l \in \{0, \dots, k\}}$ for the remainder of the paper, where the choice of $\Co{D_l}$ is implicit.
\end{remark}

\subsection{Using Flatness To Simplify Our Second Moment Calculation}\label{eq:paley-Section}

Now that we have introduced the concept of flatness, we turn to bounding the number of size $k$ \textit{flat} (for some choice $\Co{D_l} > 0$) subsets $\sigma$ which leave $yM$ sets uncovered. We will defer the proofs in this subsection to Appendix \ref{subsec:FlatnessBounds}.
\begin{definition}\label{def:countFlat}
    Given $\Co{D_l}> 0$ and Definition \ref{def:flat}. Consider an $(\alpha,C)$ \kset instance on $n$ elements. Denote the set of all $k$ subsets by $\Omega=\Omega_k$. Define the counting random~variable
    \[Y_y \coloneqq |\{\sigma:\sigma \in \Omega, \sigma \text{ is }\Co{D_l}\text{-flat},  X_\sigma = yM\}|\].
\end{definition}
It is obvious that,
\[Y_y \leq Z_{yM,0},\label{eq:obvious}\]as leaving exactly $y M$ positive tests uncovered is a requirement to be counted by $Y_y$.

%\begin{mdframed}\textbf{OLD}
%Consider the event $\scr{E}$ defined as
%\[\scr{E} = \left\{(1-N^{-3/16})\frac{N}{2} \leq M \leq (1+N^{-3/16})\frac{N}{2}\right\} \bigcap \left\{p \geq (1-k^{-3/16})n\left(\frac{k}{n}\right)^{\frac{C}{2}(1+k^{3/16})}\right\},\label{eq:conditioning}\]
%Using Lemma \ref{lem:M:limit}, Lemma \ref{lem:p:limit} and a union bound we can see that $\P(\mathcal{E}^C) = o(1)$ and thus by the law of total probability and the Paley-Zygmund inequality we have
%\begin{align}
%    \P(Y_y > 0) &= \P(Y_y > 0|\scr{E})\P(\scr{E}) + \P(Y_y > 0|\scr{E}^C)\P(\scr{E}^C)\\
  %  &\geq (1-o(1))\P(Y_y > 0|\scr{E})\\
 %   &\geq (1-o(1))\frac{\E[Y_y^2|\scr{E}]}{\E[Y_y|\scr{E}]^2}
%\end{align}
%\end{mdframed}
Thus, by a utilization of the Paley-Zygmund inequality, we have reduced the asymptotic almost sure existence of a size $k$ flat subset leaving $yM$ sets uncovered to showing the following condition:
\[\lim_{n \conv{} \infty}\frac{\E[Y_{y}^2]}{\E[Y_{y}]^2} = 1.\label{eq:Paley}\]for some $y=H_C+o(1).$ To do so we first study this second moment to first moment squared ratio for a general $y \in (0,1/2).$

The function $\tilde{G}$ defined below is going to be of crucial importance.

\begin{definition}
Given $x \in (0,1)$, $y \in (0,1/2)$, $y ' \in [y,1]$ and $H_C$ from Definition \ref{def:H_C}, define
\begin{align}\label{Gprimetilde}
\tilde{G}(y, y', x) &= \frac{x}{2}D(H_C||1/2) + y' D\left(y/y'||2^{-(1-x)}\right) - D(y||1/2) + \frac{1}{2}D(y'||2^{-x}).
\end{align}
\end{definition}

Using the definition of $\tilde{G}$, we can derive the following.

\begin{lemma}\label{lem:2mm bound 1}
\stateScale Given $Y_y$ from Definition \ref{def:countFlat}, $\tilde{G}$ from \eqref{Gprimetilde}, $y \in (0,1/2)$, and $y_{(l)}$ from Definition \ref{def:SubsetProp}, define for $\ell=1,\ldots,k-1$,
\[S_l=S_l(y) \coloneqq \{y' : y' M \in [y_{(l)} M - D_l, y_{(l)} M +  D_l] \cap \{0,\dots,M\}\}.\]
Then the following statement holds.

If $\alpha$ and $C$ satisfy Assumption \ref{as:alphaC_0} then there exists $\Co{9+eO(1)} > 0$ such that for a sufficiently large $n$,
    \[\label{eq:2mm1Goal}\begin{split}\frac{\E[Y^2_y]}{\E[Y_y]^2} -1 &\leq o(1) + \Co{9+eO(1)}M \sum_{l = 1}^{k-1}\sum_{y ' \in S_l}\exp\left(-2M\tilde{G}(y, y', l/k) - l \log(l/k) + O(l) \right) \\&\qquad+ 2(M+1)\exp\left({M(\log(2)-h(y)) - k\log(p/k)}\right)\end{split}.\]

\end{lemma}

We can see in the above theorem that the error term in the exponent is at most order $k$. Under Assumption \ref{as:MPScale} we have that $M = \Theta(N) = \Theta(k \log n)$, meaning that terms with order $M$ will be of leading order in the exponent. This gives us hope that if we set $y$ to be a suitable perturbation of $H_C$ then the $O(k)$ term will be negligible.

\subsection{Simplifying the upper bound in Lemma \ref{lem:2mm bound 1}}Unfortunately, the upper bound in Lemma \ref{lem:2mm bound 1} remains complicated to work directly with. We now explain how to further simplify it for an appropriate choice of $y$ of interest.

Let us start with the last term in the upper bound: $2(M+1)\exp({M(\log(2) - h(y)) - k\log(p/k)})$. In order for our second moment method argument to succeed, this term must be $o(1)$ for our choice of $y$. 

Say one sets $y = H_C + \Co{Cprime}\kpert$. For any $\Co{Cprime} > 0$, by mean value theorem, we have for some $z \in (H_C, H_C + \Co{Cprime}\kpert)$ that 
\[h(H_C + \Co{Cprime}\kpert) = h(H_C) + \log\left(\frac{1-z}{z}\right)\Co{Cprime}\kpert.\]
Recall that $H_C = h^{-1}_2(2- 2/C)$, meaning $0 < H_C < 1/2$ as $1< C < 2$. As such $\log\left(\frac{1-z}{z}\right) > 0$ for large enough $n$, Thus, using that for all $x \in (0,1)$ it holds $h(x)=\log(2)h_2(x)$, we have, with $\Co{pertEnt} = \log\left(\frac{1-z}{z}\right)\Co{Cprime} > 0$, that
\begin{align}\exp\left[M(\log(2)-h(y))\right] & \leq \exp\left(M(\log(2)-\log(2)h_2(H_C) - \Co{pertEnt}\kpert)\right) \\
    &= \exp\left[M\log(2)(1-(2-2/C)) - M\Co{pertEnt}\kpert)\right]\\
    &\leq \exp\left[(1+N^{-\err})\log(2)(1/C - 1/2)N- \Co{pertEnt}\Theta(k)\right],
\end{align}
where the last line is due to the upper and lower bounds on $M$ from Assumption \ref{as:MPScale}. Similarly, by Assumption \ref{as:MPScale}, Lemma \ref{lem:plower} and Lemma \ref{lem:Mbounds},
\begin{comment}
\begin{align}
    k\log(p/k) &\geq k\log\left((1-k^{-\err})n(k/n)^{(C/2)(1+k^{-\err})} \right) = k\left((1-\alpha)(1-C/2) - O(k^{-\err})\right)\log(n)\\
    &= \frac{2\log(2)(1-C/2)}{C}k\frac{C}{2\log(2)}\log\left(\frac{ne}{k}\right) - O(k) - O(k^{1-\err})\log(n) \\
    &\geq \log(2)(2/C -1)\E[M] - O(k) - O(k^{1-\err})\log(n),
\end{align}
\end{comment}
\begin{align}
    k\log(p/k) &\geq k(1-\alpha)\left(1-\frac{C}{2}\right)\log(n) - O(k^{1-\err}\log(n))\\
    &= \frac{2\log(2)(1-C/2)}{C}\left(\frac{Ck(1-\alpha)\log(n)}{2\log(2)}\right) - O(k^{1-\err}\log(n))\\
    &\geq \frac{2\log(2)(1-C/2)}{C}(N/2) - O(k)\\
    &= \log(2)(1/C - 1/2)N - O(k),
\end{align}
We can then finally see that 
\begin{align}
    &2(M+1)\exp\left(M(\log(2)-h(y)) - k\log(p/k)\right) \\&\leq 3M\exp\big[(1+N^{-\err})\log(2)(1/C - 1/2)N- \Co{pertEnt}\Theta(k) - \log(2)(1/C -1/2)N + O(k)\big]\\
    &= 3M\exp\left[O\left(N^{1-\err}\right) + O(k) - \Co{pertEnt}\Theta(k)\right]\\
    &= o(1),
\end{align}
where we used  $k=n^{\Omega(1)}, N^{1-\err} = O(k^{1-\err}\log(n)^{1-\err}) = O(k)$ and choose $\Co{Cprime}$ to be sufficiently large  (thus making $\Co{pertEnt}$ sufficiently large) in the last line. This motivates the value of $H_C + \Co{Cprime}\kpert$, with $\Co{Cprime}$ sufficiently large, as a potential candidate for a choice of $y$ in Lemma \ref{lem:2mm bound 1} and proves the following Lemma.
\begin{lemma}\label{lem:y=1o(1)}
    \stateScale If $y = H_C + \Co{Cprime}\kpert$, then for a sufficiently large $\Co{Cprime}$ \[ 2(M+1)\exp\left[{M(\log(2)-h(y)) - k\log(p/k)}\right] = o(1).\]
\end{lemma}
    
Now, defining $y_*= H_C + \Co{Cprime}\kpert$, a Taylor expansion around $y=y^*$ gives that $\frac{x}{2}D(H_C||1/2) = \frac{x}{2}D(y||1/2) + \mathcal{P}_l$ for some ``controlled'' perturbation term $\mathcal{P}_l > 0.$ This leads us to the study of a surrogate function $G$ (instead of $\tilde{G}$), independent of the value of $H_C$, which plays an important role in our technical analysis.

\begin{definition}
Given $x \in (0,1)$, $y \in (0,1/2)$, $y ' \in [y,1]$, define
    \begin{align}\label{Gprime}
G(y, y', x) &= \frac{x}{2}D(y||1/2) + y' D\left(y/y'||2^{-(1-x)}\right) - D(y||1/2) + \frac{1}{2}D(y'||2^{-x})
\end{align}
\end{definition}

\begin{remark}
Notice that the functions $G$ and $\tilde{G}$ (and $\breve{G}$, to be defined later) are defined on the domain $x = (0,1)$ as the summation of interest in Lemma \ref{lem:2mm bound 1} corresponds only to terms $x=l/k$ for $l \in \{1, \dots, k-1\}$. 
\end{remark}

For such a choice of $y_*$, with $\Co{Cprime}$ chosen sufficiently large, we get the far simpler upper bound on the second to first moment squared ratio.

\begin{lemma}\label{lem:2mmbound3}
\stateScale Let $Y_y$ be from Definition \ref{def:countFlat}, $G$ from \eqref{Gprime}, $H_C$ from Definition \ref{def:H_C}, $S_l$ from Lemma \ref{lem:2mm bound 1}.

Considering $y_* = H_C + \Co{Cprime}\kpert$, and denoting $S^*_l=S_l(y^*)$, i.e.,
\[S^*_l := \{y': y' M \in [(y_*)_{(l)} M - D_l, (y_*)_{(l)} M + D_l] \cap \{0,\dots, M\},\] 
the following statement holds:

    Consider an $(\alpha, C)$ \kset instance, if $\alpha$ and $C$ satisfy Assumption \ref{as:alphaC_0} then
    \[\frac{\E[Y^2_{y_*}]}{\E[Y_{y_*}]^2} -1 \leq o(1) + \Co{9+eO(1)}M \sum_{l = 1}^{k-1}\sum_{y ' \in S^*_l}\exp\left(-2MG(y_*, y', l/k) - l \log(l/k) \right)\label{eq:gammaC'bound}\]
    for sufficiently large $n$.
\end{lemma}

\begin{comment}

We can even further simplify the above bound by upper bounding the difference between $\mathcal{E}_l$ and $\mathcal{P}_l$ from Lemma \ref{lem:2mmbound2} up to leading order.

\begin{lemma}\label{lem:err goes away}
    \stateScale Given $\mathcal{E}_l$, $\mathcal{P}_l$ from Lemma \ref{lem:2mmbound2}, the following holds.
    
    If $\Co{Cprime}$ is chosen sufficiently large, then for sufficiently large $n$, we have that $\mathcal{E}_l - \mathcal{P}_l \leq 0$. \textcolor{red}{0??}
\end{lemma}

Invoking Lemma \ref{lem:err goes away} we can get a new, substantially cleaner, upper bound on the second to first moment squared ratio.

\begin{lemma}\label{lem:2mmbound3}
    \stateScale Given $\Co{Cprime} \geq 0$, $y_*$ from Lemma \ref{lem:2mmbound2}, $H_C$ from Definition \ref{def:H_C}, $G$ from \eqref{Gprime}, we have, for a sufficiently large $\Co{Cprime}$, that
    \[\begin{split}\frac{\E[Y^2_{y_*}]}{\E[Y_{y_*}]^2} -1 &\leq o(1) +\Co{9+eO(1)}M \sum_{l=1}^{k-1}\sum_{y ' \in S_l}\exp\bigg{(}-2MG(y_*, y', l/k) - l \log(l/k)\bigg{)}.\end{split}\]
\end{lemma}
\end{comment}

\subsection{Auxiliary Lemmas About $G$}\label{subsub:G}

In Lemma \ref{lem:2mmbound3} we have reduced the upper bound on $1-\Phi_K$ (similarly an upper bound on $\phi(0)$) to proving that the upper bound in Lemma \ref{lem:2mmbound3} is $o(1)$. In order to accomplish this goal, we must first study some specific properties of the function $G$. All proofs in this section are deferred to Appendix \ref{subsec:proofG}. 

During this study it is natural to study $G(y, y', x)$ at $y' = y_{(x)}$, the center of the interval $S_l$ for $x=l/k$.
\begin{definition}
Given $G$ from \eqref{Gprime}, $y \in [0,1/2)$, $x \in (0,1)$ and $y_{(x)}$ from Definition \ref{def:SubsetProp} define
\[\label{eq:g-eps}\begin{split}
\Breve{G}(y, x) \coloneqq G(y, y_{(x)}, x) &= \frac{x}{2}D(y||1/2) + (\gl) D\left(\frac{y}{\gl}\|\|2^{-(1-x)}\right)\\
&\qquad - D(y||1/2) + \frac{1}{2}D(\gl||2^{-x}).
\end{split}\]
\end{definition}

A key ingredient for our proof is that we establish that $G$ is strictly positive over $x \in (0,1)$ and on the boundary of $x$ we have that $\partial_xG$ is non-vanishing. To prove this we first control the difference between $G$ and $\Breve{G}$ by bounding the derivative of $G$ with respect to $y'$.

\begin{lemma}\label{DerPrime}
Given $G$ from \eqref{Gprime}, $y_{(x)}$ from Definition \ref{def:SubsetProp}, $x \in (0,1)$, $y \in (0,1/2)$, $y' \in [y,1]$, the following two statements hold.

\begin{enumerate}
    \item \[\sup_{x \in (0,1)} \bigg{|}\left[ \partial_{y'}G(y, y', x)\right]|_{y'=y_{(x)}}\bigg{|} \leq \frac{1}{2}\log\left(\frac{1-y }{y}\right)\]
    \[\inf_{x \in (0,1)} \bigg{|}\left[ \partial_{y'}G(y, y', x)\right]|_{y'=y_{(x)}}\bigg{|} \geq \frac{1}{2}\log(2(1-y))\]
    \item For a fixed $x \in (0,1)$,
\[\bigg{|}\left[\partial_{y'}G(y, y', x)\right]|_{y'=y_{(x)}}\bigg{|} \leq \frac{1}{2}\log(2(1-y)) + x\frac{\log(4)}{2^{2-x} - 2}\]
\end{enumerate}
\end{lemma}

As the ratio of $D_l/M$ is vanishing, we would hope that for any $y' \in [y_{(x)} - D_l/M, y_{(x)} + D_l/M]$ that $G(y, y', x)$ is sufficiently close to $G(y, y_{(x)}, x)$. The following Lemma provides such a one-sided guarantee.

\begin{lemma}\label{lem:justifyLowerGamma'}
    Given $G$ from \eqref{Gprime}, $x \in (0,1)$ and $y \in (0,1/2)$, we have the following statement:

    For any $\epsilon > 0$ there exists a sufficiently large $n$ such that, if $y' \in [y_{(x)} - \frac{D_l}{M}, y_{(x)} + \frac{D_l}{M}]$ then for all $x \in (0,1)$,
    \[G(y, y', x) \geq G(y, y_{(x)}, x) - (1+\epsilon)\frac{D_l}{M}[\partial_{y'}G(y, y', x)]|_{y' = y_{(x)}}.\]
\end{lemma}

Now that with the help of Lemma \ref{DerPrime} and Lemma \ref{lem:justifyLowerGamma'} we have the necessary tools to control the value of $G$ over our region $S^*_l$ by controlling the value of $\breve{G}$. We first prove that for all $x \in (0,1)$ and $y \in (0,1/2)$ we have that $\breve{G}(y, x) > 0$. The first step to this result is to identify its limiting value on the boundary of $x \in (0,1)$.
\begin{lemma}\label{zerobound}
    Given $\Breve{G}$ from \eqref{eq:g-eps}, if $y \in (0,1/2)$ and $x' \in \{0,1\}$ then $\lim_{x \conv{} x'}\Breve{G}(y, x) = 0.$
\end{lemma}
Another important property is that the derivative of $\breve{G}$ with respect to $x$ is strictly positive and strictly negative as $x \conv{} 0$ and $x \conv{} 1$ respectively, implying that the function $\breve{G}$ is positive locally around the points $x = 0$ and $x = 1$.
\begin{lemma}\label{posder}
Given $\Breve{G}$ from \eqref{eq:g-eps}, we have the following statement:

If $y \in (0, 1/2)$, then
\[\lim_{x' \conv{} 0}[\partial_{x}\Breve{G}(y, x)]{|}_{x = x'} = \log(2)\left((1-y)(1-\log(2-2y)) - \frac{h_2(y)}{2}\right)\] 
and
\[\lim_{x' \conv{} 1}[\partial_{x}\Breve{G}(y, x)]{|}_{x = x'} = \log(2)\left((1-y)\left(1+\frac{1}{2}\log\left(\frac{y}{1-y} \right)\right) - \frac{h_2(y)}{2}\right).\]
In particular, for all $y \in (0,1/2)$, $\lim_{x' \conv{} 0}[\partial_{x}\Breve{G}(y, x)]|_{x = x'} > 0$ and $\lim_{x' \conv{} 0}[\partial_{x}\Breve{G}(y, x)]|_{x = x'} < 0$.
\end{lemma}

Using both Lemma \ref{zerobound} and Lemma \ref{posder}, the condition that $\Breve{G}$ is strictly concave in $x \in (0,1)$ is sufficient to prove that $\Breve{G}(y, x) > 0$ for all $x \in (0,1)$ and $y \in (0,1/2)$. The following lemma establishes this result.

\begin{lemma}\label{positive}
Given $\Breve{G}$ from \eqref{eq:g-eps}, if $y \in (0,1/2)$ then $\Breve{G}(y,x)$ is strictly concave over $x \in (0,1)$. Moreover, $\Breve{G}(y,x) > 0$ for all $x \in (0,1), y \in (0,1/2)$. In particular, whenever $x,y$ are bounded away from 0 and 1, $\Breve{G}(y,x)=\Omega(1)$.
\end{lemma} 

\subsection{Using Properties Of $G$ and $\Breve{G}$ Derive The Limiting Value Of $1-\Phi_k$}\label{subsec:PhikLimitingValue}

We are now in a position to show that the upper bound in Lemma \ref{lem:2mmbound3} is vanishing.

\begin{lemma}\label{lem:o(1)sum}

\stateScale Given $\Co{Cprime}$, $y_*=y_*(\Co{Cprime})$ from Lemma \ref{lem:2mmbound3}, $y_{(l)}$ from Definition \ref{def:SubsetProp}, and $S^*_l$ from Lemma \ref{lem:2mmbound3}, 
if $\alpha$ and $C$ satisfy Assumption \ref{as:alphaC_0}, then
\[M\sum_{l = 1}^{k-1}\sum_{y' \in S^*_l}\exp\left({-2MG(y_*, y', l/k) + l\log(k/l)}\right) = o(1).\]
\end{lemma}

Before we prove Lemma \ref{lem:o(1)sum}, we show that it completes the proof of upper bound on the maximal number of sets left uncovered by any $k$-subset $\sigma$. This gives an upper bound on $1- \Phi_k$ under Assumption \ref{as:MPScale} that matches the first moment method lower bound $1-\Phi_k \geq H_C - o(1)$ as described in Section \ref{subsec:relate} and specifically Lemma \ref{lem:lowerbound} at the point $l = 0$.

\begin{lemma}\label{thm:upperbound}
    \stateScale Given $\phi$ from \eqref{eq:phil}, $\Phi_k$ as the maximal number of sets covered by a set of size $k$ in \kset and $y_* = H_C + \Co{Cprime}\kpert$ from Lemma \ref{lem:2mmbound3},
    if $\alpha$ and $C$ satisfy Assumption \ref{as:alphaC_0} and $\Co{Cprime}$ is sufficiently large, then for an $(\alpha, C)$ \kset instance we have,
    \[1 - \Phi_k \leq y_* = H_C + o(1).\]
\end{lemma}

\begin{proof}[Proof of Lemma \ref{thm:upperbound}, assuming Lemma \ref{lem:o(1)sum}]
    Assumptions \ref{as:alphaC_0} and \ref{as:MPScale} are necessary to get our initial upper bound on the second to first moment squared ratio in Lemma \ref{lem:2mm bound 1}. Recalling $Y_y$ from Definition \ref{def:countFlat}, using \eqref{eq:obvious}, we can see proving $Y_y > 0$ under Assumption \ref{as:MPScale} implies that a $\sigma$ leaves $yM$ sets uncovered. To show that $Y_{y_*}>0$ under Assumption \ref{as:MPScale}, we utilize Paley-Zygmund inequality as in Section \ref{eq:paley-Section}. Thus, we are left to show for $y_* = H_C + \Co{Cprime}\kpert$ that
    \[\lim_{n \conv{} \infty}\frac{\E[Y_{y_*}^2]}{\E[Y_{y_*}]^2} = 1,\label{eq:PZ}\]
    when $(M, p)$ satisfy Assumption \ref{as:MPScale} and $\Co{Cprime}$ is chosen sufficiently large. Applying Lemma \ref{lem:2mmbound3} to gain a further upper bound on the second to first moment squared ratio and invoking Lemma \ref{lem:o(1)sum} to show the upper bound from Lemma \ref{lem:2mmbound3} is $o(1)$, we have $\frac{\E[Y_{y_*}^2]}{\E[Y_{y_*}]^2} - 1 = o(1)$ for $y_* = H_C + \Co{Cprime}\kpert$, implying \eqref{eq:PZ}, giving the result.
\end{proof}

Combining the upper and lower bounds on $1- \Phi_k$ we have that when $(\mSet, \pSet)$ satisfy Assumption \ref{as:MPScale} and $(\alpha, C)$ satisfy Assumption \ref{as:alphaC_0} then $1- \Phi_k = H_C + o(1)$ implying that $\Phi_k = 1 - H_C + o(1)$ giving the proof of Theorem \ref{thm:Phi_kLim}.

\subsection{The Proof of Lemma \ref{lem:o(1)sum}}

%We can now give the proof of the vital Lemma \ref{lem:o(1)sum}, a visual aid is given in Figure \ref{fig:aid}.
%\begin{figure}[htb]
%    \centering
%    \includegraphics[scale = .6]{figs/G Function.pdf}\\
%    \caption{ \textcolor{red}{IZ: confusing, what is this integral upon?}A visual aid for the proof technique of Lemma \ref{lem:o(1)sum}, on the $x$-axis we have each of our 4 cases and their given boundaries. We then have a lower bound for the $\breve{G}$ function, from \eqref{eq:g-eps}, given by a dotted red line. We have included $\breve{G}'(y_*,0)x$ and $\breve{G}'(y_*,1)(1-x)$ for reference.\maxwell{I removed the hueristic part because it was just confusing, I am going to edit this picture to update the argument last.}}
%    \label{fig:aid}
%\end{figure}

\begin{proof}[Proof of Lemma \ref{lem:o(1)sum}]

     Recall $y_{(x)}$ from Lemma \ref{lem:gammaepsexp}. Below, we recollect our results on $G$ and $\Breve{G}$ from the previous section. We summarize the following results for the reader's convenience.
    \begin{enumerate}[(a)]
        \item By Lemma \ref{posder}, the derivative of $\Breve{G}(y, x)$ with respect to $x$ is positively bounded away from zero as $x \conv{} 0$ and negatively bounded away from zero as $x \conv{} 1$. By a slight abuse of notation\footnote{Although $\breve{G}'_y(x)$ is not defined for $x \in \{0,1\}$ as $\breve{G}$ has domain $x \in (0,1)$, we can consider the limiting value of the derivative for $x \conv{} 0$ or $x \conv{} 1$ respectively. It is easy to check (see the proof of Lemma \ref{posder}) that $\breve{G}'_y(x)$ can be continuously extended to $x \in [0,1]$. Thus, for any $\epsilon > 0$, there exists a $n$ large enough such that $\breve{G}'_y(1/k) \geq \breve{G}'_y(0) - \epsilon$.}, we denote the limits of these derivatives as $\Breve{G}'_y(0)$ and $\Breve{G}'_y(1)$ respectively. 
        \item By Lemma \ref{positive}, $\Breve{G}(y, x)$ is strictly concave in $x$ and bounded positively away from zero for $x \in (0,1)$. Moreover, by Lemma \ref{zerobound}, $\lim_{x \conv{} x'}\breve{G}(y, x) = 0$ for $x' \in \{0,1\}$.
        \item By Lemma \ref{DerPrime} we have the following uniform upper bound which we denote as $C_d$,
    \[C_d \coloneqq \sup_{x \in (0,1)}\bigg{|}[\partial_{y'}G(y, y', x)]|_{y'=y_{(x)}}\bigg{|} = \frac{1}{2}\log\left(\frac{1-y }{y}\right) \label{eq:unifDerBound}.\]
        \item By Lemma \ref{DerPrime}, we also have a bound on the derivative for a fixed $x \in (0,1)$, with
    \[\bigg{|}[\partial_{y'}G(y, y', x)]|_{y'=y_{(x)}}\bigg{|} \leq \frac{1}{2}\log(2(1-y)) + x\frac{\log(4)}{2^{2-x} - 2}.\label{eq:DerPrimeNear0}\]
        \item Thus, for any region $(0,x']$, $x' \in (0,1)$, we know by the monotonicity of the right-hand side of \eqref{eq:DerPrimeNear0} in $x$ that 
    \[C_d^{x'} \coloneqq \sup_{x \in (0,x']}\bigg{|}[\partial_{y'}G(y, y', x)]|_{y'=y_{(x)}}\bigg{|} \leq \frac{1}{2}\log(2(1-y)) + x'\frac{\log(4)}{2^{2-x'} - 2} \label{eq:unifDerBound0}.\]
    Moreover, if $0\leq x_1 \leq x_2 \leq 1$ then $C_d^{x_1} \leq C_d^{x_2}$.
    \end{enumerate}
    
    For brevity, we denote $\Breve{G}(y_*,x)$ as $\Breve{G}_{y_*}(x)$ ,$\partial_x \Breve{G}(y_*, x)$ as $\Breve{G}'_{y_*}(x)$ and define
    \[s_{y, y', x} \coloneqq \exp\left[-2MG(y, y', x) + x k \log(x)\right].\]
    Our goal is to show, for any choice of $\Co{Cprime} > 0$, that
    \[M \sum_{1 \leq l \leq k-1}\sum_{y' \in S^*_l} s_{y_*, y', l/k} = o(1).\label{eq:theSum}\]
    To do so, we decompose the sum in \eqref{eq:theSum} over $l$ (parameterized as $x = l/k$) into three regions and showing that the sum over each region is vanishing. 
    %(See Figure \ref{fig:aid} for pictorial representation of these regions)
    
    %\textbf{Case 1}: $x = 1$
    
    \textbf{Case 1}: $x \leq \zeta_1$, $x \ne 0$, for some $n$-independent $\zeta_1 \in (0,1)$ appropriately be chosen later
    
    \textbf{Case 2}: $x \geq \zeta_2$, $x \ne 1$, for some $n$-independent $\zeta_2 \in (0,1)$ appropriately be chosen later
    
    \textbf{Case 3}: $x \in [\zeta_1, \zeta_2]$
    \vspace{.5cm}
     
    \hspace{-.5cm}We begin with the first case,
    
%    \textbf{Case 1}, $x = 1$ or equivalently $l=k$.
    
%    By Lemma \ref{lem:D_linterval}, we have that $D_k = 0$, making $S_k = \{y_*\}$. Thus,
%    \[M\sum_{y' \in S_k} s_{y_*, y', 1} = M\cdot s_{y_*, y_*, 1}.\]
%    Setting $\Co{Cprime} = \frac{1.03}{\log(\frac{1-H_C}{H_C})}$, Lemma \ref{lem:err goes away} gives that $\Co{gap} \geq 1.02$. Using conditioning event $\scr{E}$ and $\alpha \leq 1$, we have that $M = O(k\log(n)) = o(n^{1.01})$. Combining this fact with Lemma \ref{zerobound}, giving $G(y_*, y_*, 1) = \Breve{G}(y_*,x) = 0$ \textcolor{red}{IZ: $x=1$ not well defined, don't get how you use a limit here...}, results in the upper bound for any $k \geq 1$,
%    \[Ms_{y_*, y, 1} \leq Me^{-\Co{gap}k^{7/8}\log(n)} \leq n^{-0.01} = o(1).\]
%    Define constants $\zeta_i = \Theta(1)$, $\delta_i = \Theta(1)$ for $i\in\{1,2\}$ which are chosen appropriately in the subsequent cases.
    
    \textbf{Case 1}, $x \leq \zeta_1$, $x \ne 0$:
    We start with a concavity argument. Consider a sufficiently small $\epsilon_1 > 0$. By the strict concavity of $\Breve{G}_{y_*}(x)$ and, by Lemma \ref{zerobound}, that $\lim_{x \conv{} 1}\Breve{G}_{y_*}(x) = 0$, the line $(\Breve{G}'_{y_*}(0) - \epsilon_1)x$ will intersect $\Breve{G}_{y_*}(x)$ at some point $(\zeta_1, \Breve{G}_{y_*}(\zeta_1))$. Notice that we can make $\epsilon_1$ sufficiently small so that $\zeta_1 \leq \delta$, where $\delta$ is from Lemma \ref{lem:cornerBound}. In particular, with $\delta_1 = \Breve{G}
    _{y_*}(\zeta_1)$, concavity gives that \[\Breve{G}_{y_*}(x) \geq \frac{\delta_1}{\zeta_1}x\] and \[\frac{\delta_1}{\zeta_1} + \epsilon_1 = \Breve{G}'_{y_*}(0).\] 
    Assumption \ref{as:MPScale} and Lemma \ref{lem:Mbounds} implies that for any $\epsilon_2 > 0$, for sufficiently large $n$, $\frac{C(1-\alpha)}{2\log(2)}k\log(n) \leq (1 + \epsilon_{2})M$, and therefore,
    \[
        l \log(k/l) = x k \log(k/l) \leq x k \log(k) \leq x\alpha k\log(n) \leq (1 + \epsilon_{2})x\frac{2\log(2)\alpha}{C(1-\alpha)}M.\label{eq:MBound1}
    \]
    Utilizing \eqref{eq:boundComplex} (as we have ensured that $\zeta_1 \leq \delta$, and can invoke Lemma \ref{lem:cornerBound}), Lemma \ref{lem:justifyLowerGamma'} and Lemma \ref{DerPrime} for $y' \in S_{xk}$, gives that, for $\epsilon_3 > 0$, there exists a sufficiently large $n$ where
    \[G(y_*, y', x) \geq \Breve{G}(y_*,x) - 7(1 + \epsilon_{3})x\sqrt{\frac{\alpha}{1-\alpha}}C_d^{x} \geq \Breve{G}(y_*,x) - 7(1 + \epsilon_{3})x\sqrt{\frac{\alpha}{1-\alpha}}C_d^{\zeta_1}\label{eq:GBound1}\]
    for each $x \leq \zeta_1$.
    Using \eqref{eq:MBound1}, and invoking Lemma \ref{lem:cornerBound} and Lemma \ref{lem:justifyLowerGamma'} we have
    \begin{align}
        s_{y_*, y', x} & \leq \exp\left({-2MG(y_*, y', x) + x(1+\epsilon_2)\frac{2\log(2)\alpha}{C(1-\alpha)}M} \right)\\
        &\leq \exp \left(-2M\left[G(y_*, y', x) - x (1+\epsilon_2)\frac{\log(2)\alpha}{C(1-\alpha)}\right]\right)\\
        &\leq \exp\left(-2M\left[\Breve{G}(y_*, x) - 7(1+\epsilon_3)\log(2)x\sqrt{\frac{\alpha}{1-\alpha}}C^{\zeta_1}_d -  x(1+\epsilon_2)\frac{\log(2)\alpha}{C(1-\alpha)}\right]\right)\label{eq:usedDerlem1}\\
        &\leq \exp\left(-2Mx\left[\frac{\delta_1}{\zeta_1}- 7(1+\epsilon_3)\log(2)\sqrt{\frac{\alpha}{1-\alpha}}C^{\zeta_1}_d - (1+\epsilon_2)\frac{\log(2)\alpha}{C(1-\alpha)}\right]\right).
    \end{align}
    Using $|S^*_l| \leq M$ in line \eqref{eq:usedbound1}, $x = l/k \geq 1/k$ in line \eqref{eq:eps>1/k}, and Assumption \ref{as:MPScale} and Lemma \ref{lem:Mbounds} in line \eqref{eq:Mupper1} with sufficiently large $n$, we have the contribution to our overall sum \eqref{eq:theSum} is, for any $\epsilon_2, \epsilon_3, \epsilon_4, \epsilon_5 > 0$,
    \begin{align}\hspace{-2cm}
    M \sum_{l\geq 1:\frac{l}{k}\leq\zeta_1}\sum_{y' \in S^*_l}s_{y_*, y', l/k}
        &\leq M^2ke^{\bigg{(}-2Mx\bigg{[}\frac{\delta_1}{\zeta_1}- 7(1+\epsilon_3)\log(2)\sqrt{\frac{\alpha}{1-\alpha}}C_d^{\zeta_1}  - (1+\epsilon_2)\frac{\log(2)\alpha}{C(1-\alpha)}\bigg{]}\bigg{)} }\label{eq:usedbound1}\\
        &\leq n^{3(1+\epsilon_4)\alpha}e^{\left(-2M\frac{1}{k}\left[\frac{\delta_1}{\zeta_1}- 7(1+\epsilon_3)\sqrt{\frac{\alpha}{1-\alpha}}C_d^{\zeta_1} - (1+\epsilon_2)\frac{\log(2)\alpha}{C(1-\alpha)}\right]\right)}\label{eq:eps>1/k}\\
    \label{eq:Mupper1}&\leq n^{3(1+\epsilon_4)\alpha}e^{\left(-(1-\epsilon_5)\log(n)\frac{C(1-\alpha)}{\log(2)}\left[\frac{\delta_1}{\zeta_1}- 7(1+\epsilon_3)\log(2)\sqrt{\frac{\alpha}{1-\alpha}}C_d^{\zeta_1} - (1 + \epsilon_2)\frac{\log(2)\alpha}{C(1-\alpha)}\right]\right)}\\
    &\leq n^{-\left[(1-\epsilon_5)\frac{C(1-\alpha)}{\log(2)}\left[\frac{\delta_1}{\zeta_1}- 7(1+\epsilon_3)\log(2)\sqrt{\frac{\alpha}{1-\alpha}}C_d^{\zeta_1} - (1+\epsilon_2)\frac{\log(2)\alpha}{C(1-\alpha)}\right] - 3(1+\epsilon_4)\alpha\right]}\\
    \end{align}
    We can see that the above bound is $o(1)$ as $n \conv{} \infty$ as long as we can choose $\zeta_1,\epsilon_i>0, i=1,\ldots,5$ so that
    \[(1-\epsilon_5)\frac{C(1-\alpha)}{\log(2)}\left[\frac{\delta_1}{\zeta_1}- 7(1+\epsilon_3)\log(2)\sqrt{\frac{\alpha}{1-\alpha}}C_d^{\zeta_1} - (1+\epsilon_2)\frac{\log(2)\alpha}{C(1-\alpha)}\right] - 3(1+\epsilon_4)\alpha > 0.\]
    Recalling that $\delta_1 / \zeta_1 +\epsilon_1 = \Breve{G}_{y_*}'(0)$ and rearranging terms we then have a sufficient condition for this contribution to the sum to be $o(1)$ is to choose $\zeta_1,\epsilon_i>0, i=1,\ldots,5$ so that
    \[(1-\epsilon_5)C\left[\frac{1}{\log(2)}\left(\Breve{G}_{y_*}'(0) - \epsilon_1\right)- 7(1+\epsilon_3)\sqrt{\frac{\alpha}{1-\alpha}}C_d^{\zeta_1}\right]> 4(1+\epsilon_2)(1+\epsilon_4)\alpha/(1-\alpha).\label{eq:equivCondition1}\]
One can rewrite 
    \[\Breve{G}'_{y_*}(0) = \Breve{G}'_{H_C + o(1)} = \Breve{G}'_{H_C}(0) + o(1) = \log(2)\left((1-H_C)(1-\log(2-2H_C)) - \frac{h_2(H_C)}{2}\right) + o(1)\] by Lemma \ref{posder} and the fact that $\breve{G}'_y(0)$ is continuous and has a bounded derivative for all $H_C > 0$ which holds true as we have assumed $C > 1$. Employing the last displayed equation, \eqref{eq:equivCondition1} is then directly implied by \eqref{eq:2mmCond1_0} in Assumption \ref{as:alphaC_0} for sufficiently small $\zeta_1, \epsilon_1, \epsilon_2, \epsilon_3, \epsilon_4, \epsilon_5 > 0$, yielding that the sum is $o(1)$ in this case.\hfill\\

    \textbf{Case 2}, $x \geq \zeta_2$,
    
    By the concavity of $\Breve{G}(y_*,x)$ in $x$, through a similar argument to Case 1, we consider $\epsilon_6 > 0$ and define $\zeta_2$ such that the line $(\breve{G}'_y(1) + \epsilon_6)x$ intersects $\Breve{G}_y(x)$ at the point $x = \zeta_2$. Define $\delta_2 = \Breve{G}(y_*, \zeta_2)$ and choose $\epsilon_6$ small enough so that $1-\zeta_2$ is sufficiently close to $1$ and as such for $\zeta_2 \leq x \leq 1$ both
    \[\Breve{G}(y_*,x) \geq \frac{\delta_2}{1-\zeta_2}(1-x)\]
    and
    \[\frac{\delta_2}{1-\zeta_2} = -\Breve{G}'(1) - \epsilon_6\]
    hold. Moreover, by having $\epsilon_6$ sufficiently small, we have $1-\zeta_2 \geq 1-\delta$ where $\delta$ is from Lemma \ref{lem:cornerBound}.
    Utilizing \eqref{eq:boundSimp} (as we have satisfied the conditions of Lemma \ref{lem:cornerBound}), Lemma \ref{lem:justifyLowerGamma'} and Lemma \ref{DerPrime} for $y' \in S_{xk}$, gives for any $\epsilon_7 > 0$, with $n$ sufficiently large, 
    \[G(y_*, y', x) \geq \Breve{G}(y_*,x) - 5(1+\epsilon_7)\log(2)(1-x)\sqrt{\frac{\alpha}{1-\alpha}}C_d. \label{eq:GBound2}\]
    Using \eqref{eq:GBound2} in \eqref{eq:usedDerBound2}, $\log(1+x) \leq x$, $n$ sufficiently large in \eqref{eq:needbign} and $(1 + \epsilon_{8})\frac{C(1-\alpha)}{2\log(2)}k\log(n) < M$ for any $\epsilon_8 > 0$ (from Assumption \ref{as:MPScale} and Lemma \ref{lem:Mbounds}), we have
    \begin{align}
        s_{y_*, y', x} &= \exp(-2MG(y_*, y', x) + l \log(k/l))\\
        &\leq \exp(-2MG(y_*, y', x) + l \log(1 + k/l - 1))\\
        &\leq \exp\left(-2M[\Breve{G}(y_*, x) - 5(1 + \epsilon_7)\log(2)(1-x)\sqrt{\frac{\alpha}{1-\alpha}}C_d] + l\left(\frac{k}{l}-1\right)\right)\label{eq:usedDerBound2}\\
        &\leq \exp\left(-2M\left[\Breve{G}(y_*, x) - 5(1+\epsilon_7)\log(2)(1-x)\sqrt{\frac{\alpha}{1-\alpha}}C_d\right] + (k-l)\right)\\
        &\leq \exp\left(-2M(1-x)\left[\frac{\delta_2}{\zeta_2} - 5(1+\epsilon_7)\log(2)\sqrt{\frac{\alpha}{1-\alpha}}C_d\right] + (k-l)\right)\\
        &= \exp\left(-2M\frac{k-l}{k}\left[\frac{\delta_2}{\zeta_2} - 5(1 + \epsilon_7)\log(2)\sqrt{\frac{\alpha}{1-\alpha}}C_d\right] + (k-l)\right)\\
        &\leq \exp\left(-(k-l)\log(n)(1-\epsilon_8)\frac{C(1-\alpha)}{\log(2)}\left[\frac{\delta_2}{\zeta_2} - 5(1 + \epsilon_7)\log(2)\sqrt{\frac{\alpha}{1-\alpha}}C_d\right] + (k-l)\right)\\
        &\leq \exp\left(-(k-l)\left[\log(n)(1-\epsilon_8)\frac{C(1-\alpha)}{\log(2)}\left[\frac{\delta_2}{\zeta_2} - 5(1+\epsilon_7)\log(2)\sqrt{\frac{\alpha}{1-\alpha}}C_d\right] - 1\right]\right)\\
        &\leq \exp\left(-\left[\log(n)(1-\epsilon_8)\frac{C(1-\alpha)}{\log(2)}\left[\frac{\delta_2}{\zeta_2} - 5(1+\epsilon_7)\log(2)\sqrt{\frac{\alpha}{1-\alpha}}C_d\right] - 1\right]\right).\label{eq:needbign}
    \end{align}
    Thus, we bound this case's contribution to \eqref{eq:theSum}, for any $\epsilon_9 > 0$, as
    \begin{align}
        \hspace{-1.5cm}M \sum_{l \geq 1 : 1 \ne \frac{l}{k} \geq \zeta_2}\sum_{y' \in S^*_l} s_{y_*, y', l/k} &\leq  M^2 k e^{\left(-\log(n)(1-\epsilon_8)\frac{C(1-\alpha)}{\log(2)}\left[\frac{\delta_2}{\zeta_2} - 5(1+\epsilon_7)\log(2)\sqrt{\frac{\alpha}{1-\alpha}}C_d\right] + 1\right)}\\
        & \leq e n^{3(1+\epsilon_9) \alpha}n^{-(1-\epsilon_8)\frac{C(1-\alpha)}{\log(2)}\left[\frac{\delta_2}{\zeta_2} - 5(1+\epsilon_7)\log(2)\sqrt{\frac{\alpha}{1-\alpha}}C_d\right]}\\
        & \leq e n^{-\left((1-\epsilon_8)\left(\frac{C(1-\alpha)}{\log(2)}\left[\frac{\delta_2}{\zeta_2} - 5(1+\epsilon_7)\log(2)\sqrt{\frac{\alpha}{1-\alpha}}C_d\right]\right) - 3(1+\epsilon_9)\alpha\right)}.
    \end{align}
    A sufficient condition for when this bound is $o(1)$ if for sufficiently small $\zeta_2,\epsilon_i>0, i=6,\ldots,9$ it holds
    \[(1-\epsilon_8)\left(\frac{C(1-\alpha)}{\log(2)}\left[\frac{\delta_2}{\zeta_2} - 5(1 + \epsilon_7)\log(2)\sqrt{\frac{\alpha}{1-\alpha}}C_d\right]\right) - 3(1+\epsilon_9)\alpha > 0.\]
    Recalling that $\delta_2 / (1-\zeta_2) = -\Breve{G}_{y_*}'(1) - \epsilon_6$ and rearranging terms we then have a sufficient condition for this contribution to the sum to be $o(1)$ is that for sufficiently small $\zeta_2,\epsilon_i>0, i=6,\ldots,9$ it holds
    \[(1-\epsilon_8)C\left[\frac{1}{\log(2)}\left(-\Breve{G}_{y_*}'(1) - \epsilon_6\right)- 5(1+\epsilon_7)\log(2)\sqrt{\frac{\alpha}{1-\alpha}}|C_d|\right]> 3(1+\epsilon_9)\alpha/(1-\alpha).\label{eq:equivCondition2}\]
    By similar arguments as in the previous case and using the fomula of $\Breve{G}'_y(1)$ from Lemma \ref{posder}, for $C > 1$, we have that 
    \[-\Breve{G}'_{y_*}(1) = -\Breve{G}'_{H_C + o(1)} = - \log(2)\left((1-y)\left(1+\frac{1}{2}\log\left(\frac{y}{1-y} \right)\right) - \frac{h_2(y)}{2}\right) + o(1).\]
    Recall $C_d = \frac{1}{2}\log\left( \frac{1-y}{y}\right)$, as $y\mapsto \frac{1}{2}\log\left( \frac{1-y}{y}\right)$ is continuous and has a bounded derivative for all $y \in (0,1)$, we invoke the mean value theorem to give that $C_d$ is equivalent up to $o(1)$ factors when we have $y = H_C$ in lieu of $y = y_* = H_C + o(1)$.  Under these two conditions \eqref{eq:equivCondition2} is equivalent to \eqref{eq:2mmCond2_0} in Assumption \ref{as:alphaC_0} for a sufficiently small $\zeta_2, \epsilon_6, \epsilon_7, \epsilon_8, \epsilon_9$ a.a.s. as \ngrow\! and hence the contribution of this case to the sum is $o(1)$ as well.\hfill\\
    
    \textbf{Case 3}, $x \in (\zeta_I, \zeta_2)$:
    
    Since we have that $x$ is bounded away from $0$ and $1$ in this case, a combination of Lemma \ref{lem:cornerBound} and Lemma \ref{positive} gives the following:
    \[D_l/M = O\left(\frac{1}{\sqrt{\log(n)}} \right) = o(1),\]
    \[l \log(k/l) = x k \log(1/x) = O(k) = o(M),\]
    \[\Breve{G}(y_*, x) \geq \min\{\Breve{G}(y_*, \zeta_1),\Breve{G}(y_*, \zeta_2)\} = \Omega(1).\]
    Thus, using Lemma \ref{lem:justifyLowerGamma'} in \eqref{eq:usedDerBound3} and the above order bounds in \eqref{eq:usedOrder} we can bound the summand for any $\epsilon_{10} > 0$ with
    \begin{align}
        s_{y_*, y', l/k} &= \exp(-2M G(y_*, y', x) + l \log(k/l))\\
        &\leq \exp(-2M G(y_*, y', x) + l \log(k/l))\\
        &\leq \exp\left(-2M \left[\Breve{G}(y_*, x) - (1 - \epsilon_{10})C_d\frac{D_l}{M}\right] - l \log(k/l) \right)\label{eq:usedDerBound3}\\
        &\leq \exp(-2M(1-o(1))[\min\{\Breve{G}(y_*, \zeta_1),\Breve{G}(y_*, \zeta_2)\}])\label{eq:usedOrder}\\
        & =o(1)
    \end{align}
    for a sufficiently large $n$.
    
    Thus, the contribution to the sum \eqref{eq:theSum} in this case is also
    \[M\sum_{l \geq 1: \frac{l}{k} \in (\zeta_I, \zeta_2)} \sum_{y ' \in S^*_l} s_{y_*, y', l/k}\leq M^2 k \exp(-2M[\min\{\Breve{G}(y_*, \zeta_1),\Breve{G}(y_*, \zeta_2)\}]) = o(1)\]
    
    \textbf{Putting it all together}\hfill\\
    
   Combining the above, we bound the sum \eqref{eq:theSum} under any choice of $\Co{Cprime}$ in Case 1 and conditions \eqref{eq:2mmCond1_0} and \eqref{eq:2mmCond2_0}, as
    \begin{align}\sum_{l \geq 1}\sum_{y' \in S^*_l} s_{y_*, y', x} &= \sum_{l \geq 1: l/k \leq \zeta_I}\sum_{y' \in S^*_l} s_{y_*, y', x} + \sum_{l \geq 1: l/k \in (\zeta_I, \zeta_2)}\sum_{y' \in S^*_l} s_{y_*, y', x}\\
    & \qquad + \sum_{l \geq 1: 1 \ne l/k \geq \zeta_2}\sum_{y' \in S^*_l} s_{y_*, y', x}\\
    &= o(1) + o(1) + o(1)\\
    &= o(1).\end{align}
    
    This concludes the proof.
    \end{proof}

%-------------------------

\section{The Proof Of Theorem \ref{lem:FMFnonMono_0}}\label{sec:monotone}

\begin{proof}[Proof of Theorem \ref{lem:FMFnonMono_0}]
  By Lemma \ref{lem:FMFnear0}, Assumption \ref{as:FMFExists} implies there exists an $\epsilon' > 0$ such that the function $y(x)$ exists and is unique for any $x \in [0,\epsilon']$, and $y(x)$, on this interval, is continuously differentiable.

    For this proof only, we introduce the notation $O_x(1)$ and $o_x(1)$. Let $A(x) = O_x(1)$ and $B(x) = o_x(1)$ if, for some $\Co{O_x(1)} > 0$, $\lim_{x \downarrow 0} \frac{A(x)}{x} \leq \Co{O_x(1)}$ and $\lim_{x \downarrow 0}\frac{B(x)}{x} = 0$. Using Lemma \ref{lem:M:limit}, \ref{lem:p:limit}, \ref{lem:solZeroRigor} and a union bound, we condition on the a.a.s. as \ngrow event for some sequence $\gamma_n = o(1)$ and constant $1/4 > \err > 0$,
    \[\label{eq:conditionH_C}\begin{split}
        &\left\{p \geq (1-k^{-\err})n\left(k/n\right)^{(1+k^{-\err})C/2} \right\} \bigcap\; \left\{M \leq (1+
        N^{-\err})\frac{N}{2}\right\} \bigcap \left\{|y(0) - H_C| \leq \gamma_n\right\}.
    \end{split}\]
    Define $C_a \coloneqq 2a\log(2)$ and $f(y,x) \coloneqq (1-y)D\left(\frac{C_a x}{1-y}\|\|r(x) \right) + D(y||s(x))$. By Definition \eqref{def:FMF}, 
    \[
        f(y(x), x) - f(y(0), 0) = \frac{\log\left(\binom{k}{l}\binom{p-k}{k-l}\right) - \log\left(\binom{k}{0}\binom{p-k}{k-0}\right)}{M} = \frac{\log\left(\frac{\prod_{i=1}^{l}(k-i+1)^2}{l!\prod_{i=1}^l(p-2k+i)}\right)}{M}.
    \]
    Thus, applying the mean value theorem on $f$ restricted to the line connecting $(y(0), 0)$ to $(y(x),x)$, we have for some $y^* \in (y(0),y(x))$ and $x^* \in (0,x)$ that for $l=\floor{xk},$
    \[\left[\frac{\der}{\der y}f(y,x)\right]\bigg{|}_{(x,y) = (x^*,y^*)} (y(x) - y(0)) + \left[\frac{\der}{\der x}f(y,x)\right]\bigg{|}_{(x,y) = (x^*, y^*)}x = \frac{\log\left(\frac{\prod_{i=1}^{l}(k-i+1)^2}{l!\prod_{i=1}^l(p-2k+i)}\right)}{M}\label{eq:MVT}\]
    As $y(x)$ is continuously differentiable, we have $y^* = y(0) + O_x(1)$. We can then calculate that
    \begin{align}
        \frac{\der}{\der x}f(y,x) &= \left(\frac{1-y}{1-s(x)} - \frac{y}{s(x)}\right)s'(x) \\
        &\qquad + (1-y)\left( \frac{C_a}{1-y}\log\left(\frac{\frac{C_a x}{1-y}(1-r(x))}{\left(1-\frac{C_a x}{1-y}\right)r(x)} \right) + \left(\frac{1-\frac{C_a x}{1-y}}{1-r(x)} - \frac{\frac{C_a x}{1-y}}{r(x)} \right)r'(x)\right)\\
        &= \left(\frac{1-y}{1-s(x)} - \frac{y}{s(x)}\right)s'(x) \\
        &\qquad + C_a\log\left(\frac{x}{r(x)}\frac{C_a(1-r(x))}{(1-y)\left(1-\frac{C_a x}{1-y}\right)} \right) + r'(x)\left(\frac{1-y-C_a x}{1-r(x)} - C_a\frac{x}{r(x)} \right)\label{eq:derWRTx}
    \end{align}
    As $s(x) = 1-2^{x-1}$ and $r(x) = 4 \cdot 2^{-x}(1-2^{-x})$ are both continuously twice differentiable we have that for any $x^* \in [0,x]$ that $s(x^*) = s(0) + O_x(1)$, $r(x^*) = r(0) + O_x(1)$, $r'(x^*) = r'(0) + O_x(1)$ and $s'(x^*) = s'(0) + O_x(1)$. We also have that $r(x^*) = r(0) + x^*r'(0) + o_{x}(1)$, meaning that 
    \[\frac{x^*}{r(x^*)} = \frac{x^*}{x^* \cdot r'(0) + o_{x}(1)} = \frac{1}{r'(0)} + O_x(1).\]
    Plugging all of these results into \eqref{eq:derWRTx}, for any $x \in [0,x]$ and $y \in [0,y(x)]$
    \begin{align}\begin{split}\frac{\der}{\der x}f(y^*,x^*) &= \left(\frac{1-y(0)}{1-s(0)} - \frac{y(0)}{s(0)} + O_x(1)\right)(s'(0) + O_x(1)) \\
        &\qquad + C_a \log\left(\left(\frac{1}{r'(0)} + O_x(1)\right)\frac{C_a}{1-y(0)}(1-r(0)) + O_x(1) \right) \\
        &\qquad + (r'(0) + O_x(1))\left(\frac{1-y(0) + O_x(1)}{1-r(0)} - \frac{C_a}{r'(0)} + O_x(1) \right)\end{split}\\
        \begin{split}&= \left(\frac{1-y(0)}{1-s(0)} - \frac{y(0)}{s(0)}\right)s'(0)  \\
        &\qquad + C_a \log\left(\frac{C_a(1-r(0))}{r'(0)(1-y(0))}\right)+ r'(0)\left(\frac{1-y(0)}{1-r(0)} - \frac{C_a}{r'(0)}\right)+ O_x(1)\end{split}
        \end{align}
    As $s(0) = 1/2, r(0) = 0, r'(0) = 4\log(2), s'(0) = -\log(2)/2$, we have that
    \[\begin{split}
    \frac{\der}{\der x}f(y^*,x^*) &= -\log(2)\left(1-2y(0)\right)  + C_a \log\left(\frac{C_a}{4\log(2)(1-y(0))}\right)  \\
        &\qquad  + 4\log(2)\left(1-y(0) - \frac{C_a}{4\log(2)}\right) + O_x(1).
    \end{split}\]
    Substituting back $C_a = 2a\log(2)$ gives,
    \[\label{eq:derX}\begin{split}
        \frac{\der}{\der x}f(y^*,x^*) &= -\log(2)\left(1-2y(0)\right) + 2a\log(2) \log\left(\frac{a}{2(1-y(0))}\right)  \\
        &\qquad  + 4\log(2)\left(1 - y(0) - \frac{a}{2}\right) + O_x(1).
    \end{split}\]
    Next, we find the derivative of $f$ with respect to $y$,
    \begin{align}
        \frac{\der}{\der y}f(y,x) &= (1-y)\left[\frac{\der}{\der y} D\left(\frac{C_a x}{1-y}\|\|r(x)\right)\right] - D\left(\frac{C_a x}{1-y}\|\|r(x)\right) + \frac{\der}{\der y} D(y||s(x))\\
        \begin{split}&= (1-y)\frac{C_a x}{(1-y)^2}\log\left(\frac{\frac{C_a}{1-y}(1-r(x))}{\left(1-\frac{C_ax}{1-y}\right)}\frac{x}{r(x)}\right) - D\left(\frac{C_a x}{1-y}\|\|r(x) \right) + \log\left(\frac{y(1-s(x))}{(1-y)s(x)} \right)\\
    &= x\frac{C_a}{1-y}\log\left( \frac{C_a(1-r(x))}{(1-y)\left(1-\frac{C_a}{1-y} \right)}\frac{x}{r(x)}\right) - D\left(\frac{C_a x}{1-y}\|\|r(x)\right) + \log\left(\frac{y(1-s(x))}{(1-y)s(x)} \right).\end{split}
    \end{align}
    Similar to the derivative with respect to $x$, we consider $y^* = y(0) + O_x(1)$, $s(x^*) = s(0) + O_x(1)$, $r(x^*) = r(0) + r'(0)x + o_x(1)$, $x^*/r(x^*) = 1/(r'(0)) + o_x(1)$, $r(0) = 0$, $r'(0) = 4\log(2)$, $s(0) = 1/2$, and $\frac{C_a x}{1-y} = \frac{C_ax}{1-y(0)} + o_x(1)$. Furthermore, denoting $x^* = C_* x$ for some $C_* \in (0,1)$, by the definition of $o_x(1)$, for any sufficiently small $\epsilon > 0$, we have that as $x \downarrow 0$,
    \[D\left(\frac{C_a}{1-y(0)} x^* + o_x(1) ||4\log(2)x^* + o_x(1) \right) \leq\!\!\!\!\!\max_{\epsilon_1, \epsilon_2 \in [-\epsilon, \epsilon]}\!\!\!\!\!D\left(\frac{C_a}{1-y(0)}(C_* + \epsilon_1)x||4\log(2)(C_* + \epsilon_{2})x\right) \]
    By the monotonicity of KL divergence, a single point of $(\epsilon_1, \epsilon_2)$ on the boundary of $-[\epsilon, \epsilon]$ will be the maximizer, for such a point we invoke Lemma \ref{lem:KLLim} to give that \[D\left(\frac{C_a}{1-y(0)} x^* + o_x(1) ||4\log(2)x^* + o_x(1) \right) = O_x(1).\] Using the above collection of facts, we calculate
    \begin{align}
    \begin{split}\frac{\partial}{\partial y}f(y^*,x^*)&= O_x(1)\left[\frac{C_a}{1-y(0)}\log\left(\frac{C_a}{4\log(2)(1-y(0))}\right) + O_x(1)\right] \\
    &\qquad - D\left(\frac{C_a}{1-y(0)} x^* + O_x(1) || 4\log(2)x^* + O_x(1) \right) + \log\left(\frac{y(0)}{1-y(0)}\right) + O_x(1)\end{split}\\
    &= \log\left(\frac{y(0)}{1-y(0)}\right) + O_x(1)\label{eq:DerY}.
    \end{align}
    Utilizing \eqref{eq:derX} and \eqref{eq:DerY} in \eqref{eq:MVT} gives
    \begin{align}
         \frac{\log\left(\frac{\prod_{i=1}^{l}(k-i+1)^2}{l!\prod_{i=1}^l(p-2k+i)}\right)}{M} &= \left[\log\left(\frac{y(0)}{1-y(0)}\right) + O_x(1)\right] (y(x) - y(0)) + \bigg{[}-\log(2)\left(1-2y(0)\right) \\
         &\qquad + 2a\log(2) \log\left(\frac{a}{2(1-y(0))}\right) + 4\log(2)\left(1 - y(0) - \frac{a}{2}\right) + O_x(1)\bigg{]}x,
    \end{align}
    and upon a rearrangement of terms,
    \[\label{eq:boundDiff}\begin{split}
        \hspace{-1cm}&\left(\log\left( \frac{y(0)}{1-y(0)}\right) + O_x(1) \right)(y(x) - y(0))\\ &= \frac{\log\left(\frac{\prod_{i=1}^{l}(k-i+1)^2}{l!\prod_{i=1}^l(p-2k+i)}\right)}{M} + x\bigg{[}\log(2)\left(1-2y(0)\right) \\
        &\qquad - 2a\log(2) \log\left(\frac{a}{2(1-y(0))}\right) - 4\log(2)\left(1 - y(0) - \frac{a}{2}\right)\bigg{]} + o_x(1).
    \end{split}\]
    Analyzing the combinatorial term via direct algebraic manipulations and recalling $l = \floor{xk}$, gives
    \[\frac{\log\left(\frac{\prod_{i=1}^{l}(k-i+1)^2}{l!\prod_{i=1}^l(p-2k+i)}\right)}{M} \leq \frac{\log\left(\frac{k^{2l}}{(p-2k)^l}\right)}{M} \leq x\left[\frac{k}{M}\log\left(\frac{k^2}{p}\right) - \log\left(1-\frac{2k}{p}\right)\right] \leq x\left[-\frac{k}{M}\log\left(\frac{p}{k^2}\right)\right] + o_x(1).
    \]
    Combining this display with \eqref{eq:boundDiff}, we have
    \[\label{eq:preDerEquation}\hspace{-2cm}\begin{split}
    &\left(\log\left( \frac{y(0)}{1-y(0)}\right) + O_x(1) \right)(y(x) - y(0)) \\ 
    &\leq x\bigg{[}-\frac{k}{M}\log\left(\frac{p}{k^2}\right) + \log(2)\left(1-2y(0)\right) \\
        &\qquad- 2a\log(2) \log\left(\frac{a}{2(1-y(0))}\right)  - 4\log(2)\left(1 - y(0) - \frac{a}{2}\right)\bigg{]} + o_x(1).
    \end{split}\]

 Now by constraint \eqref{eq:s:constraint}, we have $y(0) < 1/2$. Moreover, using $H_C > 0$ (as $C \in (1,2)$) and the conditioned event \eqref{eq:conditionH_C} we have that $y(0) > 0$ for sufficiently large $n$. Thus, for large enough $n$, there exists a small enough $x > 0$ and constant $\Co{y(0)} > 0$ such that, $-\infty < \left(\log\left( \frac{y(0)}{1-y(0)}\right) + O_x(1) \right) \leq -\Co{y(0)}$. Hence, for any sufficiently small (yet non-vanishing) $x=\Omega(1)$,
    rearranging terms in \eqref{eq:preDerEquation} gives a.a.s. as \ngrow that
    \[\begin{split}
    \label{eq:derequation}y(x) - y(0) &\geq \frac{1}{\left(\log\left( \frac{y(0)}{1-y(0)}\right) + O_x(1) \right)}x\bigg{[}-\frac{k}{M}\log\left(\frac{p}{k^2}\right) + \log(2)\left(1-2y(0)\right)\\
    &\qquad - 2a\log(2) \log\left(\frac{a}{2(1-y(0))}\right) - 4\log(2)\left(1 - y(0) - \frac{a}{2}\right)\bigg{]} + o_x(1).
    \end{split}\]
    Next, we evaluate the limiting values of both $y(0) = H_C + o(1)$ and $\frac{k}{M}\log(\frac{p}{k^2})$ by using the conditioned event \eqref{eq:conditionH_C}. We have that
    \begin{align}
        \frac{p}{k^2} &\geq \frac{(1-k^{-\err})n\left(\frac{k}{n}\right)^{(1 + k^{-\err})C/2}}{k^2}\\
        &= (1-k^{-\err})n^{1 - (1+k^{-\err})C/2}k^{(1+k^{-\err})C/2 - 2}\\
        &\geq (1-k^{-\err})n^{1 - (1+k^{-\err})C/2}n^{\alpha[(1+k^{-\err})C/2 - 2]}(1-n^{-\alpha})^{(1+k^{-\err})C/2 - 2}\\
        &= (1-k^{-\err})n^{1 - (1+k^{-\err})C/2 + \alpha[(1+k^{-\err})C/2 - 2]}(1-n^{-\alpha})^{(1+k^{-\err})C/2 - 2},
    \end{align}
    and thus, 
    \begin{align}
        \log(p/k^2) &\geq \log((1-k^{-\err})n^{1 - (1+k^{-\err})C/2 + \alpha[(1+k^{-\err})C/2 - 2]}(1-n^{-\alpha})^{(1+k^{-\err})C/2 - 2})\\
        &\geq \log(1-k^{-\err}) + (1 - (1+k^{-\err})C/2 + \alpha[(1+k^{-\err})C/2 - 2]) \log(n) \\
        &\qquad+ ((1+k^{-\err})C/2 - 2)\log(1-n^{-\alpha})\\
        &= (1-C/2 + \alpha C/2 - 2\alpha)\log(n) + O(n^{-\alpha}) + O(k^{-\err}\log(n)).\label{eq:pordBound}
    \end{align}
    Moreover, using Lemma \ref{lem:Mbounds} and the conditioned event \eqref{eq:conditionH_C},
    \begin{align}
        \frac{k}{M} \geq \frac{k}{(1+N^{-\err})\left(\frac{Ck(1-\alpha)\log(n)}{2\log(2)} + O(k)\right)} \geq \frac{2\log(2)}{C(1-\alpha)\log(n) + O(1)}.\label{eq:MordBound}
    \end{align}
    Combining \eqref{eq:pordBound} and \eqref{eq:MordBound} gives,
    \begin{align}
    \frac{k}{M}\log\left(\frac{p}{k^2}\right) &\geq 2\log(2)\frac{(1-C/2 + \alpha C/2 - 2\alpha)\log(n) + O(n^{-\alpha}) + O(k^{-\err}\log(n))}{C(1-\alpha)\log(n) + O(1)}\\
    &= 2\log(2)\frac{(1-C/2)(1-\alpha) - \alpha}{C(1-\alpha)} + o(1)\\
    &= 2\log(2)\left[\frac{1-\frac{C}{2}}{C} - \frac{\alpha}{C(1-\alpha)}\right] + o(1).\label{eq:CombEta}
    \end{align}
    Plugging in both \eqref{eq:CombEta} and $H_C = y(0) + o(1)$ into condition \eqref{eq:derequation} gives,
    \begin{align}
    \begin{split}\label{eq:posDerCondition}
      y(x) - y(0)&\geq x\bigg{[}2\log(2)\left[\frac{1-\frac{C}{2}}{C} - \frac{\alpha}{C(1-\alpha)}\right] - \log(2)\left(1-2H_C\right) \\
    &\qquad + 2a\log(2) \log\left(\frac{a}{2(1-H_C)}\right) + 4\log(2)\left(1 - H_C - \frac{a}{2}\right) + o(1)\bigg{]} + o_x(1).
    \end{split}
    \end{align}

     Now directly combining the above equation with Assumption \ref{as:der_0}, implies, for a sufficiently large $n$, that there exists a $\Co{Inc}$ such that for all sufficiently small $x$ with $\epsilon'\geq x>0$ it holds $y(x) - y(0) \geq \Co{Inc}x$. We then choose $\delta_1 = \Co{Inc}$ and set $\epsilon_1$ to be our sufficiently small choice of $x$. Observing that this argument will then hold for any $l/k \leq \epsilon_1 \leq \epsilon'$ completes the proof.

    %Thus, for sufficiently large $n$ (and correspondingly large $k$) and sufficiently small $x$ we have that $\max(O(k^{-3/16}) + o(1),  o_x(1)) \leq z$ for any $z > 0$. Meaning that the above equation becomes,
    %\[\begin{split}
    %    y(x) - y(0)&\geq x\bigg{[}2\log(2)\left[\frac{1-\frac{C}{2}}{C} - \frac{\alpha}{C(1-\alpha)}\right] - \log(2)\left(1-2H_C\right) \\
    %  &\qquad + 2a\log(2) \log\left(\frac{a}{2(1-H_C)}\right) + 4\log(2)\left(1 - H_C - \frac{a}{2}\right) - z\bigg{]}.
    %\end{split}\]

    %Thus, we are left to show that $2\log(2)\left[\frac{1-\frac{C}{2}}{C} - \frac{\alpha}{C(1-\alpha)}\right] - \log(2)\left(1-2H_C\right)  + 2a\log(2) \log\left(\frac{a}{2(1-H_C)}\right) + 4\log(2)(1 - H_C - \frac{a}{2}) > z$, for a sufficiently small $z > 0$. By rearranging terms we can see that this is equivalent to Assumption \ref{as:der_0}, allowing us to conclude that there exists a $\Co{Inc}$ where, for a sufficiently small $x > 0$, that results in $y(x) - y(0) \geq \Co{Inc}x$. We then choose $\delta_1 = \Co{Inc}$ and set $\epsilon_1$ to be our sufficiently small choice of $x$, upon observing that this argument will then hold for any $l/k \leq \epsilon_1$ we complete the proof.
    \end{proof}

%---------------------------

\section{The Proof Of Theorem \ref{thm:QualInc_0}}\label{sec:q}

\begin{proof}[Proof of Theorem \ref{thm:QualInc_0}]
    Using Assumption \ref{as:FMFExists} we invoke Lemma \ref{lem:FMFnear0} to conclude the existence of an $\epsilon > 0$ such that the first moment function $y(x)$ exists for all $x \in [0,\epsilon]$. 

    Setting $\epsilon_1 = \epsilon/3$ and $\epsilon_2 = 2\epsilon/3$, Assumption \ref{as:FMFExists}, Assumption \ref{as:der_0} and Assumption \ref{as:alphaC_0} allows us to invoke Theorem \ref{lem:FMFnonMono_0} and Theorem \ref{thm:combinedBound} to give for some $\Co{Diff} > 0$ that, for any $l \in \{l:l/k \in (\epsilon_1, \epsilon_2)\}$, a.a.s. as \ngrow\!,
    \begin{align}
        \r{Using Theorem \ref{thm:combinedBound},}{\phi(l) - \phi(0) \;&{\geq\;} y(l/k) - y(0) - o(1)}\\
        \r{Using Theorem \ref{lem:FMFnonMono_0},}{&{\geq\;} \Co{Diff}\frac{l}{k}}
    \end{align}
    We can set $\delta = \Co{Diff}\epsilon_1 > 0$ to prove that for $l/k \in [\epsilon_1, \epsilon_2]$, we have $\phi(l) - \phi(0) \geq \delta$ a.a.s. as \ngrow\!. 
    Hence we can conclude by Theorem \ref{thm:combinedBound} that a.a.s. as \ngrow\! $$\min_{l:l/k \in [\epsilon_1, \epsilon_2]}\phi(l) \geq \phi(0) + \delta \geq y(0) + \delta/2.$$ Furthermore, choosing $\zeta_1 = \epsilon_1$, $\zeta_2 = \epsilon_2$ and $r = y(0) + \delta / 2$ gives the b-OGP since $\phi(k)=0$ and therefore, a.a.s. as \ngrow\!, it holds $\max\{\phi(0),\phi(k)\}=\phi(0) \leq r.$ 
\end{proof}

\subsection{Existence Of A Pair Satisfying Theorem \ref{thm:QualInc_0}}\label{subsec:qual}
     Recall that in the Figures \ref{fig:310}, \ref{fig:JustDerCond}, \ref{fig:PhiCond} we plotted the regions of $\alpha$ and $C$ such that the required Assumptions \ref{as:FMFExists}, \ref{as:der_0}, \ref{as:alphaC_0} for Theorem \ref{thm:QualInc_0} hold. In particular, numerically we can conclude that they holds for all $1<C<C^* \approx 1.4719$ if $\alpha>0$ is sufficiently small.
     
     In this small section, we prove analytically that if $\alpha>0$ is small enough and $C>1$ is sufficiently close to $1$ then Assumptions \ref{as:FMFExists}, \ref{as:der_0} and \ref{as:alphaC_0} hold.
    
    %Although such a limit will lead to $H_C \conv{} 0$, we can use the continuity of these conditions and consider a $(C,\alpha)$ sufficiently close to $(0,1)$ to have both $H_C > 0$ and the conditions of Theorem \ref{thm:QualInc_0} satisfied. 
    
    For $C \conv{} 1^+$ and $\alpha \conv{} 0^+$ it can easily be checked that:
    \begin{enumerate}
        \item for any $\epsilon > 0$, $a = 1 + \epsilon$ is a valid choice of $a$, as it will be in the set \eqref{eq:a:set}.
        \item $H_C \conv{} 0^+$, where $H_C$ is from Definition \ref{def:H_C}.
    \end{enumerate}
    
    \textbf{Assumption }\ref{as:FMFExists}: The first condition holds as we can choose $(a, \Co{roomR}, \Co{roomI})$ sufficiently close to $(1,0,0)$ respectively so that $D\left(1-\frac{a}{2(1-\Co{roomR})}\big{|}\big{|}\frac{1}{2}\right) \leq \delta$, for any desired $\delta > 0$. In particular, for any $C \in (1,2)$ we can guarantee $D\left(1-\frac{a}{2(1-\Co{roomR})}\big{|}\big{|}\frac{1}{2}\right) \leq (1-\Co{roomI})\frac{2-C}{C}\log(2)$. The second condition of $\frac{a}{2(1-\Co{roomR})} < 1$ is trivial since $a$ can be made arbitrarily close to one and we can choose $\Co{roomR} < 1/2$.
    
    \textbf{Assumption }\ref{as:der_0}: This assumption requires that $C$, $\alpha$, $a$ satisfy
    \[C < \frac{1-\frac{\alpha}{1-\alpha}}{a\left(1 - \log\left(\frac{a}{2(1-H_C)}\right)\right) + H_C - 1}.\]
    By continuity it suffices to plug in $C = 1$, $\alpha = 0$, $a = 1$ and $H_C = 0$, this inequality becomes $1< \frac{1}{(1+\log(2)) - 1}$ which obviously holds.

    \textbf{Assumption }\ref{as:alphaC_0}: This assumption requires that $\alpha < 28/1000$ which clearly we can satisfy. It also needs $C < 2\frac{1-2\alpha}{1-\alpha}$, which holds when $C = 1$ and $\alpha = 0$. We also need to satisfy the other two conditions:
    \[\hspace{-1cm}
    C\bigg{[}(1-H_C)(1-\log(2(1-H_C))) - \frac{h_2(H_C)}{2}-7\sqrt{\frac{\alpha}{1-\alpha}}\left(\frac{1}{2}\log(2(1-H_C)) \right)\bigg{]} > 4 \alpha/(1-\alpha)\]
    and 
    \[C\left[\frac{h_2(H_c)}{2} + \frac{1}{2}\log\left(\frac{1-H_C}{H_C} \right)\left( 1-H_C - 5\sqrt{\frac{\alpha}{1-\alpha}}\right) +H_C -1 \right] > 3 \alpha / (1-\alpha).\]
    By continuity, suffices to plug in $C = 1$, $\alpha = 0$, $a = 1$ and $H_C = 0$. Then the first inequality becomes $1-\log(2) > 0$ and the left-hand side of the second inequality becomes unbounded (as $H_C \conv{} 0$), meaning both these inequalities hold.

    Hence, we conclude the following statement. 
    
    \begin{theorem}\label{thm:01}
        There exists a $\delta > 0$ such that, for all $\alpha \leq \delta$ and $C \leq 1 + \delta$, the conditions of Theorem \ref{thm:QualInc_0} are satisfied for a sufficiently small $\Co{roomR}>0$.
    \end{theorem}

%----------------------------

\section{The Proof Of Corollary \ref{thm:mix}}\label{sec:mix_pf}

Before we give the proof of Corollary \ref{thm:mix}, we introduce the definition of a $T$-bottleneck.

\begin{definition}\label{def:Tbottle}
    Given a Markov chain with stationary distribution $\pi$, define the set $\scr{B}$ as a $T$-bottleneck if
    \[\frac{\pi(\partial \scr{B})}{\pi(\scr{B})} \leq 1/T,\]
    where $\partial \scr{B} \coloneqq \{b \in \scr{B}: b \text{ can transition to }\scr{B}^C \text{ in one step}\}$
\end{definition}

\begin{proof}[Proof of Corollary \ref{thm:mix}]
The stationary distribution is proportional to $e^{-\beta H(\sigma)}$. Consider the choice of $\beta \geq  C_{\epsilon}k\log(p/k)$ with $C_{\epsilon} > 0$ to be chosen later and the choice of $\epsilon_1 > 0$ inducing the event \[\scr{B} := \{\sigma: |\sigma| = k ,  |\sigma \cap \sigma^*| \leq \epsilon_1 k\}\]with the value of $\epsilon_1$ being the value guaranteed by Theorem \ref{thm:QualInc_0}, so that for $l =\epsilon_1 k$ it holds $\phi(l)\geq \phi(0)+\delta$ for some constant $\delta>0$ given also in Theorem \ref{thm:QualInc_0}. Note that, without loss of generality, by slightly perturbing $\epsilon_1$ if necessary, we assume here $\epsilon_1 k \in \mathbb{Z}$ and of course $\epsilon_1=\Omega(1).$

%First, we will show that $\pi_\beta(\scr{B}) \leq 1/2$. Using $H(\sigma^*) = 0$ in line \eqref{eq:sigmaIssigma*} and Theorem \ref{thm:QualInc_0}, we have
%\begin{align}
%    \P(\scr{B}) &= \frac{\sum_{ |\sigma \cap \sigma^* | \leq \epsilon_1 k} e^{-\beta H(\sigma)}}{\sum_{\sigma} e^{-\beta H(\sigma)}} \leq \frac{\sum_{| \sigma \cap \sigma^* | \leq \epsilon_1 k} e^{-\beta H(\sigma)}}{1}\label{eq:sigmaIssigma*}\\
  %  &\leq \sum_{l : l/k \in [0,\epsilon_1]}\sum_{\sigma: | \sigma \cap \sigma^* | = l}e^{-\beta H(\sigma)} \leq \epsilon_1 k \binom{p}{k} e^{-\beta \delta} \leq \epsilon_2 k \left( e\frac{p}{k}\right)^k e^{-\beta \delta}\\
  %  &= \epsilon_2 k e^{k\log(ep/k) - \beta \delta}.
%\end{align}
%We can consider $C_{\epsilon} \geq \frac{2.01}{\delta}$, and thus $\beta \geq \frac{2.01}{\delta}k\log(e\frac{p}{k})$, giving
%\[\lim_{n \conv{} \infty} \P(\scr{B}) \leq \lim_{k \conv{} \infty} \epsilon_2 k e^{-k\log(ep/k)} = 0.\]

%Thus, for sufficiently large $n$ (and thus $k$), $\P(\scr{B}) \leq 1/2$. 

As $\partial \scr{B} = \left\{\sigma: \frac{|\sigma \cap \sigma^*|}{k} = \epsilon_1\right\}$, we have that a.a.s. as \ngrow\!
\begin{align}
\frac{\pi_\beta(\partial \scr{B})}{\pi_\beta(\scr{B})}&= \frac{\pi_\beta\left(\frac{|\sigma \cap \sigma^*|}{k} =\epsilon_1\right)}{\pi_\beta\left( \frac{|\sigma \cap \sigma^*|}{k} \leq \epsilon_1\right)} = \frac{\sum_{\sigma: | \sigma \cap \sigma^* | =\epsilon_1 k}e^{-\beta H(\sigma)}}{\sum_{\sigma:  | \sigma \cap \sigma^* | \leq \epsilon_1 k}e^{-\beta H(\sigma)}} \leq \frac{\binom{k}{\epsilon_1 k}\binom{p-k}{k-\epsilon_1 k} e^{-\beta \phi(\epsilon_1 k)}}{e^{-\beta \phi(0)}} \leq \binom{k}{\epsilon_1 k}\binom{p-k}{k-\epsilon_1 k} 
 e^{-\beta \delta}.\label{eq:boundryRatio}
\end{align}
 We can then calculate, 
\begin{align}
    \binom{k}{\epsilon_1 k}\binom{p-k}{k-\epsilon_1 k} &\leq \left( e\epsilon_1\right)^{\epsilon_1 k}\left( e\frac{p-k}{k(1-\epsilon_1)}\right)^{k(1-\epsilon_1)} \\
    &= \exp\left(\epsilon_1 k\log\left( e\epsilon_1 \right) + (1-\epsilon_1)k\left( e\frac{p-k}{k-l}\right)\right)
    = e^{(1+o(1))(1-\epsilon_1)k\log(p/k)}.\label{eq:combCalc}
\end{align}
Plugging in \eqref{eq:combCalc} into \eqref{eq:boundryRatio}, we get that a.a.s. as \ngrow\!,
\begin{align}
    \frac{\pi_\beta\left(\frac{|\sigma \cap \sigma^*|}{k}=\epsilon_1\right)}{\pi_\beta\left( \frac{|\sigma \cap \sigma^*|}{k} \leq \epsilon_1\right)} & \leq 2\exp(-\beta \delta + (1+o(1))(1-\epsilon_1)k\log(p/k)).
\end{align}
Now setting $\beta = C_\epsilon k\log(p/k)$ and choosing\footnote{Note that this value of $\delta$, guaranteed by Theorem \ref{thm:QualInc_0}, does rely on the value of $\epsilon_1$} $C_{\epsilon}$ such that $C_\epsilon \geq \frac{2.01}{\delta}$ gives that, for sufficiently large $n$,
\[\frac{\pi_\beta\left(\frac{|\sigma \cap \sigma^*|}{k} =\epsilon_1\right)}{\pi_\beta\left( \frac{|\sigma \cap \sigma^*|}{k} \leq \epsilon_1\right)}\leq  \exp(-k\log(p/k)).\]
This demonstrates that $\scr{B}$ is a $\exp(k\log(p/k))$-bottleneck  for any $\pi_\beta$ with \[\beta \geq \frac{2.01}{ \delta}k \log(ep/k).\] Thus, using standard results (e.g., \cite[Proposition 2.2]{arous2023free}) there exists an initialization for which the Markov chain requires at least $\exp(\Omega(k \log(p/k)))$ iterations to reach any $k$-subset $\sigma$ in $\mathcal{B}^c$, that is with $| \sigma \cap \sigma^* | \geq \epsilon_1 k.$
\end{proof}

\newpage
\medskip
\bibliographystyle{alpha}
\bibliography{refs}

\newpage
\medskip
\appendix

\section{Deferred Lemmas And Proofs From Prior Results}

\subsection{Commonly Used Auxiliary Lemmas And Proofs}\label{subsec:commonProofs}

\begin{proof}[Proof of Lemma \ref{lem:M:limit}]
We prove the upper bound for the statement, the lower bound follows similarly.

Recall that $M$, the number of positive tests, is a $\Bi{N, 1/2}$ random variable, meaning that $\E[M] = \frac{N}{2}$. Using a standard Chernoff bound on a binomial random variable, for any $\delta> 0$, we have that
\[\P\left(M \geq \frac{N}{2}(1 + \delta)\right) \leq e^{-\delta^2N/6}\]
For the left-hand side of the above equation to go to zero as $n$ (and thus $N$) grows we can set $\delta = N^{-\eta}$ for any $\eta \in (0,1/2)$. 
\end{proof}

\begin{proof}[Proof of Lemma \ref{lem:p:limit}] 
    We prove the lower bound for the statement, the upper bound follows similarly.

    Define the number of negative tests as $\bar{M} = N - M$. We condition on the event that $\bar{M} \leq (1+\delta)\frac{N}{2}$. Through the exact same calculation given in \cite[Section G.2]{iliopoulos2021group}, using $\alpha < 1/2$ in \eqref{eq:alpha<1/2}, $\binom{n}{k} \geq (n/k)^k$ in \eqref{eq:combBound}, and choosing $\delta = N^{-\eta}$ for $\eta \in (0,C/4) \subseteq (0,1/2)$ in \eqref{eq:delta<1/3}, we have 
    \begin{align}
        \E\left[p-k| \bar{M} \leq (1+\delta)\frac{N}{2}\right] &= (n-k)\left(1-\frac{q}{k}\right)^{\bar{M}}\\
        &= (n-k)2^{-\bar{M}/k}\\
        &\geq (n-k)2^{-\frac{C}{2}(1+\delta)\log_2\left(\frac{n}{k}\right)}\label{eq:combBound}\\
        &= (1-k/n) n \left(\frac{k}{n}\right)^{C/2(1+\delta)}\label{eq:alpha<1/2}\\
        &=\Omega\left(n^{1-(1+\delta)C/2}k^{(1+\delta)C/2}\right)\\
        &=\Omega\left(n^{1-(1+N^{-\eta})C/2}k^{(1+N^{-\eta})C/2}\right)\label{eq:delta<1/3}
    \intertext{As $C < 2$ there exits an $N$ large enough such that $(1 + N^{-\eta})C < 2$, and thus,}
        \E\left[p-k| \bar{M} \leq (1+\delta)\frac{N}{2}\right] &= \Omega(k^{C/2}).
    \end{align}
    \stepcounter{Ccnt}
  Using a standard Chernoff bound we then have that
    \[\P\left(p -k \leq (1-\delta)(1-k/n)n\left( \frac{k}{n}\right)^{\frac{C}{2}(1+\delta)}| \bar{M} \leq (1+\delta)\frac{N}{2} \right) \leq \exp\left( -\frac{\delta^2k^{C/2}}{3}\right)\label{eq:highprobP}\]
    Notice that the pre-factor in the above probability can be rewritten as
    \[(1-\delta)(1-k/n) = 1-\delta - k/n + \delta k/n.\]
    Since $\delta=N^{-\eta+o(1)}$ for some $\eta \in (0,C/4)$, $N=\Theta(k \log n)$ and $k=n^{\alpha+o(1)}$ for some $ \alpha>0$ it holds that $\delta = \omega(k^{-1/2})$. Hence, choosing $\alpha<1/3$ we have $\delta= \omega(k^{-1/2}) = \omega(n^{-2/3}) =\omega(k/n)$ and thus for some $\Co{delta} > 0$ we have that
    \[(1-\delta)(1-k/n) \geq 1-\Co{delta}\delta.\]
    Moreover, using $0 < \eta < C/4$ we have that there exists $\epsilon > 0$ such that $\eta = C/4 - \epsilon$, and thus since $N=\Theta(k \log n)$
    \[\delta^2k^{C/2} = N^{-C/2 + 2\epsilon}k^{C/2}  = \omega(k^\epsilon).\]
    This allows us to conclude from \eqref{eq:highprobP} that a.a.s. as \ngrow, $p -k \leq (1-\delta)(1-k/n)n\left( \frac{k}{n}\right)^{\frac{C}{2}(1+\delta)}$ conditioned on $\bar{M} \leq (1+\delta)\frac{N}{2}$. Using Lemma \ref{lem:M:limit}, the event that $\{\bar{M} \leq (1+\delta)\frac{N}{2}\}$ is a.a.s. as \ngrow for any choice of $\delta = N^{-\eta}$ for $\eta \in [0,1/2)$, giving the a.a.s. as \ngrow bound without the conditioning on $\bar{M}$. Moreover, we can get the bound from the statement from the Lemma with $k^{-\eta}$ instead of $N^{-\eta}$ by observing that for $\eta \in (0,C/4)$, we have $k^{-\eta} = \Omega(N^{-\eta})$. 
\end{proof}

\begin{lemma}\label{lem:KLBinomial}
    For $X \sim \Bi{n,p}$ it holds
    \[\frac{1}{3 \sqrt{n}}e^{-nD(k/n||p)} \leq \P(X \leq k) \leq e^{-nD(k/n||p)}\] and
    \[\frac{1}{3 \sqrt{n}}e^{-nD(k/n||p)} \leq \P(X \geq k) \leq e^{-nD(k/n||p)}\]
\end{lemma}

\begin{proof}[Proof of Lemma \ref{lem:KLBinomial}]
    See \cite[Lemma 4.7.2]{AshInfoTheory}
\end{proof}
    
\begin{lemma}\label{lem:D_3}
    For any $\delta > 0$ such that $a < (1-\delta)b$, there exists a $c_0 = c_0(\delta) > 0$ such that
    \[\frac{\der}{\der a}D(a||b) < -c_0.\]
    As a consequence, when $a_0 < a$, we have
    \[D(a - a_0 || b) \geq D(a||b) +c_0a_0.\]
\end{lemma}

\begin{proof}[Proof of Lemma \ref{lem:D_3}]
    See \cite[Lemma D.3]{iliopoulos2021group}
\end{proof}

\begin{lemma}\label{lem:KLLim} For any two constants $\Co{KL1} > 0$ and $\Co{KL2} > 0$, we have that as $x$ approaches zero from above
    \[D(\Co{KL1}x||\Co{KL2}x) = O(x).\]
\end{lemma}

\begin{proof}[Proof of Lemma \ref{lem:KLLim}]
    Consider the function $g(x) = D(\Co{KL1}x||\Co{KL2}x)$, we can immediately see that   \begin{align}
       \lim_{x \downarrow 0}g(x) &=  \lim_{x \downarrow 0}(\Co{KL1}x)\log\left( \frac{\Co{KL1}x}{\Co{KL2}x} \right) + (1-\Co{KL1}x)\log \left( \frac{1-\Co{KL1}x}{1-\Co{KL2}x} \right) \\
       &= \lim_{x \downarrow 0}\Co{KL1}x\log(\Co{KL1}/\Co{KL2}) + (1-\Co{KL2})\log\left(1+x\frac{\Co{KL2} - \Co{KL1}}{1-\Co{KL2}x}\right) = 0.
    \end{align}
    Moreover, we calculate
    \begin{align}
    \frac{\der}{\der x} g(x) &= \Co{KL1}\left(\log\left(\frac{\Co{KL1}x}{1-\Co{KL1}x}\right) - \log \left( \frac{\Co{KL2}x}{1-\Co{KL2}x}\right)\right) + \Co{KL2}\left(\frac{1-\Co{KL1}x}{1-\Co{KL2}x} - \frac{\Co{KL1}x}{\Co{KL2}x}\right)\\
    &= \Co{KL1}\log\left(\frac{\Co{KL1}(1-\Co{KL2}x)}{\Co{KL2}(1-\Co{KL1}x)}\right) + \Co{KL2}\left(\frac{1-\Co{KL1}x}{1-\Co{KL2}x} - \frac{\Co{KL1}}{\Co{KL2}}\right) 
    \end{align}
    this derivative remains uniformly upper bounded by a constant for sufficiently small $x$. An application of the mean value theorem gives the proof.
\end{proof}

\subsection{The Proof Of Lemma \ref{lem:D_linterval} And Lemma \ref{lem:cornerBound}}\label{subsec:proofD_linterval}

For reference, we paste below the value of $D_l$:
\[D_{l} = \sqrt{6p_l(1-p_l)(1-y)M \left[\log \binom{k}{l} + (1+\Co{D_l})\log k\right]}.\]

\begin{lemma}\label{lem:cornerBound}
    \stateScale Given $p_l$ from Lemma \ref{lem:p_l} and $D_l$ from Lemma Definition \ref{lem:D_linterval}, there exists a $n$ independent $\delta > 0$ such that the following holds with a sufficiently small $\Co{D_l}$ and sufficiently large $n$:
    \begin{align}\max_{l/k \leq \delta} \frac{D_l}{M}&\leq 7\log(2)\sqrt{\frac{\alpha}{(1-\alpha)}}\frac{l}{k},\label{eq:boundLow}\\
    \max_{l/k \geq 1-\delta}\frac{D_l}{M} &\leq 5\log(2)\sqrt{\frac{\alpha}{(1-\alpha)}}\left(1-\frac{l}{k}\right),\label{eq:boundHigh}\\
    \max_{l \in \{1, \dots, k\}}\frac{D_l}{M} &= O\left(\frac{1}{\sqrt{\log(n)}} \right).\label{eq:D_lUnif}\end{align}
    Moreover, if $\alpha < 28/1000$, then for all $l \in \{1, \dots, k-1\}$, we have for any $|x| \leq 2D_l/M$ with $n$ sufficiently large, that
    \[D(p_l + x||p_l) \geq \frac{x^2}{6p_l(1-p_l)}\label{eq:KLChanging}.\]
\end{lemma}

\begin{proof}[Proof of Lemma \ref{lem:cornerBound}]
        We first show \eqref{eq:D_lUnif}. As $0 \leq p_l \leq 1$ for all $l$, we have $p_l(1-p_l) \leq 1$. Moreover, we observe that
        \[
           0 \leq \max_{l \in \{1, \dots, k\}} \binom{k}{l} \leq \binom{k}{\lceil{k/2}\rceil} \leq \left(e\frac{k}{k/2}\right)^k = (2e)^k.
        \]
        Combining these two bounds gives for large enough $n$, $\max_{l \in \{1, \dots, k\}} \frac{D_l}{M} =O\left( \sqrt{\frac{k}{M}}\right)$. By Assumption \ref{as:MPScale} we then have that $M = \Theta(k\log(n))$ and thus also $\max_l \frac{D_l}{M} = O\left(\sqrt{\frac{1}{\log(n)}}\right)$.
        
        Next we can show \eqref{eq:boundLow}. For $n$ sufficiently large, we bound $D_l / M$ from above with
        \[\frac{D_l}{M} \leq \sqrt{\frac{6(2 + \Co{D_l})(1-y)p_l(1-p_l)\log\binom{k}{l}}{M}}.\label{eq:boundstart}\]
        By Assumption \ref{as:MPScale} and Lemma \ref{lem:Mbounds},
        \begin{align}
            M &\geq (1-N^{-c})\left(\frac{C(1-\alpha)k\log(n)}{2\log(2)} - O(1)\right) \geq (1-N^{-c})\left(\frac{Ck\log(n/n^{\alpha})}{2\log(2)} - O(1)\right) \\
            &\geq (1-N^{-c})\left(\frac{Ck\log(n/k)}{2\log(2)} - O(1)\right) \geq (1-\epsilon_1)\frac{C}{2\log(2)}k\log(n/k),
        \end{align}
        for any $\epsilon_1 > 0$ with $n$ large.
        Using the above lower bound, $0<y<1$, $p_l \leq 1$, $1-p_l = 2(1-2^{-l/k}) \leq 2\log(2)\frac{l}{k}$ (by recognizing the $1-p_l$ is concave in $l/k$, has value $0$ at $l/k = 0$ and first derivative with respect to $l/k$ of $2\log(2)$ at $l/k = 0$), $\log\binom{k}{l} \leq l \log(k)$, gives that,
        \begin{align}
            \frac{D_l}{M} &\leq \sqrt{\frac{6(2+\Co{D_l})2\log(2)\frac{l^2}{k}\log(k)}{M}}\\
            &\leq \sqrt{\frac{6(2+\Co{D_l})2\log(2)(l/k)^2 k\log(k)}{(1-\epsilon_1)\frac{C}{2\log(2)}k\log(n/k)}}\\
            &\leq \sqrt{\frac{6(2+\Co{D_l})2\log(2)(l/k)^2 k(\log(n^\alpha) + \log(1 + n^{-\alpha}))}{(1-\epsilon_1)\frac{C}{2\log(2)}k(\log(n/n^\alpha) - \log(1-n^{-\alpha}))}}\\
            & \leq \sqrt{\frac{24(2 + \Co{D_l})\log(2)^2}{C(1-\epsilon_1)}}\sqrt{\frac{\log(n^\alpha) + \log(1 + n^{-\alpha})}{{\log(n/n^\alpha) - \log(1-n^{-\alpha})}}}\frac{l}{k}\\
            &= \log(2)\sqrt{\frac{24(2+\Co{D_l})}{1-\epsilon_1}}\sqrt{\frac{\alpha}{(1-\alpha)} + o(1)}\frac{l}{k}.\label{eq:boundLeft}
        \end{align}
        This confirms \eqref{eq:boundLow} by choosing $\epsilon_1$ and $\Co{D_l}$ sufficiently small and then letting $n$ be~sufficiently~large.
        
        To show \eqref{eq:boundHigh}, we start from \eqref{eq:boundstart}, using $0<y<1$, $1-p_l \leq 1$, $\log\binom{k}{l} = \binom{k}{k-l} \leq (k-l)\log(k)$ and choosing $\delta > 0$ small enough such that $p_l \leq (1+\epsilon_2)\log(2)(1-\frac{l}{k})$ for any desired $\epsilon_2 > 0$ over $1 \geq l/k \geq 1-\delta$, gives that,
        \begin{align}
            \frac{D_l}{M} &\leq \sqrt{\frac{6(2+\Co{D_l})(1+\epsilon_2)\log(2)\frac{(k-l)^2}{k}\log(k)}{M}}\\
            &\leq \sqrt{\frac{6(2+\Co{D_l})(1+\epsilon_2)\log(2)(1-l/k)^2 k\log(k)}{(1-\epsilon_1)\frac{C}{2\log(2)}k\log(n/k)}}\\
            & \leq \sqrt{\frac{12(2 + \Co{D_l})(1+\epsilon_2)\log(2)^2}{C(1-\epsilon_1)}}\sqrt{\frac{\log(k)}{{\log(n/k)}}}\left(1-\frac{l}{k}\right)\\
            &= \log(2)\sqrt{\frac{12(2 + \Co{D_l})(1+\epsilon_2)}{1-\epsilon_1}}\sqrt{\frac{\alpha}{(1-\alpha)} + o(1)}\left(1-\frac{l}{k}\right)\\
            &\leq \log(2)\sqrt{\frac{12(1+\epsilon_2)(2 + \Co{D_l})}{1-\epsilon_1}}\sqrt{\frac{\alpha}{(1-\alpha)} + o(1)}\left(1-\frac{l}{k}\right)
            \label{eq:boundRight}.
        \end{align} 
        This confirms \eqref{eq:boundHigh} by choosing $\epsilon_1,\epsilon_2, \Co{D_l}$ small enough and $n$ sufficiently large.
        
        Finally, we demonstrate the lower bound on the KL divergence. Observe the identities,
        \begin{align}
            \frac{\partial}{\partial x} D(p + x||p)|_{x = 0} &= 0\\
            \frac{\partial^2}{\partial x^2} D(p + x||p)|_{x = 0} &= \frac{1}{p(1-p)}\\
            \frac{\partial^3}{\partial x^3} D(p + x||p) &= \frac{2(x+p)-1}{(x+p)^2(1-(x+p))^2}.
        \end{align}

        Consider $p \in (\epsilon_3, 1-\epsilon_3)$ for arbitrary $\epsilon_3$ with $1/4 > \epsilon_3 > 0$. As $2D_l/M = O(\log(n)^{-1/2})$, we have that $\frac{x^3}{6}\frac{2(x+p)-1}{(x+p)^2(1-(x+p))^2} = O(\log(n)^{-3/2})$ when $|x| \leq 2D_l/M$. Thus, by Taylor's theorem, when $n$ is sufficiently large we always have that $D(p + x||p) \geq \frac{x^2}{6p(1-p)}$ for $|x| \leq 2D_l/M$ for all $p \in (\epsilon_3, 1-\epsilon_3)$.

        We now consider the case of $p \leq \epsilon_3$, for any $x$ such that $-p \leq x \leq 0$, we have that
        \[\frac{x^3}{6}\frac{2(x+p)-1}{(x+p)^2(1-(x+p))^2} \geq 0,\]
        invoking Taylor's theorem then gives that $D(p+x||p) \geq \frac{x^2}{2p(1-p)}  \geq \frac{x^2}{6p(1-p)}$ for $x \leq 0$ and $p \leq \epsilon_3$. Taking the derivative of $\frac{1}{(x+p)^2(1-(x+p))^2}$ with respect to $x$ we get 
        \[\frac{2(1- 2x - 2p)}{(x + p - 1)^3(x + p)^3}.\label{eq:derpx}\]
        We can see when $2D_l/M \geq x \geq 0$, for sufficiently large $n$, that $(x + p)^3 \geq 0$, $(x + p - 1)^3 \leq 0$ and $1-2x-2p \geq 0$. Thus, the derivative \eqref{eq:derpx} is negative for all considered values of $x$ and $p$. This means that
        \[\max_{0 \leq \xi \leq 2D_l/M ,p \leq \epsilon_3}\|\frac{2(\xi+p)-1}{(\xi+p)^2(1-(\xi+p))^2}\| \leq \max_{0 \leq \xi \leq 2D_l/M ,p \leq \epsilon_3}\|\frac{1}{(\xi+p)^2(1-(\xi+p))^2}\| \leq \frac{1}{p^2(1-p)^2}.\]
        The first inequality is due to $\xi + p \in [0,1]$ (as it is an input into the two point KL divergence) and thus $|2(\xi + p) - 1| \leq 1$. The second inequality is due to the negativity of the derivative \eqref{eq:derpx} meaning that the maximum occurs at $\xi = 0$.
    
        Now considering $p \geq 1-\epsilon_3$, then for any $1-p \geq x \geq 0$, we have that,
        \[\frac{x^3}{6}\frac{2(x+p)-1}{(x+p)^2(1-(x+p))^2} \geq 0,\]
        invoking Taylor's theorem then gives that $D(p+x||p) \geq \frac{x^2}{2p(1-p)}  \geq \frac{x^2}{6p(1-p)}$ for $x \geq 0$ and $p \geq 1-\epsilon_3$. When $-2D_l/M \leq x \leq 0$, for sufficiently large $n$, we have that $(x + p)^3 \geq 0$, $(x + p - 1)^3 \leq 0$ and $1-2x-2p \leq 0$. Thus, the derivative \eqref{eq:derpx} is positive for all considered values of $x$ and $p$. This means that (by the same argument as above with reversed signs),-
        \[\max_{0 \geq \xi \geq -2D_l/M ,p \geq 1-\epsilon_3}\|\frac{2(\xi+p)-1}{(\xi+p)^2(1-(\xi+p))^2}\| \leq \max_{0 \geq \xi \geq -2D_l/M ,p \geq 1-\epsilon_3}\|\frac{1}{(\xi+p)^2(1-(\xi+p))^2}\| \leq \frac{1}{p^2(1-p)^2}.\]
        Combining these facts with Taylor's theorem gives the lower bound for $p \in [0,\epsilon_3] \cup [1-\epsilon_3,1]$ and $|x| \leq 2D_l/M$,
        \begin{align}
            D(p + x || p) &\geq \frac{x^2}{2p(1-p)} - \frac{|x|^3}{6}\frac{1}{p^2(1-p)^2}\\
            &\geq \frac{x^2}{p(1-p)}\left(\frac{1}{2} - \frac{|x|}{6p(1-p)} \right)
        \end{align}
        This bound elicits a sufficient condition for $D(p + x || p) \geq \frac{x^2}{6p(1-p)}$ when $\frac{1}{2} - \frac{|x|}{6p(1-p)} \geq \frac{1}{6}$, or equivalently, $|x| \leq 2p(1-p)$ holds. Thus, we can prove our lower bound on the two point KL divergence if $D_l/M \leq p_l(1-p_l)$. Continuing from line \eqref{eq:boundstart}, we have by Assumption \ref{as:MPScale} and Lemma \ref{lem:Mbounds} that $M \geq (1-\epsilon_4)\frac{C}{2\log(2)}k(1-\alpha)\log(n)$ for any $\epsilon_4 > 0$ for large $n$. We then use that $y \leq 1$ and $\binom{k}{l} \leq \min(l,k-l)\alpha \log(n)$ to give
        \begin{align}
        \frac{D_l}{M} &\leq \sqrt{\frac{6(2 + \Co{D_l})p_l(1-p_l)\min(l,k-l)\alpha\log(n)}{(1-\epsilon_4)\frac{C}{2\log(2)}k(1-\alpha)\log(n)}}\\
        &\leq \sqrt{\min(l/k,1-l/k)\frac{12(2+\Co{D_l})\log(2)\alpha}{(1-\epsilon_4)(C(1-\alpha))}}\sqrt{p_l(1-p_l)}.
        \end{align}
        This upper bound gives a sufficient condition for $D_l/M \leq p_l(1-p_l)$ to be the demonstration of the existence of some $\epsilon_4, \Co{D_l} > 0$ where
        \[\min(l/k,1-l/k)\frac{12(2+\Co{D_l})\log(2)\alpha}{(1-\epsilon_4)(1-\alpha)} \leq p_l(1-p_l).\]
        Recalling that $p_l = 2^{1-l/k} - 1$ and setting $l/k = z$ then we need to just show, for all $z \in [0,1]$, that
        \[\min(z,1-z)\frac{12(2+\Co{D_l})\log(2)\alpha}{(1-\epsilon_4)(1-\alpha)} \leq (2^{1-z}-1)(1-(2^{1-z}-1))\label{eq:ineqwhatever}.\]
        With $g(z) = \min(z,1-z)\frac{12(2+\Co{D_l})\log(2)\alpha}{(1-\epsilon_4)(1-\alpha)}$ and $h(z) =  (2^{1-z}-1)(1-(2^{1-z}-1))$, for $\epsilon_3, \Co{D_l}$ small enough and $\alpha < 28/1000$ we have,
        \begin{itemize}
        \item[(a)] $h(1/2)$ dominates the maximum of $g(z)$:
        \[\max_{z \in [0,1]} g(z) = \frac{1}{2}\frac{12(2+\Co{D_l})\log(2)\alpha}{(1-\epsilon_4)(1-\alpha)} \leq 24/100 \leq (\sqrt{2}-1)(2-\sqrt{2}) = h(1/2),\]
        \item[(b)] $h$ is zero at $z = 0$ and $z = 1$,
        \item[(c)] $h$ is concave:
        \[\left[ \frac{d^2}{dz^2}h(z) \right] = \log^2(2)2^{1-2z}(3\cdot 2^z - 8),\]
        which is negative for all $z \in [0,1]$.
        \end{itemize}
        Combining all of these facts gives that $h(z) \geq (\sqrt{2}-1)(2-\sqrt{2})\min(z,1-z) \geq g(z)$ for all $z \in [0,1]$, completing the proof.
\end{proof}

\begin{proof}[Proof of Lemma \ref{lem:D_linterval}]
    Observe that $X_{\sigma_l}|(X_{\sigma } = y)$ is equal in distribution to $y M+\Bi{(1-y)M, p_l}.$  
   Denoting $Z_l \sim \Bi{ (1-y)M,p_l}$, we first show that,
    \[\P(|Z_l - (y_{(l)} - 
    y)M| \geq D_l) = O\left(\frac{1}{k^{1+\Co{D_l}}\binom{k}{l}}
    \right).\]
    Using $\E[Z_l] = p_l(1-y)M = (y_{(l)} - y)M$, Lemma \ref{lem:KLBinomial} and \ref{lem:cornerBound} (alongside $y < 1/2$) for a sufficiently small $\Co{D_l} > 0$ and sufficiently large $n$ gives that,
    \begin{align}
        \P(|Z_l - (y_{(l)} - 
    y)M| \geq D_l) &  \leq 2\exp\left(-(1-y)MD(p_l + D_l/((1-y)M)||p_l) \right)\\
    &\leq 2\exp\left({-\frac{D_l^2}{6M(1-y)p_l(1-p_l)}}\right)\\
    & \leq 2\exp\left({-[\log \binom{k}{l} + (1+\Co{D_l})\log k]}\right)\\
    &\leq \frac{2}{\binom{k}{l}k^{1+\Co{D_l}}}\oleq{\frac{1}{\binom{k}{l}k^{1+\Co{D_l}}}},
    \end{align}
    or, equivalently
    \[\binom{k}{l}\P(|X_{\sigma_l} - y_{(l)} M| \geq D_l) = O\left(\frac{1}{k^{1+\Co{D_l}}}\right).\]
    Thus, by a union bound
    \begin{align}\P\left(\bigcup_{\substack{0 \leq l\leq k\\ \sigma_l \subseteq \sigma, |\sigma_l|=l}} |X_{\sigma_l} - y_{(l)} M| \geq D_l\right) &\leq \sum_{l=0}^k\binom{k}{l}\P(|X_{\sigma_l} - y_{(l)} M|\leq D_l)\\
        &= O\left(\frac{1}{k^{\Co{D_l}}}\right)= o(1).\end{align}
\end{proof}

\subsection{Proofs For Subsection \ref{eq:paley-Section}}\label{subsec:FlatnessBounds}

\begin{proof}[Proof of Lemma \ref{lem:2mm bound 1}]

Define $Y_{\sigma, y}$ to be the indicator of the event that a specific $k$-subset $\sigma$ is $\Co{D_l}$-flat and leaves exactly $y M$ target sets uncovered. It holds $Y_y=\sum_{\sigma, |\sigma|=k} \mathbf{1}(Y_{\sigma, y})$
and by some standard expansion.
\begin{align}
    \frac{\E[Y_y^2]}{\E[Y_y]^2} &= \sum_{l=0}^k \frac{\binom{k}{l}\binom{p-k}{k-l}}{\binom{p}{k}}\frac{\P(Y_{\sigma, y} \cap Y_{\tau, y})}{\P(Y_{\sigma, y})^2}\\
    &= \frac{\binom{p-k}{k}}{\binom{p}{k}} + \frac{1}{\binom{p}{k}\P(Y_{\sigma,y})} + \sum_{l=1}^{k-1} \frac{\binom{k}{l}\binom{p-k}{k-l}}{\binom{p}{k}}\frac{\P(Y_{\sigma, y} \cap Y_{\tau, y})}{\P(Y_{\sigma, y})^2}\label{eq:HGbound}
\end{align}
where for the $l$-th term in summation has $|\sigma \cap \tau| = l$. 

To bound the first term in \eqref{eq:HGbound}, we use Assumption \ref{as:alphaC_0} which gives $k^2=o(p)$,
therefore $0 \leq 1-\frac{\binom{p-k}{k}}{\binom{p}{k}} \leq 1-\left(\frac{p-2k}{p}\right)^k=O(k^2/p)=o(1)$. Thus, $\frac{\binom{p-k}{k}}{\binom{p}{k}}=1+o(1)$. To upper bound the third term in \eqref{eq:HGbound}, we use Lemma \ref{lem:D_linterval} to give $\P(Y_{\sigma, y}) = (1-o(1))\P(X_{\sigma, y}) = (1-o(1))\binom{M}{yM} 2^M$, $\binom{M}{yM} \geq \frac{1}{M+1}e^{Mh(y)}$ and $\binom{p}{k} \geq e^{k \log(p/k)}$, which gives that $\frac{1}{\binom{p}{k}\P(Y_{\sigma,y})} \leq 2(M+1)\exp\left[M(\log(2)-h(y)) - k\log(p/k)\right]$ for large enough $n$. 

It remains to bound the summation in \eqref{eq:HGbound} by the summation term in \eqref{eq:2mm1Goal}. This is accomplished by bounding each summand for $l= 1, \cdots, k-1$. 

We decompose the event $Y_{\sigma, y} \cap Y_{\tau, y}$ based on the number of target sets the intersection $\sigma \cap \tau$ leaves uncovered. Notice, by only considering $k$-subsets $\sigma$ that are $\Co{D_l}$-flat, the number of the uncovered target test left by $\sigma \cap \tau \subseteq \sigma$ must fall into the interval $S_l$. Moreover, by Lemma \ref{lem:D_linterval} we have that a.a.s. as \ngrow that any $k$-subset which leaves $yM$ sets uncovered is $\Co{D_l}$-flat, meaning
\[\P(Y_{\sigma, y})=(1-o(1)) \P(X_{\sigma, y}), \label{eq:2mmProof1:flatnessDenom}\]
where $X_{\sigma, y}$ is an indicator random variable that there exists a $k$ sized set $\sigma$ (not necessarily flat) leaving $yM$ target sets uncovered.

By the definition of $S_l$, Definition \ref{def:flat} and \eqref{eq:2mmProof1:flatnessDenom}, for any $\sigma, \tau$ with $|\sigma \cap \tau| = l$, we have
\begin{align}\hspace{-0cm}\frac{\P(Y_{\sigma, y} \cap Y_{\tau, y})}{\P(Y_{\sigma, y})^2} &= \sum_{y' \in S_l}\frac{\P(Y_{\sigma, y} \cap Y_{\tau, y} \mid Y_{\sigma \cap \tau, y'})\P(Y_{\sigma \cap \tau, y '})}{\P(Y_{\sigma, y})^2}\\
&= (1+o(1))\sum_{y' \in S_l}\frac{\P(Y_{\sigma, y} \cap Y_{\tau, y} \mid Y_{\sigma \cap \tau, y'})\P(Y_{\sigma \cap \tau, y '})}{\P(X_{\sigma, y})^2}\label{eq:binomSum}
\end{align}
Defining a random variable $B_l$ distributed as $\Bi{M, (1-q)^l}$, observe that $Y_{\sigma, y }$ corresponds to the event $\{B_k = yM\}$ and $Y_{\sigma \cap \tau, y '}$ corresponds to $\{B_l = y'M\}$. Conditional on the event $Y_{\sigma \cap \tau, y '}$, the events  $Y_{\sigma, y}, Y_{\tau, y}$ are independent and each corresponds to the event defined by $\{ B'= yM\},$ for $B'$ distributed as a $\Bi{y'M, (1-q)^{k-l}}$.
Letting $x \coloneqq l/k$, utilizing Lemma \ref{lem:KLBinomial} and \eqref{eq:binomSum} we conclude that, for sufficiently large $n$,
\begin{align}
    \frac{\P(Y_{\sigma, y} \cap Y_{\tau, y})}{\P(Y_{\sigma, y})^2} &\leq (1+o(1))\sum_{y' \in S_l} \frac{\exp\left(-2y' MD(y / y' || 2^{-(1-x)})\right)\exp\left(-MD(y ' || 2^{-x})\right)}{\frac{1}{9M}\exp\left(-2M D(y || 1/2)\right)}\\
    &\hspace{-3cm}\leq 10M \sum_{y' \in S_l} \exp \left(-2M(y'D(y/y' || 2^{-(1-x)}) - D(y || 1/2) + \frac{1}{2}D(y' || 2^{-x}))\right) \label{eq:prob_part}
\end{align}
This bounds the rightmost summand term for the summation in \eqref{eq:HGbound}. To bound the other summand term we can make use of the following upper bound. For sufficiently large $n$, there exists a constant $\Co{boundComb}$ such that, with $x = l/k$,

\[\frac{\binom{k}{l}\binom{p-k}{k-l}}{\binom{p}{k}} \leq \Co{boundComb}\exp\left(-x k \log(px/k) \right)\label{stirling}\]
    To prove this notice for each $l,$
    \begin{align}
        \frac{\binom{k}{l+1}\binom{p-k}{k-l-1}}{\binom{k}{l}\binom{p-k}{k-l}} =\frac{(k-l)^2}{(l+1)(p-2k+l+1)} \leq \frac{k^2}{l(p-2k)}.
    \end{align}So for each $l$, by a telescopic product, using the inequality $l! \geq \frac{l^l}{e^{l-1}}$  and $\frac{\binom{p-k}{k}}{\binom{p}{k}} = 1 + o(1)$ in line \eqref{eq:usedFactBound},
    \begin{align}\label{stirlingpf}
    \frac{\binom{k}{l}\binom{p-k}{k-l}}{\binom{p}{k}} &\leq \frac{k^{2l}}{l!(p-2k)^l} \frac{\binom{p-k}{k}}{\binom{p}{k}} \\
    &\leq (1+o(1))\frac{e^{l-1}k^{2l}}{l^l(p-2k)^l} \label{eq:usedFactBound}\\
    &= (1+o(1))e^{-l [\log((p-2k)l/k^2)+\frac{l-1}{l}]}\\
    &= (1+o(1))e^{-l\log(pl/k^2) - xk\log(1-2k/p)}.
    \end{align}
Moreover, we know that when $C > 1$ and $\alpha > 0$ that $k = o(p)$ and thus for large enough $p$ we have $\log(1-2k/p) \geq -4k/p$. This gives, recalling $x = l/k$, 
\[\frac{\binom{k}{l}\binom{p-k}{k-l}}{\binom{p}{k}} \leq (1+o(1))e^{-x k \log(px/k) + 4k^2/p}\label{eq:needlargep}.\]
By Assumption \ref{as:alphaC_0} we then have that $k^2 = o(p)$ and thus there exists some constant $\Co{boundComb}$ for which \eqref{stirling} holds (this constant also absorbs the $1 + o(1)$ error term).

    %\begin{align}
    %    \frac{p}{k} &\geq (1-k^{-\err})\frac{n}{k}\left(\frac{k}{n}\right)^{\frac{C}{2}(1+k^{-\err})}\\
    %    &\geq (1-k^{-\err})n^{1-\frac{C}{2}(1 + k^{-\err})}k^{\frac{C}{2}(1 + k^{-\err}) - 1}\\
    %    &\geq (1-k^{-\err})n^{1-\frac{C}{2}(1 + k^{-\err})}(n^\alpha - 1)^{\frac{C}{2}(1 + k^{-\err}) - 1}\\
    %    &\geq (1-k^{-\err})n^{1-\frac{C}{2}(1 + k^{-\err})}{n}^{\alpha(\frac{C}{2}(1 + k^{-\err}) - 1)}(1-n^{-\alpha})^{\frac{C}{2}(1 + k^{-\err}) - 1}\\
    %    &\geq (1-k^{-\err})n^{1-\frac{C}{2}(1 + k^{-\err})}{n}^{\alpha(\frac{C}{2}(1 + k^{-\err}) - 1)}(1-n^{-\alpha})^{\frac{C}{2}(1 + k^{-\err}) - 1}\\
    %    &\geq (1-k^{-\err})n^{(1-\alpha)(1-\frac{C}{2}(1 + k^{-\err}))}(1-n^{-\alpha})^{\frac{C}{2}(1 + k^{-\err}) - 1}
    %\end{align}
Thus, using Lemma \ref{lem:plower} and Lemma \ref{lem:Mbounds}, we have that,

\begin{align}
    k\log(p/k) &= k(1-\alpha)(1-C/2)\log(n) - O(k^{1-\err}\log(n))\\
    &= \frac{2\log(2)(1-C/2)}{C}\left(\frac{Ck(1-\alpha)\log(n)}{2\log(2)}\right) - O(k^{1-\err}\log(n))\\
    &\geq \frac{2\log(2)(1-C/2)}{C}N/2 - O(k),
\end{align}
and thus,
\[xk\log(xp/k) \geq \frac{2\log(2)(1-C/2)}{C}N/2 - xk\log(x) - xO(k)\label{eq:reduction:p}.\]
Using \eqref{eq:reduction:p} in conjunction with \eqref{stirling} gives,
\[\label{eq:HG:upperBound}\begin{split}\frac{\binom{k}{l}\binom{p-k}{k-l}}{\binom{p}{k}} &\leq \Co{boundComb}\exp\bigg{(}-x M \frac{2\log(2)(1-C/2)}{C} - x k \log(x) \\
&\qquad + x \frac{2\log(2)(1-C/2)}{C}(M -N/2) - xO(k)\bigg{)}\end{split}\]
Under Assumption \ref{as:MPScale}, we can see that $M - N/2$ is bounded above by
\[M - N/2 \leq N^{-\err}\frac{N}{2} = \frac{N^{1-\err}}{2}.\]
As $N^{1-\err} = O(k^{1-\err}\log(n)^{1-\err}) = o(k)$ and $k = \Theta(n^\alpha)$, we finally get the bound,
\[\frac{\binom{k}{l}\binom{p-k}{k-l}}{\binom{p}{k}}  \leq \Co{boundComb}\exp\bigg{(}-x M \frac{2\log(2)(1-C/2)}{C} - x k \log(x) + xO(k)\bigg{)}.\]
Recalling Definition \ref{def:H_C}, we have the identity
\[h(H_C) = \log(2)\left(1-\frac{2-C}{C}\right)= \log(2)-\log(2)\frac{2-C}{C}\] or by rearranging terms,
\[2\log(2)\frac{1-C/2}{C} = \log(2)-h(H_C)= D(H_C||1/2).\label{eq:HG:zeroSolRelation}\]
Combining \eqref{eq:HG:upperBound} and \eqref{eq:HG:zeroSolRelation} we now have the upper bound
\[\frac{\binom{k}{l}\binom{p-k}{k-l}}{\binom{p}{k}} \leq  \Co{boundComb}\exp\left(-x MD(H_C||1/2) - x k \log(x) + xO(k) \right) \label{eq:ent_part}\]
Utilizing \eqref{eq:ent_part} in combination with \eqref{eq:prob_part} allows us to upper bound for every $y \in (0,1/2)$, with $n$ sufficiently large, the term $\frac{\binom{k}{l}\binom{p-k}{k-l}}{\binom{p}{k}}\frac{P(Y_{\sigma, y} \cap Y_{\tau, y})}{P(Y_{\sigma, y})^2}$ by
\begin{align}\hspace{-2cm}
& 10\Co{boundComb}M \sum_{y' \in S_l}\exp\left(-x MD(H_C||1/2) - x k \log(x) + xO(k)\right)  \\
&\times \exp\left(-2M(y'D(y/y' || 2^{-(1-x)}) - D(y || 1/2) + \frac{1}{2}D(y' || 2^{-x}))\right)\\
&\qquad\leq \Co{9+eO(1)}M \sum_{y ' \in S_l}\exp\left(-2M\tilde{G}(y, y', x) - x k \log(x) + xO(k)\right),
\end{align}
where we choose \(\Co{9+eO(1)} \geq 10\Co{boundComb}\). Plugging back $x = l/k$ above gives the result.
\end{proof}

\begin{proof}[Proof of Lemma \ref{lem:2mmbound3}]

Invoke Lemma \ref{lem:2mm bound 1} under the choice $y_* = H_C + \Co{Cprime}\kpert$ with $\Co{Cprime}$ to be chosen later. By Lemma \ref{lem:y=1o(1)}, we have that the term $2(M+1)\exp\left({M(\log(2)-h(y_*)) - k\log(p/k)}\right) = o(1)$. Now we bound the summation component of Lemma \ref{lem:2mm bound 1}, by the mean value theorem, for some $z \in (H_C, y_*)$,
\[D(H_C + \Co{Cprime}\kpert||1/2) = D(H_C||1/2) + \left[\partial_{y'}D(y'||1/2)\right]\bigg{|}_{y' = z}\left(\Co{Cprime} \kpert\right).\] As $k$ grows, we have that $z=(1+o(1))H_C$ due to $\Co{Cprime}\kpert = o(1)$. By the continuity and bounded derivative of $\partial_{y}D(y||1/2)$ for $y \in (0,1/2)$, we have that 
\begin{align}
\left[\partial_{y'}D(y'||1/2)\right]\bigg{|}_{y' = z}&=(1+o(1))\left[\partial_{y'}D(y'||1/2)\right]\bigg{|}_{y' = H_C} \\
&=(1+o(1))\log(H_C/(1-H_C))\label{eq:derwithH_C}\\
&=-\Omega(1)
\end{align}
as $H_C < 1/2$ (due to $C \in (1,2)$). Hence by Assumption \ref{as:MPScale}, for sufficiently large $n$ 
\[\left[\partial_{y'}D(y'||1/2)\right]\bigg{|}_{y' = z}\Co{Cprime}M \kpert x = -x\Co{Cprime}\Omega(k).\]
Combining the above alongside an application of the mean value theorem,
\begin{align} 
-x M D(H_C || 1/2)&=-x M D(H_C + \Co{Cprime}\kpert|| 1/2) +  \left[\partial_{y}D(y||1/2)\right]\bigg{|}_{y = z} \Co{Cprime} M \kpert x  \\
&= - x M D(H_C + \Co{Cprime}\kpert || 1/2) -x\Co{Cprime}\Omega(k) \label{eq:\Co{Cprime}perturb}.
\end{align}

Thus, by interchanging the differing terms between $\tilde{G}$ and $G$, we get
\[\frac{\E[Y^2_{y_*}]}{\E[Y_{y_*}]^2} -1 \leq o(1) + \Co{9+eO(1)}M \sum_{l = 1}^{k-1}\sum_{y ' \in S^*_l}\exp\left(-2MG(y_*, y', l/k) - l \log(l/k) + xO(k) - x\Co{Cprime}\Omega(k) \right).\]
Choosing $\Co{Cprime}$ sufficiently large so that the implicit constant in $\Co{Cprime}\Omega(k)$ dominates the implicit constant in $O(k)$ and substituting $x$ for $l/k$ gives the proof.
\end{proof}

\subsection{Proofs For Subsection \ref{subsub:G}}\label{subsec:proofG}

\begin{proof}[Proof of Lemma \ref{DerPrime}]
    
    Fix $x,y$, with a slight abuse of notation, we abbreviate $G(y, y', x) = G_{y,x}(y')$ and $G'_{y,x}(y') = \partial_{y'}G(y, y', x)$. By rearranging terms in the definition of $G$ we have,
    \[
        G_{y,x}(y') = \underbrace{y \log\left(\frac{y}{y'} 2^{1-x}\right)}_{A} + \underbrace{(y' - y)\log\left(\frac{1-y/y'}{1-2^{-(1-x)}}\right)}_{B} - \underbrace{(1-x/2)(\log(2) - h(y))}_{C}+ \underbrace{\frac{1}{2}D(y'||2^{-x})}_{D}
    \]
    We can then calculate for each term that,
    \begin{align}
    \partial_{y'}(A) &= - \frac{y}{y'}&
    \partial_{y '}(B) &= \log\left( \frac{1-\frac{y}{y'}}{1-2^{-(1-x)}}\right) + \frac{y}{y'}\\
    \partial_{y'}(C) &= 0&
    \partial_{y'}(D) &= \frac{1}{2}\log\left(\frac{y'}{1-y'}\frac{1-2^{-x}}{2^{-x}}\right).
    \end{align}
    Therefore, it holds,
    \[G'_{y,x}(y') = \log\left( \frac{1-\frac{y}{y'}}{1-2^{-(1-x)}}\right) + \frac{1}{2}\log\left(\frac{y'}{1-y'}\frac{1-2^{-x}}{2^{-x}}\right).\]
    Now plugging in $y' = y_{(x)}$ we get
    \[G'_{y,x}(y')\bigg{|}_{y' = y_{(x)}} = \log\left( \frac{1-\frac{y}{y_{(x)}}}{1-2^{-(1-x)}}\right) + \frac{1}{2}\log\left(\frac{y_{(x)}}{1-y_{(x)}}\frac{1-2^{-x}}{2^{-x}}\right),\label{eq:gamma':der}\]
    and using that $y_{(x)} = y + (1-y)(2^{1-x} - 1)$,
    \[[G'_{y,x}(y')]\bigg{|}_{y' = y_{(x)}} = \log\left( \frac{1-\frac{y}{\gl}}{1-2^{-(1-x)}}\right) + \frac{1}{2}\log\left(\frac{\gl}{1-(\gl)}\frac{1-2^{-x}}{2^{-x}}\right).\]
    Taking the derivative of the above function with respect to $x$, we have the function
    \[\frac{\log(2) - y \log(4)}{4y + (1-y)2^{2-x} - 2}\label{eq:derGamma'derEpsilon}\]
    One can directly see the numerator is positive when $y < 1/2$. Moreover, one can justify that the denominator is positive. Indeed, plugging in $x = 1$ into the denominator of \eqref{eq:derGamma'derEpsilon} we get $4y + 2(1-y) - 2 > 0$,
    when $y > 0$. Taking the derivative of the denominator in \eqref{eq:derGamma'derEpsilon} with respect to $x$ again, gives the function $-(1-y)2^{2-x}$ which is always negative for $y < 1/2$. Hence, the denominator of \eqref{eq:derGamma'derEpsilon} is positive. Thus, the whole term \eqref{eq:derGamma'derEpsilon} is positive for all $x \in (0,1)$. Therefore, we have shown that $[G'_{y,x}(y')]|_{y' = y_{(x)}}$ is increasing with respect to $x$. Thus, a sufficient condition to show the first statement in the lemma is to prove
    \[\lim_{x\conv{} 1}[G'_{y,x}(y')]\bigg{|}_{y' = y_{(x)}} = \frac{1}{2}\log\left(\frac{1-y}{y}\right)\]
    and 
    \[\lim_{x\conv{} 0}[G'_{y,x}(y')]\bigg{|}_{y' = y_{(x)}} = \frac{1}{2}\log(2(1-y)).\]
    By direct reasoning, $x \conv{} 1$ gives $y_{(x)} \conv{} y$. Hence, the second logarithm on the right-hand side of \eqref{eq:gamma':der} is equal to $\frac{1}{2}\log\left( \frac{y}{1-y}\right)$. Using L'Hospital's rule and continuity of the logarithm gives,
    \begin{align}
        &\lim_{x \conv{} 1}\log\left(\frac{1- \frac{y}{y + (1-y)(2^{1-x}-1)}}{1-2^{-(1-x)}} \right) =\log\left(\lim_{x \conv{} 1}\frac{1- \frac{y}{y + (1-y)(2^{1-x}-1)}}{1-2^{-(1-x)}} \right) =\log\left(\lim_{x \conv{} 1} \frac{\frac{2^{1+x}(y - 1)y \log(2)}{(2^x - 2 - 2(2^x -1)y)^2}}{-2^{x - 1}\log(2)} \right)\\
        &\qquad=\log\left(\frac{1-y}{y} \right).\label{eq:limFirstLog}
    \end{align}
    Using \eqref{eq:limFirstLog} alongside our arguments above we have
    \[
        \lim_{x\conv{} 1}[G'_{y,x}(y')]\bigg{|}_{y' = y_{(x)}} = \log\left(\frac{1-y}{y}\right) + \frac{1}{2}\log\left( \frac{y}{1-y}\right) = \frac{1}{2}\log\left(\frac{1-y}{y}\right).
    \]
    In order to prove the second claim of the first statement, we also need to calculate the limiting derivative as $x \conv{} 0$. Similar to above, $x \conv{} 0$ gives $y_{(x)} \conv{} 1$ from below. This means that the first logarithm on the right-hand side in \eqref{eq:gamma':der} converges to $\log(2(1-y))$. Using L'Hospital's rule and the continuity of the logarithm, we have,
    \begin{align}
        &\lim_{x \conv{} 0}\frac{1}{2}\log\left(\frac{\gl}{1-(\gl)}\frac{1-2^{-x}}{2^{-x}}\right)\\ 
        &\qquad= \frac{1}{2}\log\left(\lim_{x \conv{} 0}\frac{\gl}{1-(\gl)}\frac{1-2^{-x}}{2^{-x}}\right)\\
        &\qquad= \frac{1}{2}\log\left(\lim_{x \conv{} 0}\frac{4^{-x}(2-2y + 2^x(4y - 3))\log(2)}{2^{1-2x}(2^x - 2)(y - 1)\log(2)}\right)\\
        &\qquad=\frac{1}{2}\log\left(\frac{1}{2(1-y)}\right).\label{eq:2ndLogLimit}
    \end{align}
    Thus,
    \[
    \lim_{x\conv{} 0}[G'_{y,x}(y')]\bigg{|}_{y' = y_{(x)}}= \log(2(1-y)) + \frac{1}{2}\log\left(\frac{1}{2(1-y)} \right) = \frac{1}{2}\log(2(1-y))
        \label{eq:derat0}\]
    confirming the second claim of the first statement.
    Combining \eqref{eq:derat0} with \eqref{eq:derGamma'derEpsilon}, we can write by the Fundamental Theorem of Calculus,
    \[[G'_{y,x}(y')]\bigg{|}_{y' = y_{(x)}} = \frac{1}{2}\log(2(1-y)) + \int_0^x \frac{\log(2) - y \log(4)}{4y + (1-y)2^{2-u} -2} \;du\]
    Notice that, with respect to $y$, the integrand is decreasing in the numerator and increasing in the denominator. Thus, we get the following upper bound by plugging in $y = 0$ inside the integrand,
    \[[G'_{y,x}(y')]\bigg{|}_{y' = y_{(x)}} \leq \frac{1}{2}\log(2(1-y)) + \int_0^x \frac{\log(2)}{2^{2-u} -2} \;du \leq \frac{1}{2}\log(2(1-y)) + x \frac{\log(2)}{2^{2-x} -2},\]
    thus proving the second statement.
\end{proof}

\begin{proof}[Proof of Lemma \ref{lem:justifyLowerGamma'}]
    We denote $G'_{y,x}(\Breve{y}) = [\partial_{y'}G(y, y', x)]|_{y' = \Breve{y}}$. By the mean value theorem we have that for some $y'_* \in \left[y_{(x)} - \frac{D_l}{M}, y_{(x)} + \frac{D_l}{M}\right]$ we have
    \[
        G(y, y', x) = G(y, y_{(x)}, x) + (y' - y_{(x)})G'_{y,x}(y'_*) \geq G(y, y_{(x)}, x) - \frac{D_l}{M}G'_{y,x}(y'_*)\label{eq:continueder}
    \]
    From Lemma \ref{lem:cornerBound} we see that $D_l/M = o(1)$ uniformly over $l$. Thus,
    $\Breve{y} = (1+o(1))y_{(x)}$ for any $\Breve{y} \in \left[y_{(x)} - \frac{D_l}{M}, y_{(x)} + \frac{D_l}{M}\right]$, since, by definition, $y_{(x)} \in [y,1]$ and thus $y_{(x)} = \Theta(1)$ if $y > 0$. By the continuity of the derivative of $G$ in $y'$ and that (using Lemma \ref{DerPrime} with $x \in (0,1)$ and $1/2 > y > 0$) $G'_{y,x}(y_{(x)}) = \Theta(1)$ uniformly over $x \in (0,1)$, we also have $G'_{y,x}(\Breve{y}) = (1+o(1))G'_{y,x}(y_{(x)})$. As such,  \eqref{eq:continueder} implies that for any $\epsilon > 0$,
    \[G(y, y', x) \geq G(y, y_{(x)}, x) - (1 + \epsilon)\frac{D_l}{M}G'_{y,x}(y_{(x)})\]
    for sufficiently large enough $n$, concluding the proof.
\end{proof}

\begin{proof}[Proof of Lemma \ref{zerobound}]
When $x \conv{} 0$, we have that $\gl \conv{} 1$. By the continuity of KL divergence in both of its arguments, we have
    \[\lim_{x \conv{} 0}\Breve{G}(y, x) = 0 + D(y||1/2) - D(y||1/2) + 0 = 0\]
    Similarly, we have that $\gl \conv{} y$ when $x \conv{} 1$, giving
    \begin{align}
        \lim_{x \conv{} 1}\frac{x}{2}D(y||1/2) &= \frac{1}{2}D(y||1/2)\label{eq:lim1}\\
        \lim_{x \conv{} 1}\frac{1}{2}D(\gl||2^{-x}) &= \frac{1}{2}D(y||1/2).\label{eq:lim2}
    \end{align}
    Thus, by the definition of $\Breve{G}(y,x)$ we are left to characterize the limit
    \[\lim_{x \conv{} 1}(\gl) D\left(\frac{y}{\gl}\|\|2^{-(1-x)}\right).\]
    First we can immediately see that $\lim_{x \conv{} 1}\gl = y$
    Considering the following Taylor expansions, $2^{1-x} = 1 + (1-x)\log(2) + O((1-x)^2)$ and $2^{x - 1} = 1 - (1-x)\log(2) + O((1-x)^2)$, we can see that
    \begin{align}&\lim_{x \conv{} 1}D\left(\frac{y}{\gl}\|\|2^{-(1-x)}\right) \\
    &= \lim_{x \conv{} 1} D\left(1 - (1-x)\frac{(1-y)\log(2) + O(1-x)}{y}\|\|1 - (1-x)(\log(2) + O(1-x))) \right) = 0,\end{align}
    as $\frac{y}{1-y} = \Theta(1)$ for $y \in (0,1/2)$. Thus, 
    \[\lim_{x \conv{} 1}(\gl) D\left(\frac{y}{\gl}\|\|2^{-(1-x)}\right) = 0\label{eq:lim3}\]
    Using \eqref{eq:lim1}, \eqref{eq:lim2} and \eqref{eq:lim3} in \eqref{eq:g-eps}, we have that $y \in (0,1/2)$ implies $\lim_{x \conv{} 1}\Breve{G}(y, x) = 0$.
    \end{proof}

\begin{proof}[Proof of Lemma \ref{posder}]
We denote $\Breve{G}_y(x) = \Breve{G}(y, x)$ and $\Breve{G}'_y(\Breve{x}) = [\partial_x \breve G(y,x)]|_{x = \Breve{x}}$. We then calculate,
    \[\begin{split}
        \Breve{G}'_y(x) &= \frac{1}{2}\bigg{[}(1-y)\log(4-4y) + y\log(y) + 2^{-x}(1-y)\log(4)\\
        &\qquad \times \bigg{(}\log(2-2y)-\log(2-2^x + 2y(2^x - 1)) -2\log\left( \frac{2-2y}{2-2y + 2^x(2y - 1)}\right)\bigg{)}\bigg{]},
    \end{split}\]
    and, by elementary inspection, we can see that the formula for $\Breve{G}_y'(x), x\in (0,1)$ above is in fact continuous as a function of $x \in [0,1].$ Hence, we consider the continuous extension of $\Breve{G}_y'(x)$ over the domain of $x \in [0,1]$, to ease the notation for this proof.
    
    Plugging in $x = 0$ gives,
   \begin{align}
   \Breve{G}'_y(0) &= \frac{1}{2}\left[(1-y)\log(4-4y) + y\log(y) - (1-y)\log(4)\log(2-2y)\right]\\
   &= \frac{1}{2}\left[(1-y)\log(4) - (1-y)\log(4)\log(2-2y) + (1-y)\log(1-y) + y\log(y)\right]\\
   &= (1-y)\log(2)(1-\log(2-2y)) - \frac{h(y)}{2}\\
   &= \log(2)((1-y)(1-\log(2-2y)) - h_2(y)/2)
   \end{align}
   One observes that for $y = 1/2$, we have $\Breve{G}'_y(0) = 0$,
   and for $y = 0$, $\Breve{G}'_y(0) = \log(2)(1-\log(2)) > 0$.
    Alongside the above equation, a sufficient condition for the positivity of $\Breve{G}'_y(0)$ when $y < 1/2$ is, for all $\epsilon \in [0,1/2)$, to have $[\partial_{y}\Breve{G}'_y(0)]|_{y = \epsilon} < 0$. Calculating this value gives
   \[[\partial_{y}\Breve{G}'_y(0)]|_{y = \epsilon} = \log(2)\log(2(1-\epsilon)) - \frac{1}{2}\log\left( \frac{1-\epsilon}{\epsilon}\right)\]
   Eliciting, $[\partial_{y}\Breve{G}'_y(0)]|_{y = 1/2} = 0$
   and $\lim_{\epsilon \conv{} 0}[\partial_{y}\Breve{G}'_y(0)]|_{y = \epsilon} = -\infty$. Thus, a further sufficient condition for the positivity of $G'_y(x)$ is, for all $0 \leq \epsilon \leq 1/2$, that $[\partial^2_{y}\Breve{G}'_y(0)]|_{y = \epsilon} > 0$.
   We can see that this second derivative takes the form of
   \[
       [\partial^2_{y}\Breve{G}'_y(0)]|_{y = \epsilon} = \frac{1}{1-\epsilon}\left( \frac{1}{2\epsilon} - \log(2)\right) > 0,
   \]
   when $0\leq \epsilon \leq 1/2 $. Thus, $\Breve{G}'_y(0)$, treated as a function of $y$, is strictly bounded away from zero for $y \in (0,1/2)$ with a maximum value at $y = 0$ with $\breve G'_y(0) = \log(2)(1-\log(2))$. Meaning that derivative $ \breve G'_y$ at $x=0$ is $\Theta(1)$ for all $y \in (0,1/2)$.
   
   Plugging in $x = 1$, gives
    \begin{align}
        \begin{split}\Breve{G}'_y(1) &= \frac{1}{2}\bigg{[}(1-y)\log(4-4y)+y\log(y)+\frac{1}{2}(1-y)\log(4)\\
        &\qquad \times\left(\log(2-2y)-\log(2y)- 2\log\left(\frac{2-2y}{2-2y +2(2y-1)}\right) \right)\bigg{]}\end{split}\\
        \begin{split}&= \frac{1}{2}\bigg{[(}1-y)\log(4) + \frac{1}{2}(1-y)\log(4)\bigg{(}\log(1-y) - \log(y) \\
        &\qquad -2\left(\log(1-y) - \log(y)\right)\bigg{)}\bigg{]} - \frac{h(y)}{2}\end{split}\\
        &= \log(2)\left((1-y)\left[1+\frac{1}{2}\log\left(\frac{y}{1-y} \right)\right] - \frac{h_2(y)}{2}\right)
    \end{align}
    We can then see that at $y = 1/2$ we have $\Breve{G}'_y(1) = 0$ and $\lim_{y\conv{} 0}\Breve{G}_y'(1) = -\infty.$
    Similar to the above statement, a sufficient condition for negativity of $\Breve{G}'_y(1)$ for all $y < 1/2$ is to show for all $\epsilon \in [0,1/2)$ that $[\partial_{y}\Breve{G}'_y(1)]|_{y = \epsilon} > 0$.
    The derivative of $\Breve{G}'_y(1)$ with respect to $y$ is
    \[[ \partial_{y}\Breve{G}'_y(1)]|_{y = \epsilon} = \frac{-\log(2)(\epsilon\log(\frac{\epsilon}{1-\epsilon}) + 2\epsilon - 1)}{2\epsilon} - \frac{1}{2}\log\left(\frac{1-\epsilon}{\epsilon}\right).\]
    This can be rewritten as
    \[\begin{split}[\partial_{y}\Breve{G}_y'(1)]|_{y = \epsilon} &= \log\left(\frac{\epsilon}{1-\epsilon}\right)\left(\frac{1}{2} - \frac{\log(2)}{2}\right) + \frac{\log(2)}{2\epsilon} - \log(2)\\
        &= \frac{2\epsilon\log\left(\frac{\epsilon}{1-\epsilon}\right)\left(\frac{1}{2} - \frac{\log(2)}{2}\right) + \log(2) - 2\epsilon\log(2)}{2\epsilon}.
    \end{split}\]
    We can then see that $[ \partial_{y}\Breve{G}_y'(1)]|_{y = 1/2} = 0$ and $\lim_{\epsilon \conv{} 0}[ \partial_{y}\Breve{G}_y'(1)]|_{y =\epsilon} = \infty$. Meaning that a further sufficient condition for the negativity of $\Breve{G}'_y(1)$ is for all $\epsilon \in [0,1/2]$ that $[ \partial^2_{y}\Breve{G}'_y(1)]|_{y = \epsilon} < 0$.
    Taking the second derivative with respect to $y$ gives $[ \partial^2_{y}\Breve{G}'_y(1)]|_{y = \epsilon} = \frac{\epsilon - \log(2)}{2(1-\epsilon)\epsilon^2}$,
    which we can clearly see is negative for $0 \leq \epsilon \leq 1/2$. Thus, we have shown that $G'_y(1) < 0$ for any $y$ bounded away from both $0$ and $1/2$. This concludes the proof.
\end{proof}

\begin{proof}[Proof of Lemma \ref{positive}]
    Using Lemma \ref{posder} and \ref{zerobound}, the lemma follows from showing that $\Breve{G}(y,x)$ is strictly concave on $x \in (0,1)$ for each fixed $y \in (0,1/2)$. Thus, we compute the second derivative of $\Breve{G}_y(x) = \Breve{G}(y, x)$ as
    \begin{align}
    \frac{\partial^2}{\partial x^2}\Breve{G}_y(x) &= -(1-y) \log ^2(2) \bigg{(}\frac{1}{\frac{1}{1-2 y}-2^x+1}+2^{-x} \bigg{(}-2 \log \left(\frac{2-2 y}{(2 y-1) 2^x-2
   y+2}\right)\\
   &\qquad - \log \left((2 y-1) 2^x-2 y+2\right)+\log (2-2 y)\bigg{)}\bigg{)}\\
   & = -(1-y) \log ^2(2) \left(\frac{1-2y}{2-2y - 2^x(1-2y)} + 2^{-x}\log\left(1+\frac{(2y-1)2^x}{2-2y}\right)\right)\\
   &< -(1-y)\log^2(2)\left(\frac{1-2y}{2-2y - 2^x(1-2y)} + 2^{-x}\left(\frac{\frac{2^x(2y-1)}{2-2y}}{1+\frac{2^x(2y-1)}{2-2y}}\right)\right)\label{eq:usedlogbound}\\
   &= -(1-y)\log^2(2)\left(\frac{1-2y}{2-2y - 2^x(1-2y)} -\frac{1-2y}{2-2y - 2^\epsilon(1-2y)}\right)\\
   &= 0,
   \end{align}
   where we used $\frac{\epsilon}{1+\epsilon} < \log(1+\epsilon)$ for $\epsilon > 0$ in \eqref{eq:usedlogbound}, which can be applied since $\frac{(2y-1)2^x}{2-2y} > 0$ for $y \in (0,1/2)$. Thus, we have established that $\Breve{G}_y(x)$ is strictly concave for all $x \in (0,1)$. From Lemma \ref{posder} we have that the derivative of $\Breve{G}_y(x)$ is bounded away from zero positively and negatively at $x = 0$ and $x = 1$ respectively. Combining this result with the strict concavity demonstrated above, we have that $\Breve{G}_y(x) > 0$ for $x \in (0,1)$.
\end{proof}

\subsection{Useful Asymptotic Bounds for $M$ and $p$}

\begin{comment}
\[\begin{split}
    &\frac{C(1-\alpha)k\log(n)}{2\log(2)} - O(1) \leq \frac{Ck\log_2((n/\floor{n^\alpha})) + O(kn^{-\alpha}) - 1}{2} \leq \frac{Ck\log_2((n/\floor{n^\alpha})(1-O(n^{\alpha}))) - 1}{2}\\
    &\leq \frac{Ck\log_2(n/(n^\alpha + 1)) - 1}{2} \leq \frac{Ck\log_2(n/\floor{n^\alpha}) - 1}{2} \leq \frac{Ck\log_2(n/k) - 1}{2} \leq \frac{C\log_2\binom{n}{k} - 1}{2}\\
    &\leq \frac{N}{2} \leq \frac{C\log_2\binom{n}{k} + 1}{2} \leq \frac{C\log\binom{n}{k} + 1}{2\log(2)} \leq \frac{Ck\log(n/k) + Ck + 1}{2\log(2)} \leq \frac{Ck\log(n/(n^\alpha - 1)) + Ck + 1}{2\log(2)}\\
    &\leq \frac{Ck\log((n/n^\alpha)(1 + O(n^{-\alpha}))) + Ck + 1}{2\log(2)} \leq \frac{Ck(1-\alpha)\log(n/n^\alpha)) + Ck\log((1 + O(n^{-\alpha})) + Ck + 1}{2\log(2)}\\
    &\leq \frac{Ck(1-\alpha)\log(n/n^\alpha)) + O(kn^{-\alpha}) + Ck + 1}{2\log(2)} \leq \frac{Ck(1-\alpha)\log(n)}{2\log(2)} + O(k)
\end{split}
\]
\end{comment}

\begin{lemma}\label{lem:Mbounds}
Recall that $N = \floor{C\log_2\binom{n}{k}}$ and $k = \floor{n^\alpha}$ with $C \in (1,2)$ and $\alpha \in (0,1)$. We have that as $n$ grows,
    $\frac{C(1-\alpha)k\log(n)}{2\log(2)} - O(1) \leq \frac{N}{2} \leq \frac{Ck(1-\alpha)\log(n)}{2\log(2)} + O(k)$
\end{lemma}
\begin{proof}
To show the upper bound we have,
\begin{align}
    N/2 &\leq \frac{C\log_2\binom{n}{k} + 1}{2} \leq \frac{Ck\log\left(ne/k\right) + 1}{2\log(2)} = \frac{Ck\log(n/k)}{2\log(2)} + O(k)\\
    &\leq  \frac{Ck\log\left(\frac{n}{n^\alpha}\frac{1}{1-n^{-\alpha}}\right)}{2\log(2)} + O(k) \leq \frac{Ck\log\left(\frac{n}{n^\alpha}\right)}{2\log(2)} + O(k\log(1-n^{-\alpha})) + O(k) \\
    &= \frac{Ck(1-\alpha)\log(n)}{2\log(2)} + O(k).
\end{align}
To show the lower bound we have,
\begin{align}
    \frac{N}{2} &\geq \frac{C\log_2\binom{n}{k} - 1}{2} \geq \frac{Ck\log(n/k)}{2\log(2)} - O(1) \geq \frac{Ck\log\left(\frac{n}{n^\alpha}\frac{1}{1+n^{-\alpha}} \right)}{2\log(2)} - O(1)\\
    &\geq \frac{Ck(1-\alpha)\log(n)}{2\log(2)} - O(k \log(1 + n^{-\alpha})) - O(1) \geq \frac{Ck(1-\alpha)\log(n)}{2\log(2)} - O(1),
\end{align}
where the last inequality is because $\log(1+n^{-\alpha}) = \Theta(n^{-\alpha}) = \Theta(k^{-1})$.
\end{proof}

\begin{lemma}\label{lem:plower}
    Recall that $p$ is the number of possibly infected individuals after COMP post-processing, $k = \floor{n^\alpha}$, $C \in (1,2)$ and $\alpha \in (0,1)$. We have that for any $1/4 > c > 0$ (or if $p = \pSet$ satisfies Assumption \ref{as:MPScale}) that,
    \[(1-\alpha)\left(1-\frac{C}{2}\right)\log(n) + O(k^{-c}\log(n)) \geq \log(p/k) \geq (1-\alpha)\left(1-\frac{C}{2}\right)\log(n) - O(k^{-c}\log(n)),\]
    a.a.s. as \ngrow\!.
\end{lemma}

\begin{proof}
We will demonstrate the lower bound of this statement, the upper bound follows similarly.

By Lemma \ref{lem:p:limit} (or by Assumption \ref{as:MPScale} with $p = \pSet$), we have the following a.a.s. as \ngrow\!,
\begin{align}
\frac{p}{k} &\geq (1-k^{-\err})\frac{n}{k}\left(\frac{k}{n}\right)^{\frac{C}{2}(1+k^{-\err})}\\
&\geq (1-k^{-\err})n^{1-\frac{C}{2}(1 + k^{-\err})}k^{\frac{C}{2}(1 + k^{-\err}) - 1}\\
&\geq (1-k^{-\err})n^{1-\frac{C}{2}(1 + k^{-\err})}(n^\alpha - 1)^{\frac{C}{2}(1 + k^{-\err}) - 1}\\
&\geq (1-k^{-\err})n^{1-\frac{C}{2}(1 + k^{-\err})}{n}^{\alpha(\frac{C}{2}(1 + k^{-\err}) - 1)}(1-n^{-\alpha})^{\frac{C}{2}(1 + k^{-\err}) - 1}\\
&\geq (1-k^{-\err})n^{1-\frac{C}{2}(1 + k^{-\err})}{n}^{\alpha(\frac{C}{2}(1 + k^{-\err}) - 1)}(1-n^{-\alpha})^{\frac{C}{2}(1 + k^{-\err}) - 1}\\
&\geq (1-k^{-\err})n^{(1-\alpha)(1-\frac{C}{2}(1 + k^{-\err}))}(1-n^{-\alpha})^{\frac{C}{2}(1 + k^{-\err}) - 1}.\\
\end{align}
Thus, taking the logarithm, we have that,
\begin{align}
    \log(p/k) &\geq \log\left((1-k^{-\err})n^{(1-\alpha)(1-\frac{C}{2}(1 + k^{-\err}))}(1-n^{-\alpha})^{\frac{C}{2}(1 + k^{-\err}) - 1}\right)\\
    &= \log(1-k^{-\err}) + (1-\alpha)\left(1-\frac{C}{2}(1 + k^{-\err})\right)\log(n) + \left(\frac{C}{2}(1 + k^{-\err}) - 1\right)\log(1-n^{-\alpha})\\
    &= (1-\alpha)\left(1-\frac{C}{2}\right)\log(n) - O(k^{-\err}\log(n)) - O(n^{-\alpha}) - O(k^{-\err})\\
    &= (1-\alpha)\left(1-\frac{C}{2}\right)\log(n) - O(k^{-\err}\log(n))\\
\end{align}
\end{proof}

\end{document}